\documentclass[12pt]{article}
\usepackage{a4wide}
\usepackage{amssymb,amsfonts,amsmath}
\usepackage[hyperref]{ntheorem} 
\makeatletter

\newtheoremstyle{longplain}
{\item[\hskip\labelsep \theorem@headerfont ##1\ ##2\theorem@separator]}
{\item[\hskip\labelsep  \theorem@headerfont ##1\ ##2]
\textbf{(##3)}  \theorem@separator\ \ }

\makeatother

\theoremstyle{longplain}

\newtheorem{thm}{Theorem}[section]
\newtheorem{cor}[thm]{Corollary}
\newtheorem{lem}[thm]{Lemma}
\newtheorem{prop}[thm]{Proposition}
\newtheorem{defn}[thm]{Definition}
\newtheorem{rem}[thm]{Remark}
\newtheorem{con}[thm]{Conclusion}

\theorembodyfont{\normalfont\rmfamily}
\theoremstyle{nonumberplain}
\newtheorem{proof}{Proof}
\newcommand{\BeweisEnde}{\hfill\hskip 1mm
  \hbox{\vrule height 1.9mm width 2mm depth
    0.1mm}}

\usepackage{latexsym}

\newcommand{\setC}{{\mathord{\mathbb C}}}
\newcommand{\setN}{{\mathord{\mathbb N}}}

\newcommand{\setR}{{\mathord{\mathbb R}}}

\newcommand{\mymarginpar}[1]{}

\def\supp{\mathop{\mathrm{supp}}}
\newcommand{\Tr}{{\mathord{\rm Tr}}}
\newcommand{\n}{{\mathord{|||}}}

\usepackage[colorlinks]{hyperref}
\usepackage{color}

\numberwithin{equation}{section} 

\usepackage{comment}

\definecolor{orange}{rgb}{1.00,0.65,0.00}
\definecolor{purple}{rgb}{0.63,0.13,0.94}

\usepackage{verbatim}
\usepackage{framed}

\begin{document}
\title{Well-Posedness of the  Einstein--Euler System in Asymptotically Flat
  Spacetimes} 
\author{Uwe Brauer and Lavi Karp}
\date{}

\maketitle{}
\begin{small}
\begin{center}
\begin{tabular}{ll}
 \textbf{Departamento Matem\'atica Aplicada}  &  \textbf{Department of
   Mathematics}\\ \textbf{Universidad Complutense Madrid}  &
 \textbf{ORT Braude College} \\ 
\textbf{28040 Madrid, Spain} & \textbf{P. O. Box 78, 21982 Karmiel,
  Israel} \\ 
{\it E-mail:} {oub@mat.ucm.es} & {\it E-mail:} {karp@braude.ac.il} 
            \end{tabular}
            \end{center}
\end{small}



\begin{abstract} 
 We prove a local in time existence and uniqueness theorem of
classical solutions of the coupled Einstein--Euler system, and
therefore establish the well posedness of this system. We use the
condition that the energy density might vanish or tends to zero at
infinity and that the pressure is a certain function of the energy
density, conditions which are used to describe simplified stellar
models. In order to achieve our goals we are enforced, by the
complexity of the problem, to deal with these equations in a new type
of weighted Sobolev spaces of fractional order. Beside their
construction, we develop tools for PDEs and techniques for elliptic
and hyperbolic equations in these spaces. The well posedness is
obtained in these spaces.  The results obtained are related to and
generalize earlier works of Rendall \cite{Rendall92:-fluid} for the
Euler-Einstein system under the restriction of time symmetry and of
Gamblin \cite{gamblin93:_solut_euler
 } for the simpler Euler--Poisson system.
\end{abstract}

\maketitle

\section{Introduction}
\label{sec:introduccion}


This paper deals with the Cauchy problem for the Einstein-Euler system
describing a relativistic self-gravitating perfect fluid, whose
density either has compact support or falls off at infinity in an
appropriate manner, that is, the density belongs to a certain weighted
Sobolev space.

The evolution of the gravitational field is described by
the Einstein equations
\begin{equation}
  \label{eq:eineul:1} 
  R_{\alpha\beta}-\frac{1}{2}g_{\alpha\beta}R = 8\pi T_{\alpha\beta}
\end{equation}
where $ g_{\alpha\beta}$ is a semi Riemannian metric having a
signature $(-,+,+,+)$, $ R_{\alpha\beta}$ is the Ricci curvature
tensor, these are functions of $ g_{\alpha\beta}$ and its first and
second order partial derivatives and $R$ is the scalar curvature.  The
right hand side of (\ref{eq:eineul:1}) consists of the energy-momentum
tensor of the matter, $T_{\alpha\beta}$ and in the case of a perfect
fluid the latter takes the form
\begin{equation}
  \label{eq:eineul:2}
  T^{\alpha\beta} = (\epsilon+p) u^\alpha u^\beta +
  pg^{\alpha\beta}, 
\end{equation}
where $\epsilon$ is the energy density, $p$ is the pressure and
$u^\alpha$ is the four-velocity vector.  The vector $u^{\alpha}$ is a
unit timelike vector, which means that it is required to satisfy the
normalization condition
\begin{equation}
  \label{eq:publ-broken:2}
  g_{\alpha\beta}  u^\alpha u^\beta=-1.
\end{equation}
The Euler equations describing the evolution of the fluid take the
form
\begin{equation}
  \label{eq:eineul:3}
  \nabla_\alpha T^{\alpha\beta}=0,
\end{equation}
where $\nabla$ denotes the covariant derivative associated to the
metric $g_{\alpha\beta}$. 
Equations (\ref{eq:eineul:1}) and (\ref{eq:eineul:3}) are not
sufficient to determinate the structure uniquely, a functional
relation between the pressure $p$ and the energy density $\epsilon$
(equation of state) is also necessary. We choose an equation of state
that has been used in astrophysical problems. It is the analogue of
the well known polytropic equation of state in the non-relativistic
theory, given by
\begin{equation}
  \label{eq:eineul:4}
  p = f(\epsilon) = K\epsilon^{\gamma}, \qquad
K,\gamma\in\setR^{+},\qquad
  1<\gamma.
\end{equation}
The sound velocity is denoted by
\begin{equation}
  \label{eq:publ-broken:14} \sigma^2=\frac{\partial
    p}{\partial\epsilon}.
\end{equation}

The unknowns of these equations are the semi Riemannian metric
$g_{\alpha\beta}$,
the velocity vector $u^\alpha$ and the energy density $\epsilon$.
These are functions of $t$ and $x^a$ where $x^a$ $(a=1,2,3)$ are the
Cartesian
coordinates on $\setR^3$. The alternative notation $x^0=t$ will also
be used and Greek indices will take the values $0,1,2,3$ in the
following.

The common method to solve the Cauchy problem for the Einstein
equations consists usually of two steps.  Unlike ordinary initial
value problems, initial data must satisfy constraint equations
intrinsic to the initial hypersurface. Therefore, the first step is to
construct solutions of these constraints. The second step is to solve
the evolution equations with these initial data, in the present case
these are first order symmetric hyperbolic systems. As we describe
later in detail, the complexity of our problems forces us to consider
an additional third step, that is, after solving  the constraint
equations, we have to construct the initial data for the fluid
equations.

The nature of this Einstein-Euler system (\ref{eq:eineul:1}),
(\ref{eq:eineul:3}) and (\ref{eq:eineul:4}) forces us to treat both
the
constraint and the evolution equations in the same type of functional
spaces.  Under the above consideration, we have established the well
posedness of this Einstein-Euler system in  weighted Sobolev spaces
of fractional order.  Oliynyk has recently studied the Newtonian limit
of this system in weighted Sobolev spaces of integer order
\cite{oliynyk07:_newton_limit_perfec_fluid
}

We will briefly resume the situation in the mathematical theory of
self gravitation perfect fluids describing compact bodies, such as
stars: For the Euler-Poisson system Makino proved a local existence
theorem in the case the density has compact support and it vanishes at
the boundary,
\cite{makino86:_local_exist_theor_evolut_equat_gaseous_stars
}.  Since the Euler equations are singular when the density $\rho$ is
zero, Makino had to regularize the system by introducing a new matter
variable ($w=M(\rho)$).  His solution however, has some disadvantages
such as the fact they do not contain static solutions and moreover,
the connection between the physical density and the new matter density
remains obscure.

Rendall generalized Makino's result to the relativistic case of the
Einstein--Euler equations,
\cite{Rendall92:-fluid
}. His result however suffers from the same disadvantages as Makino's
result and moreover it has two essential restrictions: 1. Rendall
assumed time symmetry, that means that the extrinsic curvature of the
initial manifold is zero and therefore the Einstein's constraint
equations are reduced to a single scalar equation; 2. Both the data
and solutions are $C^{\infty}_0$ functions. This regularity condition
implies a severe restriction on the equation of state
$p=K\epsilon^{\gamma}$, namely $\gamma\in\setN $.

Similarly to Makino and Rendall, we have also used the Makino variable
\begin{equation}
  \label{eq:eineul:11}
  w=M(\epsilon)=\epsilon^{\frac{\gamma-1}{2}}.
\end{equation}
Our approach is motivated by the following observation. As it turns
out, the system of evolution equations have the following form
\begin{equation}
  \label{eq:publ-broken:13}
  A^0 \partial_t U + \sum_{k=1}^3  A^k \partial_k U = Q(\epsilon,..),
\end{equation}
where the unknown $U$ consists of the gravitational field
$g_{\alpha\beta}$ the velocity of the fluid $u^\alpha$ and the Makino
variable $w$, and the lower order term $Q$ contains the energy density
$\epsilon$.  Thus, we need to estimate $\epsilon$ by $w$ in the
corresponding norm of the function spaces.  Combining this estimation
with the Makino variable (\ref{eq:eineul:11}), it results in an
algebraic relation between the order of the functional space $k$ and
the coefficient $\gamma$ of the equation of state (\ref{eq:eineul:4})
of the form
\begin{equation}
  \label{eq:publ-broken:1}
  1<\gamma \leq \frac{2+k}{k}.
\end{equation}
This relation can be easily derived by considering $\|\partial^{\alpha} w\|_{L_2}$,
$|\alpha|\leq k$. Moreover, it can be interpreted either as a restriction on
$\gamma$ or on $k$. Thus, unlike typical hyperbolic systems where often the
regularity parameter is bounded from below, here we have both lower
and upper bounds for differentiability conditions of the sort
$\frac{5}{2}<k\leq\frac{2}{\gamma-1}$.  A similar phenomenon for the
Euler-Poisson equations was noted by Gamblin
\cite{gamblin93:_solut_euler}.

We want to interpret (\ref{eq:publ-broken:1}) as a restriction on $k$
rather than on $\gamma$.  Therefore, instead of imposing conditions on
the equation of state and in order to sharpen the regularity
conditions for existence theorems, we are lead to the conclusion of
considering function spaces of fractional order, and in addition, the
Einstein equations consist of quasi linear hyperbolic and elliptic
equations.  The only function spaces which are known to be useful for
existence theorems of the constraint equations in the asymptotically
flat case, are the weighted Sobolev spaces $H_{k,\delta}$, $k\in\setN
$, $\delta\in\setR$, which were introduced by Nirenberg and Walker,
\cite{nirenberg73:_null_spaces_ellip_differ_operat_r}
and Cantor
\cite{cantor75:_spaces_funct_condit_r
}, 
and they are the completion of $C_0^\infty(\mathbb{R}^3)$ under the
norm
\begin{equation}
  \label{eq:intro:5} \left(\|u\|_{k,\delta}\right)^{2}
  =\sum_{|\alpha|\leq k}\int\left((1+| x|)^{\delta+|\alpha|}
    |\partial^\alpha u|\right)^2dx. 
\end{equation}

Hence we are forced to consider new function spaces $H_{s,\delta}$, $s\in\setR $
which generalize the spaces $H_{k,\delta}$ to fractional order.  The well
posedness of the Einstein-Euler system is obtained in these spaces.
In order to achieve this, we have to solve both the constraint and the
evolution equations in the $H_{s,\delta}$ spaces.

Another difficulty which arises from the non-linear equation of state
(\ref{eq:eineul:4}) is the compatibility problem of the initial data
for the fluid and the gravitational field.  There are three types of
initial data for the Einstein-Euler system:
\begin{itemize}
  \item The gravitational data is a triple $(M,h,K)$, where $M$ is
  space-like manifold, $h=h_{ab}$ is a proper Riemannian metric on $M$
  and $K=K_{ab} $ is a second fundamental form on $M$ (extrinsic
  curvature). The pair $(h,K)$ must satisfy the constrain equations
  \begin{equation}
    \label{eq:intro:6}
    \left\{
      \begin{array}{ccc}
        R(h) - K_{ab}K^{ab}+(h^{ab}K_{ab})^2 &=&16\pi z,\\
        {}^{(3)}\nabla_b K^{ab}- 
        {}^{(3)}\nabla^b(h^{bc}K_{bc}) &=& -8\pi j^a,
      \end{array}\right.
  \end{equation}
  where $R(h)=h^{ab}R_{ab}$ is the scalar curvature with respect to
  the metric $h$. \item The matter variables, consisting of the energy
  density $z$ and the momentum density $j^a$, appear in the right hand
  side of the constraints (\ref{eq:intro:6}). \item The initial data
  for Makino's variable $w$ and the velocity vector $u^\alpha$ of the
  perfect fluid.
\end{itemize}

The projection of the velocity vector $u^\alpha$, $\bar u^\alpha$, on
the tangent space of the initial manifold $M$ leads to the following
relations
\begin{equation}
  \label{eq:intro:7}
  \left\{
    \begin{array}{ccc}
      z &= &\epsilon+(\epsilon+p)h_{ab}\bar{u}^a\bar{u}^b\\
      j^\alpha
      &=&(\epsilon+p)\bar{u}^a\sqrt{1+ h_{ab}\bar{u}^a\bar{u}^b}
    \end{array}\right.
\end{equation}
between the matters variable $(z,j^a)$ and $(w,\bar u^a)$. We cannot
give $w$, $\bar{u}^b$ and by relations (\ref{eq:eineul:4}) and
(\ref{eq:intro:7}) solve for $z$ and $j^{\alpha}$,
since this is incompatible with the conformal scaling (see Section
\ref{sec:comp-probl-init}). In order to overcome this obstacle, we
let $z=y^{\frac{2}{\gamma-1} }$ and $j^\alpha=v^\alpha/z$, then
(\ref{eq:intro:7}) is equivalent to (\ref{eq:The_Initial_Data:7})
and the last one is  invertible.

The paper is organized as follows: In the next section we perform the
reduction of the Einstein-Euler system into a first order symmetric
hyperbolic system.  Choquet-Bruhat showed that the choice of harmonic
coordinates converts the field equations (\ref{eq:eineul:1}) into
wave equations  which then can be written as a first order
symmetric hyperbolic system
\cite{choquet--bruhat52
},
\cite{CHY
}, \cite{HAE
}.  Reducing the Euler equations (\ref{eq:eineul:3}) to a first order
symmetric hyperbolic system is not a trivial matter.  We use a fluid
decomposition and present a new reduction of the Euler equations.
Beside having a very clear geometric interpretation, we give a
complete description of the structure of the characteristics conformal
cone of the system, namely, it is a union of a three-dimensional
hyperplane tangent to the initial manifold and the sound cone.

In Section \ref{sec:New_function_spaces} we define the weighted
Sobolev spaces of fractional order $H_{s,\delta}$ and present our main
results. These include a solution of the compatibility problem, the
construction of initial data and a solution to the evolution equations
in the $H_{s,\delta}$ spaces. The announcement of the main results has
been published in
\cite{ICH12
}.

Section \ref{sec:The_Initial_Data} deals with the constructions of the
initial data. The common Lichnerowicz-York  scaling method for solving
the constraint equations cannot be
applied here directly \cite{CHY},
\cite{cantor79
}, \cite{york73:_confor_rieman
}, since it violates the relations
(\ref{eq:intro:7}). We need to invert of (\ref{eq:intro:7}) in order
construct the initial data and there are two conditions which
guarantee it: the dominate energy condition $h_{ab}j^a j^b\leq z^2$,
this  is invariant under scaling; and the causality condition,
this is the speed of sound (\ref{eq:publ-broken:14}) less than speed
of light.  Unfortunately the last condition is not invariant under
scaling.  It is also necessary to restrict the matter variables
$(z,j^a)$ to a certain region.  We show the inversion of
(\ref{eq:intro:7}) exists provided that $(z,j^a)$ belong to a
certain region.  This fact enables us to construct initial data
for the evolution equations.

The local existence for first order symmetric hyperbolic systems in
$H_{s,\delta}$ is discussed in Section \ref{sec:local-exist-hyperb}.
The known existence results in the $H^s$ space
\cite{FMA
},
\cite{KATO
},
\cite{hughes76:_well
},
\cite{Taylor91
},
\cite{taylor97c
},
\cite{majda84:_compr_fluid_flow_system_conser
} cannot be applied to the $H_{s,\delta}$ spaces.  The main difficulty
here is the establishment of energy estimates for linear hyperbolic
systems. In order to achieve it we have defined a specific
inner-product in $H_{s,\delta}$ and in addition the Kato-Ponce
commutator estimate
\cite{kato88:_commut_euler_navier_stokes
},
\cite{Taylor91
}, \cite{taylor97c
} has an essential role in our approach. Once the energy estimates and
other tools have been established in the $H_{s,\delta}$ space, we
follow Majda's \cite{majda84:_compr_fluid_flow_system_conser}
iteration procedure and show existence, uniqueness and continuity in
that norm.

In Section \ref{sec:elliptic} we study elliptic theory in
$H_{s,\delta}$ which is essential for the solution of the constraint
equations.  We will extend earlier results in weighted Sobolev spaces
of integer order which were obtained by Cantor
\cite{cantor79
}, Choquet-Bruhat and Christodoulou
\cite{choquet--bruhat81:_ellip_system_h_spaces_manif_euclid_infin%
}, Choquet-Bruhat, Isenberg and York
\cite{y.00:_einst_euclid
}, and Christodoulou and O'Murchadha
\cite{OMC
} to the fractional ordered spaces. The central tool is a priori
estimate for elliptic systems in the $H_{s,\delta}$ spaces
(\ref{eq:a_priori_estimates_weighted:9}). Its proof requires first the
establishment of analogous a priori estimate in Bessel potential
spaces $H^s$. Our approach is based on the techniques of
pseudodifferential operators which have symbols with limited
regularity and in order to achieve that we are adopting ideas being
presented in Taylor's books
\cite{Taylor91
}
and
\cite{taylor00
}.
A different method was derived recently by Maxwell
\cite{maxwell06:_rough_einst} who also showed existence of solutions
to Einstein constraint equations in vacuum in $H_{s,\delta}$ with the
best possible regularity condition, namely $s>\frac{3}{2}$. The
semi-linear elliptic equation is solved by following Cantor's homotopy
argument \cite{cantor79} and generalized to the  $H_{s,\delta}$ spaces.

Finally, in the Appendix we deal with of the construction, properties
and tools for PDEs in the weighted Sobolev spaces of fractional order
$H_{s,\delta}$.  Triebel extended the $H_{k,\delta}$ spaces given by
the norm
(\ref{eq:intro:5}) to a fractional order
\cite{triebel76:_spaces_kudrj2}, 
\cite{triebel95
}. We present three
equivalent norms, one of which is a combination of the norm
(\ref{eq:intro:5}) and the norm of Lipschitz-Sobolevskij spaces
\cite{stein70:_singul_integ_differ_proper_funct}.  This definition is
essential for the understanding of the
relations between the integer and the fractional order spaces (see
(\ref{eq:const:3})).  However the double integral makes it almost
impossible to establish {any property} needed for PDEs.  Throughout
the effort to solve this problem, we were looking for an equivalent
definition of the norm: we let $\{\psi_j\}_{j=0}^\infty$ be a dyadic
resolution
of unity in $\mathbb{R}^3$ and set
\begin{equation}
  \label{eq:intro:9}
  \left(\|u\|_{H_{s,\delta}}\right)^2=
  \sum_{j=0}^\infty 2^{( \frac{3}{2} + \delta)2j} \| (\psi_j  u)_{2^j}
  \|_{H^{s}}^{2},
\end{equation}
where $( f)_\epsilon(x)=f(\epsilon x)$. When $s$ is an integer, then
the norms (\ref{eq:intro:5}) and (\ref{eq:intro:9}) are equivalent.
Our guiding philosophy is to apply the known properties of the Bessel
potential spaces $H^s$ term-wise to each of the norms in the infinite
sum (\ref{eq:intro:9}) and in that way to extend them to the
$H_{s,\delta}$ spaces.  Of course, this requires a careful treatment
and a sound consideration of the additional parameter $\delta$. Among
the properties which we have extended to the $H_{s,\delta}$ spaces are
the algebra, Moser type estimates, fractional power,
the embedding to the continuous and an intermediate estimate.


\section{The Initial Value Problem for the Euler-Einstein System}
\label{sec:init-value-probl}

This section deals with the reduction of the coupled evolution
equations (\ref{eq:eineul:1}), (\ref{eq:eineul:3})  and
(\ref{eq:eineul:4}) into a first order symmetric hyperbolic system.

\subsection{The Euler equations written as a symmetric hyperbolic
  system}
\label{sec:euler-equat-writt}

It is not obvious that the Euler equations written in the conservative
form $\nabla_{\alpha}T^{\alpha\beta}=0$ are symmetric hyperbolic. In
fact these equations have to be transformed in order to be expressed
in a symmetric hyperbolic form. Rendall presented such a
transformation of these equations in
\cite{Rendall92:-fluid
}, however, its geometrical meaning is not entirely clear and it might
be difficult to generalize it to the non time symmetric case.  Hence
we will present a different hyperbolic reduction of the Euler
equations and discuss it in some details, for we have not seen it
anywhere in the literature. The basic idea is to perform the standard
\textit{fluid decomposition} and then to modify the equation by
adding, in an appropriate manner, the normalization condition
(\ref{eq:publ-broken:2}) which will be considered as a constraint
equation.

The fluid decomposition method consists of:
\begin{enumerate}
  \item The equation $ \nabla_{\nu}T^{\nu\beta} = 0 $ is once
  projected orthogonal onto $u^{\alpha}$ which leads to
  \begin{equation}
    u_\beta\nabla_{\nu}T^{\nu\beta} =  0
    \label{kap3.flp20}.
  \end{equation}
  \item The equation $ \nabla_{\nu}T^{\nu\beta} = 0 $ is projected
  into the rest pace ${\cal O}$ orthogonal to $u^\alpha$ of a fluid
  particle gives us:
  \begin{equation}
    P_{\alpha\beta}\nabla_{\nu}T^{\nu\beta} = 0
    \qquad \mbox{with}  \quad
    P_{\alpha\beta}= g_{\alpha\beta}+u_\alpha u_\beta,
    \quad
    P_{\alpha\beta}u^{\beta}=0.
    \label{kap3.flp1}
  \end{equation}
\end{enumerate}

Inserting this decomposition into (\ref{eq:eineul:2}) results in a 
system in the following form:
\begin{subequations}
  \begin{eqnarray}
    \label{eq:euler-rel:10}
    u^{\nu}\nabla_{\nu}\epsilon + (\epsilon+p) \nabla_{\nu}u^{\nu}
    &=& 0; \\[0.2cm]
    \label{kap3.flz3}
    (\epsilon+p)   P_{\alpha\beta}u^{\nu}\nabla_{\nu}{u^\beta} +
P^\nu{}_\alpha
    \nabla_\nu p     &=& 
    0.
  \end{eqnarray}
\end{subequations}

Note that we have beside the evolution equations
(\ref{eq:euler-rel:10}) and (\ref{kap3.flz3}) the following constraint
equation: $g_{\alpha\beta}u^{\alpha}u^{\beta}=-1$.  We will show
later, in subsection \ref{sec:cons-constr-equat} that this constraint
equation is conserved under the evolution equation, that is, if it
holds initially at $t=t_0$, then it will hold for $t>t_0$.  Note that
in most textbooks, the equation (\ref{kap3.flz3}) is presented as
$(\epsilon+p)
g_{\alpha\beta}u^{\nu}\nabla_{\nu}{u^\beta} + P^\nu{}_\alpha
\nabla_\nu p = 0$, which is an equivalent form, since due to the
normalization condition (\ref{eq:publ-broken:2}) we have
$u_{\beta}\nabla_{\nu}u^{\beta}=0$.

In order to obtain a symmetric hyperbolic system  we have to
modify it in the following way. The normalization condition
(\ref{eq:publ-broken:2}) gives that $ u_\beta u^\nu\nabla_\nu u^\beta
= 0$, so we add $(\epsilon+p) u_\beta u^\nu \nabla_\nu u^\beta = 0$ to
equation
(\ref{eq:euler-rel:10}) and $ u_\alpha u_\beta u^\nu\nabla_\nu u^\beta
= 0$ to (\ref{kap3.flz3}), which together with
(\ref{eq:publ-broken:14}) results in,
\begin{subequations}
  \begin{eqnarray}
    \label{eq:eineul:8}
    u^{\nu}\nabla_{\nu} \epsilon + (\epsilon+p)
P^\nu{}_\beta\nabla_\nu
u^\beta&=& 0 \\
    \label{eq:publ-broken:3} \Gamma_{\alpha\beta}
    u^{\nu}\nabla_{\nu}u^{\beta} +\frac{\sigma^2}{(\epsilon+p)}
P^\nu{}_\alpha
    \nabla_\nu \epsilon  &=& 0,
  \end{eqnarray}
\end{subequations}

where $\Gamma_{\alpha\beta}= P_{\alpha\beta}+u_{\alpha}u_{\beta}=
g_{\alpha\beta}+2u_{\alpha}u_{\beta}$. As mentioned above we will
introduce a new nonlinear matter variable which is given by
(\ref{eq:eineul:11}). The idea which is behind this is the following:
The system (\ref{eq:eineul:8}) and (\ref{eq:publ-broken:3}) is almost
of symmetric hyperbolic form, it would be symmetric if we multiply the
system by appropriate factors, for example, (\ref{eq:eineul:8}) by $
\frac{\partial p}{\partial \epsilon }=\sigma^2$ and
(\ref{eq:publ-broken:3}) by
$(\epsilon+p)$. However, doing so we will be faced with a system in
which the
coefficients will either tend to zero or to infinity, as $\epsilon\to
0$. Hence, it is impossible to represent this system in a
non-degenerate form using these multiplications.

The central point is now to introduce a new variable $w=M(\epsilon)$
which will regularize the equations even for $\epsilon=0$. We do this
by multiplying equation (\ref{eq:eineul:8}) by $\kappa^2
M'=\kappa^2\frac{\partial M}{\partial\epsilon}$. This results in the
following system which we have written in matrix form:
\begin{equation}
  \label{eq:Initial:1}
  \left(
    \begin{array}{c|cccc}
      \kappa^2u^\nu    &  & \kappa^{2}(\epsilon+p) M^{\prime}
      P^{\nu}{}_{\beta}  &     \\ \hline
      &  &                     &  &  \\
      \frac{\sigma^2}{(\epsilon+p) M^{\prime}}  P^{\nu}{}_{\alpha}
      &  & \Gamma_{\alpha\beta} u^\nu       &     \\
      &  &                     &  &  \\
    \end{array}
  \right)\nabla_\nu\left(
    \begin{array}{c}
      w \\
      u^\beta \\
    \end{array}
  \right)=\left(
    \begin{array}{c}
      0 \\
      0 \\
    \end{array}
  \right).
\end{equation}

In order to obtain symmetry we have to demand
\begin{equation}
  \label{eq:Initial:6}
  M'=\frac{\sigma}{(\epsilon+p) \kappa},
\end{equation}
where ${\kappa} \gg0$ has been introduced in order to simplify the
expression for $w$. We choose $\kappa$ so that
\begin{equation}
  \label{eq:Initial:7}
  \frac{\sqrt{f'(\epsilon)}}{(\epsilon+p) \kappa}=
  \frac{2}{\gamma-1}\frac{\epsilon^{\frac{\gamma-1}{2}}}{\epsilon},
\end{equation}
which gives the Makino variable (\ref{eq:eineul:11}). Taking into
account the equation of state (\ref{eq:eineul:4}), we see that
\begin{equation}
  \label{eq:Initial:8} {\kappa}=\frac{\gamma-1}{2}
  \frac{\sqrt{K\gamma}}{1+K\epsilon^{\gamma-1}}\gg0.
\end{equation}

Finally we have obtained the following system
\begin{equation}
  \label{eq:Initial:2} \left(
    \begin{array}{c|cccc}
      \kappa^2 u^\nu    &  & \sigma\kappa P^{\nu}{}_{\beta}  &     \\
\hline
      &  &                     &  &  \\
      \kappa\sigma   P^{\nu}{}_{\alpha}
      &  & \Gamma_{\alpha\beta} u^\nu       &     \\
      &  &                     &  &  \\
    \end{array}
  \right)\nabla_\nu\left(%
    \begin{array}{c}
      w \\
      u^\beta \\
    \end{array}%
  \right)=\left(%
    \begin{array}{c}
      0 \\
      0 \\
    \end{array}%
  \right),
\end{equation}
which is both symmetric and non-degenerated. The covariant derivative
$\nabla_{\nu}$ takes in local coordinates the form $\nabla_{\nu} =
\partial_{\nu} + \Gamma(g^{\gamma\delta},\partial g_{\alpha\beta})$
which expresses the fact that the fluid $u^\alpha$ is coupled to
equations (\ref{eq:eineul:1}) for the gravitational field
$g_{\alpha\beta}$. In addition, from the definition of the Makino
variable (\ref{eq:eineul:11}) we see that $\epsilon^{\gamma-1}=w^2$,
so from the expression (\ref{eq:publ-broken:14}), is follows that
$\sigma=\sqrt{\gamma K}w$ and $\kappa$ --which is given by
(\ref{eq:Initial:8})-- are $C^\infty$ functions of $w$.  Thus the
fractional power of the equation of state (\ref{eq:eineul:4}) does not
appear in the coefficients of the system (\ref{eq:Initial:2}), and
these coefficients are $C^\infty$ functions of the scalar $w$, the
four vector $u^\alpha$ and the gravitational field $g_{\alpha\beta}$.

Let us now recall a general definition of symmetric hyperbolic
systems.
\begin{defn}[First order symmetric hyperbolic systems]
  \label{def:euler-rel:1}
  A quasilinear, symmetric hyperbolic system is a system of
  differential equations of the form
  \begin{equation}
    \label{eq:publ-broken:7}
    L\lbrack   U \rbrack      =      \sum_{\alpha = 0}^{4}
    A^\alpha (U;x)\partial_\alpha U  +
    B(U;x) = 0 
  \end{equation}
  where the matrices $A^\alpha $ are symmetric and for every arbitrary
  $U\in G $ there exists a covector $\xi$ such that
  \begin{equation}
    \label{eq:publ-broken:8}
    \xi_\alpha   A^\alpha (U;x)
  \end{equation} is  positive   definite.   The  covectors 
$\xi_\alpha$
  for which (\ref{eq:publ-broken:8}) is positive definite, are 
  {\it spacelike with respect to the equation
}(\ref{eq:publ-broken:7}).
  Both matrices $A^\alpha$, $B$ satisfy certain regularity conditions,
  which are going to be formulated later.
\end{defn}

Usually $\xi$ is chosen to be the vector $(1,0,0,0)$ which implies via
the condition (\ref{eq:publ-broken:8}) that the matrix $A^0$ has to be
is positive definite.

Now we want to show that $A^0$ of our system (\ref{eq:Initial:2}) is
indeed positive definite. We do this in several steps.

\begin{enumerate}
  \item Explicit computation of the principle symbol
  (\ref{eq:Initial:2});
  \item We show that $-u_{\alpha }$ is a space like covector with
  respect to the equations;
  \item Then we apply a deformation argument and show that the
  covector $t_{\alpha}:=(1,0,0,0)$ is a space like covector with
  respect to the equation.
\end{enumerate}
For each $\xi_\alpha\in T_x^*V$ the principle symbol is a linear map
from $\mathbb{R}\times E_x$ to $\mathbb{R}\times F_x$, where $E_x$ is
a fiber in $T_xV$ and $F_x$ is a fiber in the cotangent space
$T_x^*V$.
Since in local coordinates $\nabla_{\nu} =\partial_{\nu} +
\Gamma(g^{\gamma\delta},\partial g_{\alpha\beta})$, the principle
symbol of system (\ref{eq:Initial:2}) is
\begin{equation}
  \label{eq:Initial:4}\xi_\nu A^\nu= \left(
    \begin{array}{c|cccc}
      \kappa^2    (u^\nu \xi_\nu)   &  & \sigma\kappa
      P^{\nu}{}_{\beta}\xi_\nu  &     \\ \hline
      &  &                     &  &  \\
      \sigma\kappa   P^{\nu}{}_{\alpha}\xi_\nu
      &  & (u^\nu\xi_\nu)\Gamma_{\alpha\beta}        &     \\
      &  &                     &  &  \\
    \end{array}
  \right)
\end{equation}

and the characteristics are the set of covectors for which $(\xi_\nu
A^\nu)$ is not an isomorphism. Hence the characteristics are the zeros
of $Q(\xi):=\det(\xi_\nu A^\nu)$.

The geometric advantages of the fluid decomposition are the following.
The operators in the blocks of the matrix (\ref{eq:Initial:4}) are
$P^{\nu}{}_{\alpha}$, the projection on the rest hyperplane
$\mathcal{O}$ and $ \Gamma_{\alpha\beta}$, that is the  reflection with respect
to the same hyperplane. Therefore, the following relations hold:
\begin{displaymath}
  \Gamma^{\alpha\gamma}\Gamma_{\gamma\beta}=\delta_\beta{}^\alpha,
  \qquad \Gamma^{\alpha\gamma}P_\gamma{}^\nu=P^{\alpha\nu}\qquad
  \text{and}\qquad P_\beta{}^\alpha P_{\alpha}{}^\nu=P^{\nu}{}_\beta,
\end{displaymath}
which yields
\begin{equation}
  \label{eq:Initial:5}\left(
    \begin{array}{c|cccc}
      1   &  & 0  &     \\ \hline
      &  &                     &  &  \\
      0
      &  &\Gamma^{\alpha\gamma}         &     \\
      &  &                     &  &  \\
    \end{array}
  \right)
  \left(\xi_\nu A^\nu\right)= \left(
    \begin{array}{c|cccc}
      \kappa^2          (u^\nu \xi_\nu)   &  & \sigma\kappa
      P^{\nu}{}_{\beta}\xi_\nu  &     \\ \hline 
      &  &                     &  &  \\
      \sigma\kappa   P^{\alpha\nu}{}\xi_\nu &
      &(u^\nu\xi_\nu)\left(\delta^{\alpha}_{\beta} \right)       &    
\\ 
      &  &                     &  &  \\
    \end{array}
  \right).
\end{equation}

It is now fairly easy to calculate the determinate of the right hand
side of (\ref{eq:Initial:5}) and we have
\begin{displaymath}
  \det\left(
    \begin{array}{c|cccc}
      \kappa^2      (u^\nu \xi_\nu)   &  & \sigma\kappa
      P^{\nu}{}_{\beta}\xi_\nu  &     \\ \hline 
      &  &                     &  &  \\
      \sigma\kappa   P^{\alpha\nu}{}\xi_\nu &
      &(u^\nu\xi_\nu)\left(\delta^{\alpha}_{\beta} \right)       &    
\\ 
      &  &                     &  &  \\
    \end{array}
  \right)
  =\kappa^2(u^\nu\xi_\nu)^3\left((u^\nu\xi_\nu)^2 -\sigma^2
    P^{\alpha\nu}\xi_\nu P_\alpha ^\nu\xi_\nu\right). 
\end{displaymath}

Since $P_\beta^\alpha$ is a projection,
\begin{equation}
  P^{\alpha\nu}\xi_\nu P_\alpha ^\nu\xi_\nu= g^{\nu\beta}\xi_\nu
  P^{\alpha}_\beta P_\alpha ^\nu\xi_\nu 
  =g^{\nu\beta}\xi_\nu P^{\nu}{}_\beta \xi_\nu
  =P^{\nu}{}_\beta\xi_\nu  \xi^\beta
\end{equation}
and since $\Gamma_\beta^\gamma:\mathbb{R}^4\to\mathbb{R}^4$ is a
reflection with respect to a hyperplane,
\begin{equation}
  \det\left(
    \begin{array}{c|cccc}
      1   &  & 0  &     \\ \hline

      0
      &  &\Gamma^{\alpha\gamma}         &     \\
    \end{array}
  \right) =\det\left(g^{\alpha\beta}\Gamma_\beta^\gamma\right)
  ={\rm det}\left(g^{\alpha\beta}\right){\rm
    det}\left(\Gamma_\beta^\gamma\right)= -\left({\rm
      det}\left(g_{\alpha\beta}\right)\right)^{-1}.
\end{equation}

Consequently,
\begin{equation}
  \label{eq:Initial:11}
  Q(\xi ):=  \det (\xi_\nu A^\nu  ) =-\kappa^{2}\det(g_{\alpha\beta})
  ( u^\nu\xi_\nu)^3
  \left\{ ( u^\nu\xi_\nu)^2 - \sigma^2
    P^{\alpha}{}_{\beta}\xi_\alpha \xi^\beta \right\}
\end{equation}
and therefore the characteristic covectors are given by two simple
equations:
\begin{eqnarray}
  \label{eq:publ-broken:39}
  \xi_\nu u^\nu & = &0 ;\\
  \label{eq:publ-broken:40} (\xi_\nu u^\nu)^2 - \sigma^2
  P^{\alpha}{}_{\beta}\xi_\alpha \xi^\beta & = &0.
\end{eqnarray}
\begin{rem}[The structure of the characteristics conormal cone of ]
  \label{rem:euler-rel:7}
  The characteristics conormal cone is therefore a union of two
  hypersurfaces in $T_x^*V$.  One of these hypersurfaces is given by
  the condition (\ref{eq:publ-broken:39}) and it is a three
  dimensional hyperplane $\mathcal{O}$ with the normal $u^\alpha$.
  The other hypersurface is given by the condition
  (\ref{eq:publ-broken:40}) and forms a three dimensional cone the so
  called {\it sound cone}.
\end{rem}
\begin{rem}
  \label{rem:euler-rel:8}
  Equation (\ref{eq:publ-broken:40}) plays an essential role in
  determining whether the equations form a symmetric hyperbolic
  system.
\end{rem}

Let us now consider the timelike vector $u_{\nu }$ and the linear
combination $-u_\nu A^\nu$, with $A^{\nu}$ from equation
(\ref{eq:Initial:2}), we then obtain that
\begin{equation}
  -u_\nu A^\nu=\left(
    \begin{array}{c|cccc}
      \kappa^2 &  & 0         &     \\
      \hline
      &  &           &  &  \\
      0 &  & \Gamma_{\alpha\beta} &     \\
      &  &           &  &  \\
    \end{array}
  \right) \label{kap3.mk3}
\end{equation}
is positive definite. Indeed, $\Gamma_{\alpha\beta}$ is a reflection
with respect to a hyperplane. The normal of this hyperplane is a
timelike vector.  Hence, $-u_\nu$ is for the hydrodynamical equations
a spacelike covector in the sense of partial differential equations.
Herewith one has showed relatively elegant and elementary that the
relativistic hydrodynamical equations are symmetric--hyperbolic.

Now we want however to show that the covector $t_{\alpha}=(1,0,0,0)$
is spacelike with respect to the system (\ref{eq:Initial:2}). Since
$P^\alpha{}_{\beta}u_\alpha=0$, the covector $-u_\nu$ belongs to the
sound cone
\begin{equation}
  \label{eq:Initial:9}
  (\xi_\nu u^\nu)^2 - \sigma^2P^{\alpha}{}_{\beta}\xi_\alpha
  \xi^\beta>0.
\end{equation}
Inserting $t_\nu=(1,0,0,0)$ the right hand side of
(\ref{eq:Initial:9}) yields
\begin{equation}
  \label{eq:Initial:10}
  (u^0)^2(1-\sigma^2)-\sigma^2 g^{00}.
\end{equation}
Since the sound velocity is always less than the light speed, that is
$\sigma^2=\frac{\partial p}{\partial \epsilon}<c^2=1$, we conclude
from (\ref{eq:Initial:10}) that $t_\nu$ also belongs to the sound cone
(\ref{eq:Initial:9}). Hence, the vector $-u_\nu$ can be continuously
deformed to $t_\nu$ while condition (\ref{eq:Initial:9}) holds along
the deformation path. Consequently, the determinant of
(\ref{eq:Initial:11}) remains positive under this process and hence
$t_{\nu}A^\nu=A^{0}$ is also positive definite.

\subsubsection{Conservation of the constraint equation
  $g_{\alpha\beta}u^{\alpha}u^{\beta}=-1$}
\label{sec:cons-constr-equat}

Now it will be shown that the condition $g_{\alpha\beta} u^\alpha
u^\beta = -1$, which acts as a constraint equation for the evolution
equation, is conserved along stream lines $u^\alpha$.  Because, if for
$t=t_0$ the condition $g_{\alpha\beta} u^\alpha u^\beta = -1$ holds
and if it is conserved a long stream lines, then $g_{\alpha\beta}
u^\alpha u^\beta = -1$ holds also for $t>t_0$.  So let $c(t)$ be a
curve such that $c'(t)=u^\alpha$ and set $Z(t)=(u\circ c)_\beta(u\circ
c)^\beta$, then we need to establish
\begin{equation}
  \frac{d}{dt}Z(t)=2 u_\beta \nabla_{c'(t)} u^\beta
  =2u^\nu u_\beta\nabla_\nu u^\beta= 0.
  \label{kap3.43A}
\end{equation}
Multiplying the last four last rows of the Euler system
(\ref{eq:Initial:2}) by $u^\alpha$ and recalling that
$P^{\nu}{}_\alpha$ is the projection on the rest space $\mathcal{O}$
orthogonal to $u^\alpha$, we have
\begin{eqnarray*}
  0 &=& u^\alpha\left(\Gamma_{\alpha\beta}u^\nu\nabla_\nu u^\beta +
    \kappa\sigma P^{\nu}{}_\alpha\nabla_\nu w\right)\\
  &=& u^\alpha P_{\alpha\beta}u^\nu\nabla_\nu u^\beta - u^\nu
  u_\beta\nabla_\nu u^\beta 
  + \kappa\sigma u^\alpha P^{\nu}{}_\alpha\nabla_\nu w\\ &=&
  - u^\nu u_\beta\nabla_\nu u^\beta.
\end{eqnarray*}

\subsection{The reduced Einstein field equations}
\label{sec:evolution-equations}

In this paper we study the field equations (\ref{eq:eineul:1}) with
the choice of the harmonic coordinate  condition which takes the form
\begin{equation}
  \label{eq:publ-broken:16}
 H^\alpha\equiv g^{\alpha\beta}g^{\gamma\delta}(\partial_\gamma
g_{\beta\delta}-
  \frac{1}{2}
  \partial_\delta g_{\beta\gamma})=0. 
\end{equation}
When (\ref{eq:publ-broken:16}) is imposed, then the Einstein equations
(\ref{eq:eineul:1}) convert to
\begin{equation}
  \label{eq:publ-broken:18}
  g^{\mu\nu}\partial_\mu \partial_\nu g_{\alpha\beta}=
  H_{\alpha\beta}({g},\partial{g})-16\pi T_{\alpha\beta}+8\pi
  g^{\mu\nu}T_{\mu\nu}g_{\alpha\beta}. 
\end{equation}
Hawking and Ellis proved the conservation of the harmonic coordinates
for Einstein equations with matter including a perfect fluid
\cite{HAE}.  Since (\ref{eq:publ-broken:18}) are quasi linear wave
equations, the introducing auxiliary variables
\begin{equation}
  \label{eq:publ-broken:17} h_{\alpha\beta\gamma}=  \partial_\gamma
  g_{\alpha\beta},
\end{equation}
reduce them into a first order symmetric hyperbolic system:
\begin{equation}
  \label{eq:publ-broken:19}
  \begin{array}{ccl}
    \partial_t  g_{\alpha\beta} &=& h_{\alpha\beta  0} \\
    g^{ab}\partial_t  h_{\gamma\delta  a} & = &g^{ab}\partial_a 
h_{\gamma\delta 0} \\
    -g^{00}\partial_t  h_{\gamma\delta  0} &=&
    2g^{0a}\partial_a h_{\gamma\delta   0} + g^{ab}\partial_a 
h_{\gamma\delta b} \\
    &&+
   
C_{\gamma\delta\alpha\beta\rho\sigma}^{
\epsilon\zeta\eta\kappa\lambda\mu}
    h_{\epsilon\zeta\eta}h_{\kappa\lambda\mu}
    g^{\alpha\beta}  g^{\rho\sigma}-16\pi T_{\gamma\delta}+8\pi
    g^{\rho\sigma}  T_{\rho\sigma}g_{\gamma\delta}
  \end{array}
\end{equation}
The object
$C_{\gamma\delta\alpha\beta\rho\sigma}^{
\epsilon\zeta\eta\kappa\lambda\mu}$
is a combination of Kronecker deltas with integer coefficients.  We
therefore conclude:

\begin{con}[The evolution equations in a first order symmetric
  hyperbolic form]
  \label{thm:hyperbolic_reduction}
  The equations for Einstein gravitational field (\ref{eq:eineul:1})
  coupled with the Euler equations (\ref{eq:eineul:3}) with the
  normalization conditions (\ref{eq:publ-broken:2}) and the equation
  of state (\ref{eq:eineul:4}), are equivalent to the system
  (\ref{eq:publ-broken:19}) and (\ref{eq:Initial:2}). The coupled
  systems (\ref{eq:publ-broken:19}) and (\ref{eq:Initial:2}) take the
  form of a first order symmetric hyperbolic system in accordance with
  Definition \ref{def:euler-rel:1} and where $A^0$ is a positive
  definite matrix.
\end{con}


\section{New Function Spaces and the Principle Results}
\label{sec:New_function_spaces}

Our principle results concern the solution of the compatibility of the
initial data for the equations of the fluid and the gravitational
field (\ref{eq:intro:7}), solution to Einstein constraint equations
(\ref{eq:intro:6}) and solution to the coupled evolution
equations (\ref{eq:eineul:1}) and (\ref{eq:eineul:3}).

The coupled evolution equations (\ref{eq:eineul:1}) and
(\ref{eq:eineul:3}) are equivalent to the first order symmetric
hyperbolic systems (\ref{eq:Initial:2}) and (\ref{eq:publ-broken:19})
as we have shown in section \ref{sec:euler-equat-writt}.
The initial data for these coupled systems cannot be given freely,
therefore they are constructed in the following way. Firstly the
compatibility of the initial data for the fluid and the gravitational
field (\ref{eq:intro:7}) have to be solved and next the constraint
equations (\ref{eq:intro:6}), which lead to an elliptic system.
The presence of an equation of state (\ref{eq:eineul:4})  compels us to
treat both the elliptic and hyperbolic
equations in the weighted Sobolev spaces of fractional order.

The Bessel potential spaces $H^s$ which are the natural choice for the
hyperbolic systems are inappropriate for the solutions of the
constraint equations in asymptotically flat manifolds. Roughly
speaking, because the Laplacian is not invertible in these spaces.

As we explained in the introduction, the Nirenberg-Walker-Cantor
weighted Sobolev spaces of integer order
$H_{m,\delta}$ 
\cite{cantor75:_spaces_funct_condit_r},
\cite{nirenberg73:_null_spaces_ellip_differ_operat_r} are suitable for
the solutions of the constraints in asymptotically flat manifolds.
Their norm is given by (\ref{eq:intro:5}).

%
%

We first define the {\it weighted fractional Sobolev spaces}.
We make  a dyadic resolution of the unity in
$\mathbb{R}^3$ as follows.
 Let $K_j=\{x:2^{j-3}\leq |x|\leq 2^{j+2}\}$,
$(j=1,2,...)$ and $K_0=\{x: |x|\leq4\}$.  Let
$\{\psi_j\}_{j=0}^\infty$ be a sequence of $C_0^\infty(\setR^3)$ such
that $ \psi_j(x)=1$ on $K_j$, $\supp(\psi_j)\subset
\cup_{l=j-4}^{j+3}K_{l}$, for $j\geq 1$ and $\supp(\psi_0)\subset
K_0\cup K_{1}$.

We denote by $H^s$ the Bessel potential spaces with the norm ($p=2$)
\begin{displaymath}
  \|u\|_{H^s}^2=c\int(1+|\xi|^2)^s |\hat{u}(\xi)|^2d\xi,
\end{displaymath}
where $\hat{u}$ is the Fourier transform of $u$.  Also, for a function
$f$, $f_\varepsilon(x)=f(\varepsilon x)$.

\begin{defn}[Weighted fractional Sobolev spaces: infinite sum of semi
  norms]
  \label{def:weighted:3}
  For $s\geq 0$ and $-\infty<\delta<\infty$,
  \begin{equation}
    \label{eq:weighted:4}
    \left(\|u\|_{H_{s,\delta}}\right)^2=
    \sum_j 2^{( \frac{3}{2} + \delta)2j} \| (\psi_j u)_{(2^j)}
\|_{H^{s}}^{2}.
  \end{equation}
  The space $H_{s,\delta}$ is the set of all temperate distributions
  with a finite norm given by (\ref{eq:weighted:4}).
\end{defn}

\subsection{The principle results}
\label{sec:principle-results}

\subsubsection{The compatibility of the initial data for the fluid and
  the gravitational field}
\label{sec:comp-init-data}

The matter data (non-gravitational) $(z,j)$ which appear in the
right hand side of (\ref{eq:intro:6}) are coupled to
the initial data of the perfect fluid (\ref{eq:eineul:2}) via the
relations (\ref{eq:intro:7}).
Thus, an indispensable condition for obtaining a
solution of the Einstein-Euler system is the inversion of
(\ref{eq:intro:7}). This system is not invertible for all
$(z,j^a)\in \mathbb{R}_+\times \mathbb{R}^3$, but the inverse does
exist in a certain region.

\begin{thm}[Reconstruction theorem for the initial data]
\label{thm:Reconstruction theorem for the initial data} 
  There is a real function $S:[0,1)\to \mathbb{R}$ such that if
  \begin{equation}
    \label{eq:Functions_spaces:2}
    0\leq z< S(\sqrt{h_{ab}j^{a}j^b}/z),
  \end{equation}
  then system (\ref{eq:intro:7}) has a unique inverse.
  Moreover, the inverse mapping is continuous in $H_{s,\delta}$ norm.
\end{thm}

\begin{rem}
  The matter initial data $(z,j^a)$ for the Einstein-Euler system with
  the equation of state (\ref{eq:eineul:4}) 
cannot be given freely. They must satisfy condition
(\ref{eq:Functions_spaces:2}).
  This condition includes the inequality
  \begin{equation}
    z^2\geq h_{ab}j^a j^b,
  \end{equation}
  which is known as the dominate energy condition.
\end{rem}

\subsubsection{Solution to the constraint equations}
\label{subsubsection:Function_space:1}

The gravitational data is a triple $(M,h,K)$, where $M$ is a
space-like asymptotically flat manifold, $h=h_{ab}$ is a proper
Riemannian metric on $M$, and $K=K_{ab} $ is the second fundamental
form on $M$ (extrinsic curvature). The metric $h_{ab}$ and the
extrinsic curvature $K$ must satisfy Einstein's constraint
equations (\ref{eq:intro:6}).  The free
  initial data is a set $(\bar h_{ab},\bar A_{ab},\hat y,\hat v^a)$,
  where $\bar h_{ab}$ is a Riemannian metric, $\bar A_{ab}$ is
divergence
  and trace free form, $\hat y$ is a scalar function and $\hat v^a$ is
  a vector.

  \begin{thm}[Solution of the constraint equations]
    \label{thm:Function_space:1}
Given free data $(\bar h_{ab}, \bar A_{ab},\hat
      y,\hat v^a)$ such that $(\bar h_{ab}-I)\in H_{s,\delta}$, $ \bar
      A_{ab}\in H_{s-1,\delta+1}$, $(\hat y,\hat v^a)\in
      H_{s-1,\delta+2}$,
      $\frac{5}{2}<s<\frac{2}{\gamma-1}+\frac{3}{2}$ and
      $-\frac{3}{2}<\delta<-\frac{1}{2}$.
    \begin{itemize}
      \item[(i)]  Then there exists two
      positive functions $\alpha$ and $\phi$ such that
      $(\alpha-1),(\phi-1)\in H_{s,\delta}$, a vector field $W\in
      H_{s,\delta}$ such that the gravitational data
      \begin{equation}
 \label{eq:Functions_spaces:5}
 h_{ab}=(\phi\alpha)^4\bar h_{ab} \qquad \text{ and}\qquad
 K_{ab}=(\phi\alpha)^{-2}\bar A_{ab}+\phi^{-2} \hat{\mathcal{L}}(W)
      \end{equation}
      satisfy the constraint equations (\ref{eq:intro:6}) with
      $z=\phi^{-8}\hat y^{\frac{2}{\gamma-1}}$ and $j^b=\phi^{-10}\hat
      y^{\frac{2}{\gamma-1}}\hat v^b$ as the right hand side, here
      $\hat{\mathcal{L}}$ is the Killing vector field operator.  In
      addition, the $H_{s,\delta}\times H_{s-1,\delta+1} $ norms of
      $(h_{ab}-I,K_{ab})$ depend continuously on the $H_{s,\delta}\times
      H_{s-1,\delta+1}\times H_{s-1,\delta+2}$ norms of $(\bar h_{ab}-I, \bar
      A_{ab},\hat y,\hat j^a)$.
      \item[(ii)] Let $\hat h_{ab}=\alpha^4\bar h_{ab}$, $\hat z=\hat
      y^{\frac{2}{\gamma-1}}$, $j^a=\hat y^{\frac{2}{\gamma-1}}\hat v^a$ and
      $\Omega^{-1}$ denote the inverse of relations (\ref{eq:intro:7}). If
      $(\hat h_{ab},\hat z,\hat j^a)$ satisfies
      (\ref{eq:Functions_spaces:2}), then the data for the four
      velocity vector and Makino variable are given by: $z=\phi^{-8}\hat
      z$, $j^a=\phi^{-10}\hat j^a$,
      \begin{equation}
         \label{eq:Functions_spaces:3}
        (w,\bar u^a):=\Omega^{-1}(z,j^a)
        \qquad \text{ and}\qquad \bar u^0=1+h_{ab}\bar u^a\bar u^b
      \end{equation}
      and they satisfy the compatibility conditions
      (\ref{eq:intro:7}).  In addition, the $H_{s-1,\delta+2} $ norms of
      $(w,\bar u^a,u^0-1)$ depend continuously on the $H_{s,\delta}\times
      H_{s-1,\delta+2}$ norms of $(\bar h_{ab}-I,\hat y,\hat j^a)$.
    \end{itemize}
  \end{thm}


\subsubsection{Solution to the evolution equations}
\label{subsubsection:Function space:2}

  The unknowns of the evolution equations are the gravitational field
  $g_{\alpha\beta}$ and its first order partial derivatives
  $\partial_\alpha g_{\gamma\delta}$, the Makino variable $ w$ and the
  velocity vector $u^\alpha$. We represent them by the vector
  $U=\left(g_{\alpha\beta}-\eta_{\alpha\beta},
    \partial_a g_{\gamma\delta}, \partial_0
    g_{\gamma\delta},w,u^{a},u^0-1\right)$, here $\eta_{\alpha\beta}$
  denotes the Minkowski metric. The initial data for equation
(\ref{eq:publ-broken:19}) are 
\begin{equation}
 \label{eq:Functions_spaces:4}
 g_{ab}{_{\mid_M}=h_{ab}}, \ \ g_{0b}{_{\mid_M}}=0, \ \
g_{00}{_{\mid_M}}=-1, \ \ -\frac{1}{2}\partial_0
g_{ab}{_{\mid_M}}=K_{ab},
\end{equation} 
where $(h_{ab},K_{ab})$ are given by (\ref{eq:Functions_spaces:5}),
and (\ref{eq:Functions_spaces:3}) for equation (\ref{eq:Initial:2}).
Theorem \ref{thm:Function_space:1} guarantees that they satisfy  the
constraints (\ref{eq:intro:6}) and the compatibility condition
({\ref{eq:intro:7}).

\begin{thm}[Solutions of the evolution equations
(\ref{eq:publ-broken:19}) and (\ref{eq:Initial:2})]
 Let $\frac{7}{2}<s<\frac{2}{\gamma-1}+\frac{3}{2}$ and
 $-\frac{3}{2}<\delta<-\frac{1}{2}$. Given the solutions of the
 constraint equations as described in Theorem
 \ref{thm:Function_space:1}, then there exists a $T>0$, a unique
 semi-Riemannian metric $g_{\alpha\beta}$ solution to
 (\ref{eq:publ-broken:19}) and a unique pair $(w,u^\alpha)$
solution to (\ref{eq:Initial:2})
 \begin{eqnarray}
   (g_{\alpha\beta}-\eta_{\alpha\beta}) &\in&
   C\left([0,T],H_{s,\delta}\right)\cap
   C^1\left([0,T],H_{s-1,\delta+1}\right)\\
   ( w,u^a,u^0-1) &\in&  C\left([0,T],H_{s-1,\delta+2}\right)\cap
   C^1\left([0,T],H_{s-2,\delta+3}\right).
 \end{eqnarray}
\end{thm}

\section{The Initial Data}
\label{sec:The_Initial_Data}

The Cauchy problem for Einstein fields equations (\ref{eq:eineul:1})
 coupled with the Euler system (\ref{eq:eineul:3}) consists of solving
 coupled hyperbolic systems (\ref{eq:publ-broken:19}) and
(\ref{eq:Initial:2})
with given initial data. There are two types of data for equations
(\ref{eq:eineul:1}), gravitational
and matter data.

The gravitational data is a triple $(M,h,K)$, where $M$ is a
space-like manifold, $h=h_{ab}$ is a proper Riemannian metric on $M$
and $K=K_{ab} $ is the second fundamental form on $M$ (extrinsic
curvature). 

Let $n$ be the unit normal to the hypersurface $M $,
$\delta_\beta^\alpha+n^\alpha n_\beta$ be the projection on $M$ and
define
\begin{eqnarray}
  \label{eq:The_Initial_Data:2}
  z &= &T_{\alpha\beta}n^\alpha n^\beta,\\
  \label{eq:The_Initial_Data:8} j^\alpha
  &=&{(\delta_\gamma^\alpha+n^\alpha n_\gamma)
    T^{\gamma\beta}n_\beta}.
\end{eqnarray}

The scalar $z$ is the energy density and the vector $j^\alpha$ is the
momentum density. These quantities are called matter variables and
they appear as sources in the constraint equations 
(\ref{eq:publ-broken:23}) and (\ref{eq:publ-broken:38}) below.

In conjunction with these we must supply initial data for the velocity
vector $u^\alpha$. So we apply the projection to $u^\alpha$ and set
$\bar u^\alpha=(\delta^\alpha{}_\beta+n^\alpha n_\beta)u^\beta$.  Then
from the relation of the perfect fluid 
(\ref{eq:eineul:2}) ,
(\ref{eq:The_Initial_Data:2}), and (\ref{eq:The_Initial_Data:8}) we
see that
\begin{eqnarray}
  \label{eq:The_Initial_Data:3}
  z &= &(\epsilon+p)(n_\beta u^\beta)^2-p,\\
  j^\alpha &=&(\epsilon+p)\bar{u}^\alpha(n_\beta u^\beta).
\end{eqnarray}
The vectors $j^\alpha$ and $\bar{u}^\alpha$ are tangent to the initial
surface and so they can be identified with vectors $j^{a}$ and
$\bar{u}^a$ intrinsic to this surface. Recalling the normalization
condition (\ref{eq:publ-broken:2}),
 we have
\begin{math}
  -1=-(n_\beta u^\beta)^2+h_{ab}\bar{u}^a\bar{u}^b
\end{math}. 
Thus the matter data $(z,j^a)$ can be identified with the
initial data for the velocity vector as follow:
\begin{eqnarray}
  \label{eq:The_Initial_Data:4}
  z &= &\epsilon+(\epsilon+p)h_{ab}\bar{u}^a\bar{u}^b,\\
  \label{eq:The_Initial_Data:5} j^\alpha
  &=&(\epsilon+p)\bar{u}^a\sqrt{1+ h_{ab}\bar{u}^a\bar{u}^b}.
\end{eqnarray}

These two types of data cannot be given freely, because the
hypersurface $(M,h)$ is a sub-manifold of $(V,g)$, therefore  the
Gauss Codazzi
equations lead to Einstein constraint equations
\begin{eqnarray}
 \label{eq:publ-broken:23}
  R(h)-K_{ab}K^{ab}+(h^{ab}K_{ab})^2 &=&16\pi z\quad
(\text{Hamiltonian constaint})\\
  \label{eq:publ-broken:38} {}^{(3)}\nabla_b K^{ab}-
  {}^{(3)}\nabla^b(h^{bc}K_{bc}) &=& -8\pi j^a \quad(\text{
momentum constraint}).
\end{eqnarray}
Here $R(h)=h^{ab}R_{ab}$ is the scalar curvature with respect to
the metric $h$.

We turn now to the conformal method which allows us to construct the
solutions of the constraint equations (\ref{eq:publ-broken:23}) and
(\ref{eq:publ-broken:38}).  Before entering into details we have to
discuss the relations between the initial data for the system of
Einstein gravitational fields (\ref{eq:eineul:1})
 and Euler equations (\ref{eq:eineul:3})  
which are given by (\ref{eq:The_Initial_Data:4}) and
(\ref{eq:The_Initial_Data:5}). As it turns out this relations is by no
means trivial, and indeed they will force us to modify the conformal
method.

\subsection{The compatibility problem of the initial data for the
  fluid and the gravitational fields}
\label{sec:comp-probl-init}

On the one hand, the initial data for the Euler equations
are $w(\epsilon)$ and $u^{\alpha}$.  On the other
hand $z=F(w(\epsilon),\bar{u}^{a})$ and
$j^a=H(w(\epsilon),\bar{u}^{a})$, which are given by
(\ref{eq:The_Initial_Data:4}) and (\ref{eq:The_Initial_Data:5})
respectively, appear as sources in the constraint equations
(\ref{eq:publ-broken:23}) and (\ref{eq:publ-broken:38}). 
There we have
the possibility of either to consider $w$ and $u^{\alpha}$ as the
fundamental quantities and construct then $z$ and $j^a$ or, vice
verse, to consider $z$ and $j^a$ as the fundamental quantities and
construct then $w$ and $u^{\alpha}$.

The first possibility does not work because the geometric quantities
which occur on the left hand side of the constraint equations are
supposed to scale with some power of a scalar function $\phi$. So $z$
and $j^a$, which are the sources in the constraint equations, must
also scale with a definite power of $\phi$.  If $\epsilon$ is scaled
with a certain power of $\phi$, then $p$ would be scaled, according to
the equation of state (\ref{eq:eineul:4}), 
to a different power.
Hence, by (\ref{eq:The_Initial_Data:4}) $z$ is a sum of different
powers. Thus, the power which $\epsilon$ and $p$ are scaled would have
to be zero and they would be left unchanged by the rescaling.
Similarly it can be seen that $\bar u^a$ would remain unchanged. So in
fact $z$ would be unchanged and this is inconsistent with the scalding
used in the conformal method.

Instead of constructing $(w,\bar u^a)$ from $(z,j^a)$ it is more
useful to introduce some auxiliary quantities. Beside the Makino
variable $ w=\epsilon^{\frac{\gamma-1}{2} }$, we set
\begin{equation}
  \label{eq:The_Initial_Data:6}  y=z^{\frac{\gamma-1}{2}}
\qquad\mbox{and}\qquad
  v^a=\frac{j^a}{z}.
\end{equation}
Now we consider the following map
\begin{equation}
  \label{eq:The_Initial_Data:7}
  \Phi\left(%
    \begin{array}{c}
      w \\
      \bar u^a \\
    \end{array}%
  \right)=\left(%
    \begin{array}{c}
      {w[1+(1+Kw^2)h_{ab}\bar{u}^a\bar{u}^b)]^{\frac{\gamma -1}{2}}}
\\ 
      \frac{(1+Kw^2)\bar{u}^a\sqrt{1+h_{bc}\bar{u}^b\bar
          u^c}}{1+(1+Kw^2)h_{bc}\bar u^b\bar 
        u^c} \\
    \end{array}%
  \right)=\left(%
    \begin{array}{c}
      y \\
      v^a \\
    \end{array}%
  \right),
\end{equation}
which is equivalent to equations (\ref{eq:The_Initial_Data:4}) and
(\ref{eq:The_Initial_Data:5}).  The initial data $(w,\bar u^a)$ for
the fluid are reconstructed through the inversion of $\Phi$ above.

\begin{thm}[Reconstruction theorem for the initial data]
  \label{thm:The_Initial_Data:1}
  There is a function $s:[0,1)\to \mathbb{R}$ such that the map $\Phi$
  defined by (\ref{eq:The_Initial_Data:7}) is a diffeormophism from
  $[0,(\sqrt{\gamma K})^{-\frac{1}{2}})\times \mathbb{R}^3$ to
  $\Omega$, where
  \begin{equation}
    \label{eq:The_Initial_Data:14} \Omega=\{(y,v^a): 0\leq y<
    s\left(\sqrt{h_{ab}v^a v^b}\right), h_{ab}v^a v^b< 1\}.
  \end{equation}
\end{thm}

\begin{proof}[of theorem \ref{thm:The_Initial_Data:1}]

  Let $\rho=\sqrt{h_{ab}\bar u^a\bar u^b}$, $\bar u_0$ be a unit
  vector and $R_{{\bar u}^a}$ be the rotation with respect to the
  metric $h_{ab}$ such that $\bar u^a=\rho R_{{\bar u}^a}\bar u_0$.
  Then

  \begin{equation}
    \label{eq:The_Initial_Data:9}
    \Phi\left(%
      \begin{array}{c}
        w \\
        \bar u^a \\
      \end{array}%
    \right)=\Phi\left(%
      \begin{array}{c}
        w \\
        \rho R_{{\bar u}^a}\bar u_0 \\
      \end{array}%
    \right)=\left(%
      \begin{array}{c}
        {w[1+(1+Kw^2)\rho^2]^{\frac{\gamma -1}{2}}} \\
        \frac{(1+Kw^2)  R_{{\bar u}^a}\bar
          u_0\rho\sqrt{1+\rho^2}}{1+(1+Kw^2)\rho^2} \\ 
      \end{array}%
    \right).
  \end{equation}

  Therefore, we can first invert the two dimensional map

  \begin{equation} \label{eq:The_Initial_Data:10} \Theta\left(%
      \begin{array}{c}
        w \\
        \rho \\
      \end{array}%
    \right):=\left(%
      \begin{array}{c}
        {w[1+(1+Kw^2)\rho^2]^{\frac{\gamma -1}{2}}} \\
        \frac{(1+Kw^2)  \rho\sqrt{1+\rho^2}}{1+(1+Kw^2)\rho^2} \\
      \end{array}%
    \right)
  \end{equation}
  for $\rho\geq 0 $ and then apply again the rotation. For $w>0$, we
  decompose $\Theta$ of (\ref{eq:The_Initial_Data:10}) as follows:
  \begin{equation}
    \label{eq:The_Initial_Data:15}
    \left(%
      \begin{array}{c}
        w \\
        \rho \\
      \end{array}%
    \right) \mapsto \left(%
      \begin{array}{c}
        \epsilon \\
        \rho \\
      \end{array}%
    \right)\mapsto \left(%
      \begin{array}{c}
        \epsilon+(\epsilon+p(\epsilon))\rho^2 \\
        (\epsilon+p(\epsilon))\rho\sqrt{1+\rho^2} \\
      \end{array}%
    \right)=:\left(%
      \begin{array}{c}
        z \\
        r \\
      \end{array}%
    \right) \mapsto \left(%
      \begin{array}{c}
        z^{\frac{\gamma-1}{2}} \\
        \frac{r}{z} \\
      \end{array}%
    \right).
  \end{equation}
  In order to show that this is a one to one map, we need to show that
  the Jacobian of
  $G(\epsilon,\rho):=(\epsilon+(\epsilon+p(\epsilon))\rho^2,
  (\epsilon+p(\epsilon))\rho\sqrt{1+\rho^2})$ does not vanish. This
  computation results with
  \begin{equation}
    \label{eq:The_Initial_Data:11}
    \det\left(%
      \begin{array}{cc}
        1+(1+p')\rho^2 &  (1+p')\rho\sqrt{1+\rho^2}\\
        (\epsilon+p)2\rho &
(\epsilon+p)\frac{1+2\rho^2}{\sqrt{1+\rho^2}} \\
      \end{array}%
\right)=\frac{(\epsilon+p)}{\sqrt{1+\rho^2}}
\left(1+\rho^2(1-p')\right).
  \end{equation}
  Recall that 
  is the speed of sound is given by (\ref{eq:publ-broken:14}),
therefore the causality condition
  $\sigma^2<c^2=1$ imposes the below restriction of the domain of
  definition of the map $\Theta$:
  \begin{equation}
    \label{eq:The_Initial_Data:12} \sigma^2=p'=\frac{\partial
      p}{\partial \epsilon}=\frac{\partial }{\partial \epsilon}\left(
      K\epsilon^\gamma\right)=\gamma K\epsilon^{\gamma-1}=\gamma K
    w^2<1.
  \end{equation}
  Let $S$ be the strip $\{0\leq w<(\sqrt{\gamma K})^{-\frac 1 2},
  0\leq \rho<\infty\}$. We now want to show that
  $\Theta:S\to\Theta(S)$ is a bijection.  Clearly,
  $\Theta(0,\rho)=(0,\frac{\rho}{\sqrt{1+\rho^2}})$ maps $\{0\}\times
  [0,\infty)$ to $\{0\}\times [0,1)$ in a one to one manner, and
  $\Theta(w,0)=(w,0)$ is of course a bijection. The line
  $(\sqrt{\gamma K})^{-\frac 1 2},\rho)$ is mapped to the curve
  \begin{equation}
    \label{eq:The_Initial_Data:16}
    \left(%
      \begin{array}{c}
        y(\rho) \\
        x(\rho) \\
      \end{array}%
    \right)=\left(%
      \begin{array}{c}
        (\sqrt{\gamma K})^{-\frac 1
2}\left(1+2\rho^2\right)^{\frac{\gamma-1}{2}} \\
        \frac{2\rho\sqrt{1+\rho^2}}{1+2\rho^2} \\
      \end{array}%
    \right).
  \end{equation}
  Since $\frac{d x}{d \rho}>0$, there exists a function $s:[0,1)\to
  \mathbb{R}$ such that the curve (\ref{eq:The_Initial_Data:16}) is
  given by the graph of $s$ and the image of $\Theta$ is the set below
  this graph, that is,
  \begin{equation}
    \label{eq:The_Initial_Data:13} \Theta(S)=\{(y,x): y<s(x), 0\leq
    x<1\}.
  \end{equation}
  By (\ref{eq:The_Initial_Data:15}), (\ref{eq:The_Initial_Data:11})
  and (\ref{eq:The_Initial_Data:12}) we conclude that the Jacobian of
  the map $\Theta$ does not vanish in the interior of $S$, hence
  $\Theta:S\to \Theta(S)$ is locally one to one map. It is well known
  that a locally one to one map between two simply connected sets is a
  bijective map. \BeweisEnde
\end{proof}

\subsection{Cantor's conformal method for solving the constraint
  equations }
\label{sec:conformal_method}
In principle there are two possibilities for solving the constraint
equation for an asymptotically flat manifold:
\begin{itemize}
  \item Either to adapt directly the method of York et all, but then
  one is forced to impose certain relations between $R(\bar h)$ and
  the second fundamental form (see Choquet-Bruhat and York
  \cite{CHY
  } for details).
  \item These undesirable conditions can be substituted by a method
  developed by Cantor and we will describe it in the following. (This
  method has been discussed in detail in the literature, see for
  example \cite{bartnik04}, 
  \cite{CHY}, 
  \cite{cantor79} 
  \cite{OMC} 
  and reference therein.)
\end{itemize}

In this method parts of the data are chosen (the so-called free data),
and the remaining parts are determined by the constraint equations
(\ref{eq:publ-broken:23}) and (\ref{eq:publ-broken:38}). 
The free initial data are
\begin{math}
  \left( \bar h_{ab}, \bar A_{ab},\bar z, \bar j \right)
\end{math}, where $A_{ab}$ is a divergence and trace free $2$-tensor.
The main idea is to consider two conformal scaling functions, $\alpha$
and $\phi$.

\begin{enumerate}
  \item We start with $\hat h_{ab}=\alpha^4\bar h_{ab}$. If $\alpha$
  is a positive solution to (\ref{eq:conformal_method:3}), then
  $R(\hat h)=0$. The Brill-Cantor condition (see Definition
  \ref{def:Brill-Cantor}) is necessary and sufficient for the
  existence of positive solutions.
  Having solved equation (\ref{eq:conformal_method:3}), we now adjust
  the given data to the new metric: $\hat A^{ab} =\alpha^{-10}\bar
  A^{ab}$, $\hat z=\alpha^{-8}\bar z$ and $\hat j^a=\alpha^{-10}\bar
  j^a$.

  \item The second step here is solve the Lichnerowicz Laplacian
  (\ref{eq:conformal method:8})
  and set
  \begin{equation}
    \hat K^{ab}=\left(\mathcal{L}(W)\right)^{ab}+\alpha^{-10}A^{ab},
  \end{equation}
  where \begin{math} \left(\mathcal{L}(W)\right)^{ab}\end{math} is the
  Killing operator giving by (\ref{eq:conformal method:6}).
  \item The third step is: If $\phi$ is a solution to the Lichnerowicz
  equation (\ref{eq:conformal method:9}), then it follows from
  (\ref{eq:conformal method:11}) that the data $h_{ab}=\phi^4\hat
  h_{ab}$, $K^{ab} =\phi^{-10}\hat K^{ab}$, $z=\phi^{-8}\hat z$ and
  $j^a=\phi^{-10}\hat j^a$ satisfy the constraint equations
  (\ref{eq:publ-broken:23}) and (\ref{eq:publ-broken:38}).
\end{enumerate}

For the Einstein-Euler system with the equation of state
(\ref{eq:eineul:4}) it is essential that the initial data will satisfy
condition (\ref{eq:The_Initial_Data:14}) of Theorem
\ref{thm:The_Initial_Data:1}. Therefore in this case it is necessary
to adjust this method.

Here the free initial data are:
\begin{equation}
  \label{eq:publ-broken:20} \left( \bar h_{ab}, \bar A_{ab},\hat y,
    \hat v^b \right).
\end{equation}
where $\bar A_{ab} $ is trace and divergence free, that is,
\begin{math} \bar D_a \bar A^{ab}=\bar h_{ab}\bar A^{ab}=0\end{math},
where $\bar D_a$ is the covariant derivative with respect to the
metric $\bar h_{ab}$.  We require that the matter data $(\hat y,\hat
v^a)$, will satisfy the condition
\begin{equation}
  \label{eq:conformal_method:4}  0\leq \hat y < s\left(\sqrt{ \hat
      h_{ab}\hat v^a\hat v^b}\right),
\end{equation}
where $s(\cdot)$ is given by (\ref{eq:The_Initial_Data:13}). The
remaining initial data are determined by the constraint equations
(\ref{eq:publ-broken:23}) and (\ref{eq:publ-broken:38}), relations
(\ref{eq:The_Initial_Data:6}) and Theorem
\ref{thm:The_Initial_Data:1}.

\begin{rem}
  The distinction between the gravitational data $(\bar h_{ab},\bar
  A_{ab})$ and the matter data $(\hat z,\hat j^b)$ is caused by
  condition (\ref{eq:The_Initial_Data:14}). For if we make the scaling
  $\hat h_{ab}=\phi^4\bar h_{ab}$, $\hat z=\phi^{-8}\bar z$, and $\hat
  j^b=\phi^{-10}\bar j^b$, then $\hat v^b=\phi^{-2}\bar v^b $, $\hat
  y=\phi^{-4(\gamma-1)}\bar y$ and \begin{math} \hat h_{ab}\hat
    v^a\hat v^b= \bar h_{ab}\bar v^a\bar v^b\end{math}.  Thus under
  this conformal transformation, the argument of $s$ in
  (\ref{eq:The_Initial_Data:14}) is invariant, while the left hand
  side will be affected. Therefore the free initial data are partially
  invariant under conformal transformations.
\end{rem}

Now, if we perform the conformal transformation
\begin{equation}
  \label{eq:conformal method:1}  \hat h_{ab}=\alpha^4\bar h_{ab},
\end{equation}
then the scalar curvature with respect to the metric $\hat h_{ab}$,
$R(\hat h)$, satisfies
\begin{equation}
  \label{eq:conformal method:2}  -8\Delta_{\bar h}\alpha +R(\bar
  h)\alpha =R(\hat h)\alpha^5.
\end{equation}
Therefore, if there exists a nonnegative solution to the equation
\begin{equation}
  \label{eq:conformal_method:3}  -\Delta_{\bar h}\alpha
  +\frac{1}{8}R(\bar h)\alpha =0,
\end{equation}
then the metric $\hat h_{ab}$ given by (\ref{eq:conformal method:1})
will have zero scalar curvature. We  continue the construction as
follow.  Let $\hat A^{ab}=\alpha^{-10}\bar A^{ab}$, $\hat D_a$ denotes
the covariant derivative with respect to the metric $\hat h_{ab}$,
since \begin{math} \hat D_a \hat A^{ab}=\alpha^{-10}\bar D_a \bar
  A^{ab}\end{math}, $\hat A^{ab}$ is a divergence and trace free $2$
tensor.

Assume $\hat K$ is a symmetric covariant $2$-tensor which satisfies
the maximal slice condition, that is $\hat h_{ab}\hat K^{ab}=0$. Then
we split $\hat K$ by writing it for some vector $W$:
\begin{equation}
  \label{eq:conformal method:5} \hat K^{ab} =\hat A^{ab} +\hat{
    \mathcal{L}}^{ab}(W),
\end{equation}
where $\hat{\mathcal{L}}$ is the Killing field operator
\begin{equation}
  \label{eq:conformal method:6} \left(\hat{
      \mathcal{L}}(W)\right)^{ab}=\left(\hat{\pounds}_W\hat
    h\right)^{ab}-\frac{1}{3}\hat h^{ab}\Tr\hat{\pounds}_W\hat h=\hat
  D_a W^b+\hat D_b W^a-\frac{1}{3}\hat h^{ab}\Tr\hat{\pounds}_W\hat
  h,
\end{equation}
and $\hat{\pounds}_W\hat h$ is the Lie derivative. The momentum
constraint (\ref{eq:publ-broken:38}) is now equivalent  to
\begin{equation}
  \label{eq:conformal method:7} \hat D_a\hat K^{ab} =\hat
  D_a\left(\hat{ \mathcal{L}}(W)\right)^{ab}=-8\pi\hat j^b,
\end{equation}
that is, $W$ is a solution to the Lichnerowicz Laplacian system
\begin{equation}
  \label{eq:conformal method:8} \left(\Delta_{L_{\hat
        h}}W\right)^b:=\hat D_a\left(\hat{
\mathcal{L}}(W)\right)^{ab}=
  \Delta_{\hat h}W+\frac{1}{3}\hat D^b\left(\hat D_a W^a\right)+
  \hat R_a^b W^a=-8\pi\hat j^b,
\end{equation}
here $\hat R_a^b$ is the Ricci curvature tensor with respect to the
metric $\hat h_{ab}$.

Having solved the Lichnerowicz Laplacian (\ref{eq:conformal
method:8}), we consider the Lichnerowicz equation
\begin{equation}
  \label{eq:conformal method:9} -\Delta_{\hat h}\phi=2\pi \hat z
  \phi^{-3}+\frac{1}{8}\hat K_a^b\hat K_b^a \phi^{-7}.
\end{equation}
Now we put $h_{ab}=\phi^4\hat h_{ab}$, $K_{ab}=\phi^{-2}\hat K_{ab}$,
$z=\phi^{-8}\hat z$ and $j^b=\phi^{-10}\hat j^b$. Since
\begin{equation}
  \label{eq:conformal method:11} -\Delta_{\hat
    h}\phi=\phi^5\frac{1}{8} R(h)=\phi^5\left(2\pi z
    +\frac{1}{8}K_a^bK_b^a\right)=\phi^5\left(2\pi \hat z \phi^{-8}
    +\frac{1}{8}\hat K_a^b\hat K_b^a\phi^{-12}\right),
\end{equation}
the Hamiltonian constraint (\ref{eq:publ-broken:23}) is  valid,
and since \begin{equation}
  \label{eq:conformal method:10} D_a K^{ab}=\phi^{-10}\hat D_a \hat
  K^{ab}=-8\pi\phi^{-10}\hat j^b=-8\pi j^b
\end{equation}
 the data  $(h_{ab},K_{ab},z,j^b)$ satisfy the constraint
equations
(\ref{eq:publ-broken:23}) and (\ref{eq:publ-broken:38}).  
In order that the matter variables and $(z,j)$ satisfy the
compatibility
conditions (\ref{eq:The_Initial_Data:4}) and
(\ref{eq:The_Initial_Data:5}) it is necessary to check that
$y=z^{\frac{\gamma-1}{2}}=\left(\phi^{-8}\hat
  z\right)^{\frac{\gamma-1}{2}}=\phi^{-4(\gamma-1)}\hat y$ and
$v^b=\frac{j^b}{z}=\phi^{-2}\hat v^b$ satisfy condition
(\ref{eq:The_Initial_Data:14}). Indeed,
\begin{equation}
  \label{eq:conformal method:12} 0\leq y< s\left(\sqrt{h_{ab}v^a
      v^b}\right)\Leftrightarrow 0\leq \phi^{-4(\gamma-1)}\hat y<
  s\left(\sqrt{\hat h_{ab}\hat v^a \hat v^b}\right),
\end{equation}
but since $\phi \geq 1$, $\phi^{-4(\gamma-1)}\hat y\leq\hat y$ and
thus assumption (\ref{eq:conformal_method:4}) assures condition
(\ref{eq:The_Initial_Data:14}).

\begin{thm}[Construction of the gravitational data]
  \label{thm:conformal_method:1}
  Given the free data $(\bar h_{ab}, \bar A_{ab},\hat y,\hat v^b)$
  such that $(\bar h_{ab}-I)\in H_{s,\delta}$, $ \bar A_{ab}\in
  H_{s-1,\delta+1}$, $(\hat y,\hat v^b)\in H_{s-1,\delta+2}$,
  $\frac{5}{2}<s<\frac{2}{\gamma-1}+\frac{3}{2}$ and
  $-\frac{3}{2}<\delta<-\frac{1}{2}$.  Then the gravitational data:
  \begin{equation*}
    h_{ab}=(\phi\alpha)^4\bar h_{ab}
    \qquad \text{ and}\qquad K_{ab}=(\phi\alpha)^{-2}\bar
    A_{ab}+\phi^{-2} \hat{\mathcal{L}}(W)
  \end{equation*}
  satisfy the constraint equations (\ref{eq:publ-broken:23}) and
  (\ref{eq:publ-broken:38}) with $z=\phi^{-8}\hat
y^{\frac{2}{\gamma-1}}$ and
  $j^b=\phi^{-10}\hat y^{\frac{2}{\gamma-1}}\hat v^b$ as the right
hand side.  In addition,
  $(h_{ab}-I)\in H_{s,\delta} $ and $K_{ab}\in H_{s-1,\delta+1}$ and
  therefore if $\frac{7}{2}<s<\frac{2}{\gamma-1}+\frac{3}{2}$, then
  these data have the needed regularity so they can serve as initial
  data for the evolution equations
  of the reduced Einstein equations (\ref{eq:publ-broken:19}).
\end{thm}

\begin{proof}[of Theorem \ref{thm:conformal_method:1}]
  \hfill
  \begin{itemize}

    \item The free data are $(\bar h_{ab}, \bar A_{ab},\hat y,\hat
    v^b)$, where $(\bar h_{ab}-I)\in H_{s,\delta}$, $\bar A_{ab}\in
    H_{s-1,\delta+1}$ a divergence a trace free $2$-tensor and $(\hat
    y,\hat v^b)\in H_{s-1,\delta+2}$.

    \item The function $\alpha$ satisfies equation
    (\ref{eq:conformal_method:3}), so by Theorem
    \ref{thm:Solutions_to_the_elliptic_systems:1} $\alpha>0$ and
    $(\alpha-1)\in H_{s,\delta}$ provided that $s\geq 2$ and $\delta >
    -\frac{3}{2}$. Since $\alpha$ is continuous and $\lim_{|x|\to
      \infty}\alpha(x)-1=0$, there is a compact set $D$ of
    $\mathbb{R}^3$ such that $\alpha(x)\geq\frac{1}{2}$ for $x\not \in
    D$ and $\min_{D}\alpha(x)\geq t_0>0$.

    \item The function $F(t):=\frac{1-t}{t}$ has bounded derivatives
    in $[\min\{t_0,\frac{1}{2}\},\infty)$, so by Moser type estimate,
    Theorem \ref{thr:Appendix:5},  
    ${\alpha}^{-1}-1=\frac{1-\alpha}{\alpha}\in H_{s,\delta}$.

    \item Now, by algebra (Proposition \ref{prop:Properties:3}),
    $(\hat h_{ab}-I)=(\alpha^4\bar h_{ab}-I)\in H_{s,\delta}$ and
    $\hat A^{ab}=\alpha^{-10} A^{ab}\in H_{s-1,\delta+1}$.

    \item The matter variables $(\hat z,\hat j^b)$ are given by $\hat
    z=\hat y^{\frac{2}{\gamma-1}}$, $\hat j^b=\hat z\hat v^b$. 
Therefore Proposition \ref{prop:properties:5a} implies that $\hat z\in
H_{s-1,\delta+2}$,
    provided that $\frac{3}{2}<s-1<\frac{2}{\gamma-1}+\frac{1}{2}$ and
    also $j^b\in H_{s-1,\delta+2}$ by the Proposition
\ref{prop:Properties:3}.

    \item The vector $W$ is a solution of the Lichnerowicz Laplacian
    (\ref{eq:conformal method:8}), thus according to Theorem
    \ref{thm:Solutions_to_the_elliptic_systems:2} below, $W\in
    H_{s,\delta}$ if $s\geq 2$. Hence $\hat K^{ab}$ given in
    (\ref{eq:conformal method:5}) belongs to $H_{s-1,\delta+1}$.
    Again, by Proposition \ref{prop:Properties:3},
    $\hat K_a^b\hat
    K_b^a\in H_{s-2,\delta+2}$ if $s\geq2$ and 
    $\delta\geq-\frac{3}{2}$.
    \item Setting $u=\phi-1$, then Lichnerowicz equation
    (\ref{eq:conformal method:11}) becomes
    \begin{equation}
  \label{eq:conformal_method:13} -\Delta_{\hat h}u=2\pi \hat z  
      (u+1)^{-3} +\frac{1}{8}\hat K_a^b\hat K_b^a(u+1)^{-7}.
    \end{equation}

    \item So applying Theorem \ref{thm:Semi_linear_elliptic:1}
    with $s'=s$ and $\delta'=\delta$ results that $(\phi-1)=u\in
    H_{s,\delta} $ and $(\phi-1)=u\geq 0$.
  \end{itemize}
  \BeweisEnde
\end{proof}

Combining our results of Section \ref{sec:comp-probl-init} with
Theorem \ref{thm:conformal_method:1} we obtain the following
corollary:
\begin{cor}[Construction of the data for the fluid]
  \label{cor:conformal_method:2}
  Given the free data $(\bar h_{ab}, \bar A_{ab},\hat y,\hat v^b)$
  such that $(\bar h_{ab}-I)\in H_{s,\delta}$, $ \bar A_{ab}\in
  H_{s-1,\delta+1}$, $(\hat y,\hat v^b)\in H_{s-1,\delta+2}$,
  $\frac{5}{2}<s<\frac{2}{\gamma-1}+\frac{3}{2}$,
  $-\frac{3}{2}<\delta<-\frac{1}{2}$ and $(\hat y,\hat v^a)\in
  \Omega$, where $\Omega$ is given by (\ref{eq:The_Initial_Data:14}).
  Then the data of the four velocity vector $u^\alpha$ and the Makino
  variable $w$ are: $y=\phi^{-4(\gamma-1)}\hat y$, $v^b=\phi^{-2}\hat
  v^b$,
  \begin{equation*}
    (w,\bar u^a):=\Phi^{-1}(y,v^a)
    \qquad \text{ and}\qquad \bar u^0=1+h_{ab}\bar u^a\bar u^b
  \end{equation*}
  and the data for the energy and momentum densities are:
  $z=y^{\frac{2}{\gamma-1}}$, $j^a=zv^a$. These data satisfy the
  compatibility conditions (\ref{eq:The_Initial_Data:4}) and
  (\ref{eq:The_Initial_Data:5}).  In addition, by Moser type estimate
  Theorem \ref{thr:Appendix:5} and Proposition
  \ref{prop:Properties:3}, 
  $(w,\bar u^a)\in H_{s-1,\delta+2} $ and
  $\bar u^0-1\in H_{s-1,\delta+2}$ and therefore if
  $\frac{7}{2}<s<\frac{2}{\gamma-1}+\frac{3}{2}$, then these data have
  the needed regularity so they can serve as initial data for 
  the hyperbolic system (\ref{eq:Initial:2}) for the perfect fluid.   
\end{cor}

\subsection{Solutions to the elliptic systems }
\label{sec:Solutions_to_the_elliptic_systems}
This section is devoted to the solutions the linear elliptic systems
(\ref{eq:conformal_method:3}) and (\ref{eq:conformal method:8}). The
assumption on the given metric $\bar h_{ab}$ is that $(\bar
h_{ab}-I)\in H_{s,\delta}$. So according to Theorem
\ref{thm:eq:a_priori_estimates_weighted:1} of Section
\ref{sec:elliptic}, the operator
\begin{math}
  \Delta_{\bar h}:H_{s,\delta}\to H_{s-2,\delta+2}
\end{math}
is semi Fredholm. In fact, it is an isomorphism, this can be shown in
a similar manner to Step 1 of Section
\ref{subsection:Semi_linear_elliptic}. We now consider the operator
\begin{equation}
  \label{eq:Solutions_to-the_elliptic_systems:2}L:=-\Delta_{\bar
    h}+\frac{1}{8}R(\bar h):H_{s,\delta}\to H_{s-2,\delta+2},
\end{equation}
which  is also a  semi Fredholm operator. If $R(\bar h)\geq 0$, then $L$ is
injective.  However, a weaker condition is that $L$ does not have
non-positive eigenvalues, the variational formulation of this property
known as the {\it Brill-Cantor condition} \cite{cantor81:_laplac}.

Let us first introduce few notations. For a Riemannian metric
$\bar h_{ab}$, we set $(Du,Dv)_{\bar h}=\bar h^{ab}\partial_a
u\partial_b v$, $|Du|^2_{\bar h}=(Du,Du)_{\bar h}$ and $\mu_{\bar h}$
is the volume element with respect to the metric $\bar h_{ab}$.
\begin{defn}[Brill-Cantor condition]
  \label{def:Brill-Cantor} A metric $\bar h_{ab}$ satisfies the
  Brill-Cantor condition if
  \begin{equation}
    \label{eq:Solutions_to-the_elliptic_systems:3}
    \inf_{u\not=0}\frac{\int\left(|Du|^2_{\bar h}+\frac{1}{8}R(\bar
        h)u^2\right)d\mu_{\bar h}}{\int u^2 d \mu_{\bar h}}>0,
  \end{equation}
  where the infimum is taken over all $u\in C_0^1(\mathbb{R}^3)$.
\end{defn}

This condition is invariant under conformal transformations, a fact
that has been proved for example in
\cite{y.00:_einst_euclid
}

\begin{thm}[Construction of a metric having zero scalar curvature]
  \label{thm:Solutions_to_the_elliptic_systems:1}
  Assume the given metric $\bar h_{ab}$ satisfies $(\bar
  h_{ab}-\delta_{ab})\in H_{s,\delta}$, $s\geq 2$,
  $\delta>-\frac{3}{2}$ and $\bar h_{ab}$ satisfies the Brill-Cantor
  condition (\ref{eq:Solutions_to-the_elliptic_systems:3}). Then there
  exists a scalar function $\alpha$ such that $\alpha-1\in
  H_{s,\delta}$, $\alpha(x)>0$ and the metric $\hat
  h_{ab}=\alpha^4\bar h_{ab}$ has a scalar curvature zero.
\end{thm}

\begin{proof}[of Theorem
\ref{thm:Solutions_to_the_elliptic_systems:1}]
  The desired $\alpha$ is a solution to the elliptic equation
  (\ref{eq:conformal_method:3}). By setting $u=\alpha+1$ this equation
  goes to
  \begin{equation}
    \label{eq:Solutions_to_the_elliptic_systems:4} Lu=  -\Delta_{\bar
      h}u +\frac{1}{8}R(\bar h)u =-\frac{1}{8}R(\bar h).
  \end{equation}
  We define for $\tau\in [0,1]$, $L_\tau u=-\Delta_{\bar{h}}u+\frac
  {\tau}{8} R(\bar{h})u$.  If $L_\tau u=0$, then by Lemma
  \ref{lem:a_priori_estimates_weighted:5}, $u\in H_{s,-1}$ so
  \begin{equation}
    \label{eq:Solutions_to_the_elliptic_systems:5}
    0=(u,L_\tau u) =\int\left(|Du|_{\bar{h}}^2+
      \frac {\tau}{8} R(\bar{h})u^2\right)d\mu_{\bar{h}}.
  \end{equation}

  Now, if $ \int R(\bar{h})u^2d\mu_{\bar{h}}\geq 0$, then obviously
  (\ref{eq:Solutions_to_the_elliptic_systems:5}) implies that $u\equiv
  0$. Otherwise $ \int R(\bar{h})u^2d\mu_{\bar{h}}< 0$, then there is
  sequence $\{u_n\}\subset C_0^\infty$ such that $u_n\to u$ in
  $H_{s,-1}$ - norm and
  \begin{equation}
    \label{eq:Solutions_to_the_elliptic_systems:6}
    \int\left(|Du|_{\bar{h}}^2+
      \frac 1 8     
R(\bar{h})u^2\right)d\mu_{\bar{h}}=\lim_n\int\left(|Du_n|_{\bar{h}}
^2+ 
      \frac 1 8 R(\bar{h})u_n^2\right)d\mu_{\bar{h}}>0
  \end{equation}
  by the Brill-Cantor condition
  (\ref{eq:Solutions_to-the_elliptic_systems:3}).  Substitution of 
  (\ref{eq:Solutions_to_the_elliptic_systems:6}) in
  (\ref{eq:Solutions_to_the_elliptic_systems:5}) yields
  \begin{equation}
    \label{eq:Solutions_to_the_elliptic_systems:7}
    0 =\int\left(|Du|_{\bar{h}}^2+
      \frac 1 8 R(\bar{h})u^2\right)d\mu_{\hat{h}}
    +\frac {(\tau-1)} 8\int R(\bar{h})u^2d\mu_{\bar{h}},
  \end{equation}
  this is certainly a contradiction. Thus $L_\tau$ is injective for
  each $\tau\in [0,1]$, $L_0=-\Delta_{\bar{h}}$ is isomorphism, hence
  $L_1=-\Delta_{\bar{h}}+\frac{1}{8} R(\bar{h})$ is isomorphism by
  Theorem \ref{thm:eq:a_priori_estimates_weighted:2}.
  Having proved the existence, we now show that $\alpha=u+1$ is
  nonnegative.  The set $\{x:\alpha(x)<0\}$ has compact support since
  $\lim_{x\to \infty}u(x)=0$ by the embedding Theorem
  \ref{thr:Appendix:2}.  So letting $w=-\min(\alpha,0)$, we have $w\in
  H_0^1(\setR^3)$ and if the set $\{x:\alpha(x)<0\}$ is not empty,
  then $w\not\equiv 0$ and then the Brill-Cantor condition gives
  \begin{equation}
    \label{eq:Solutions_to_the_elliptic_systems:8}
    \int_{\{\alpha<0\}}\left(|Dw|^2_{\bar{h}}+ \frac 1 8 R(\bar{h})
      w^2\right)d\mu_{\bar{h}}>0.
  \end{equation}
  On the other hand, according to Definition
  \ref{def:Semi_linear_elliptic:1} of weak solutions,
  \begin{equation}
    \label{eq:Solutions_to_the_elliptic_systems:9}
    0 =\int\left((D\alpha,Dw)_{\bar{h}}+ \frac 1 8 R(\bar{h})\alpha
      w\right)d\mu_{\bar{h}}
    =-\int_{\{\alpha<0\}}\left(|Dw|^2_{\bar{h}}+ \frac 1 8 R(\bar{h})
      w^2\right)d\mu_{\bar{h}}.
  \end{equation}
  So we conclude that $\alpha\geq 0$. Since $\alpha\geq0$, we have by
  Harnack's inequality
  \begin{equation*}
    \sup_{B_r}\alpha\leq C \inf_{B_r}\alpha
  \end{equation*}
  provided that $B_r$ is sufficiently small ball. Hence, the set
  $\{\alpha(x)=0\}$ is both open and closed, which is impossible.
  Thus $\alpha(x)>0$.
  \BeweisEnde
\end{proof}

\begin{rem}
  The conditions for applying Harnack's inequality to a second order
  elliptic operator
  \begin{equation*}
    Lu=\partial_a\left(A_{ab}(x)\partial_b u\right) + C(x)u
  \end{equation*}
  are boundedness of the coefficients (see
  e.~g.~\cite{gilbarg83:_ellip_partial_differ_equat_secon_order};
  Section 8) However, following carefully the proofs we found it can
  be applied also when the zero order coefficient belongs to $L^q_{\rm
    loc}(\setR^3)$ with $q>\frac 3 2$. In local coordinates equation
  (\ref{eq:conformal_method:3}) takes the form
  \begin{equation*}
L\alpha=\partial_a\left(\sqrt{|\bar{h}|}\bar{h}^{ab}\partial_b\alpha
    \right) +\sqrt{|\bar{h}|}R(\bar{h})\alpha=0.
  \end{equation*}
  For $s\geq 2$, $\sqrt{|\bar{h}|}\bar{h}^{ab}$ are bounded and
  non-degenerate, while $\sqrt{|\bar{h}|}R(\bar{h})\in L^2_{{\rm
      loc}}(\setR^3)$.
\end{rem}

We turn now the Lichnerowicz Laplacian system (\ref{eq:conformal
  method:8}).
\begin{thm}[Solution of Lichnerowicz Laplacian]
  \label{thm:Solutions_to_the_elliptic_systems:2}
  Let $\hat{h}_{ab}$ be a Riemannian metric in $\setR^3$ so that
  $(\hat{h}-I)\in H_{s,\delta}$. Let vector $\hat{j}^b\in
  H_{s-2,\delta+2}$, $s\geq 2$ and $\delta>-\frac 3 2$. Then equation
  (\ref{eq:conformal method:8}) has a unique solution $W\in
  H_{s,\delta}$.
\end{thm}

\begin{proof}[of theorem
\ref{thm:Solutions_to_the_elliptic_systems:2}]
  In order to verify condition (H1) of Section
  \ref{subsec:a_priori_estimates_weighted} we compute the principle
  symbol of $L_{\Delta_{\hat h}}$ in (\ref{eq:conformal method:8}).
  For each $\xi\in T_x^*M$, the principle symbol $
\left(\Delta_{L_{\hat h}}(\xi)\right)_a^b$ is a linear map
  from $ E_x$ to $ F_x$, where $E_x$ and $F_x$ are a fibers in $T_xM$.
  In local coordinates $\Delta_{\hat h}=\hat
  h^{ab}\partial_a\partial_b +\text{lower terms}$ and
  $D_a=\partial_a+\Gamma(\hat h^{ab},\partial \hat h_{ab})$, hence
  \begin{equation}
    \label{eq:Solutions_to_the_elliptic_systems:13}
    \left(\Delta_{L_{\hat h}}(\xi)\right)_a^b=|\xi|^2_{\hat
      h}\delta_a^b+\frac{1}{3}\xi^b\xi_a.
  \end{equation}
 Therefore
  \begin{equation}
    \label{eq:Solutions_to_the_elliptic_systems:14}\left
      (\left(\Delta_{L_{\hat h}}(\xi)\right)\eta,\eta\right)_{\hat
      h}=\hat h^{bc}\left(L_{\Delta_{\hat
          h}}(\xi)\right)_a^b\eta^a\eta^c=|\xi|^2_{\hat h}|\eta|_{\hat
      h}^2+\frac{1}{3}(\xi_a\eta^a)^2\geq |\xi|^2_{\hat h}|\eta|_{\hat
      h}^2.
  \end{equation}
  Thus $\left(\Delta_{L_{\hat h}}(\xi)\right)_a^b$ has positive
  eigenvalues and therefore $L_{\Delta_{\hat h}}$ is strongly elliptic.
  Furthermore, by Proposition \ref{prop:Properties:3}
  and Remark \ref{rem:Properties:2} we have that if $(\hat
h_{ab}-I)\in
  H_{s,\delta}$, $s\geq 2$ and $\delta>-\frac 3 2$, then 
  \begin{equation*} 
    \Delta_{L_{\hat
        h}}:H_{s,\delta}\to H_{s-2,\delta+2}.
  \end{equation*}
  Hence, we may apply Theorem
  \ref{thm:eq:a_priori_estimates_weighted:2} in order to obtain
  existence of the elliptic system (\ref{eq:conformal method:8}). For
  the given metric $\hat{h}_{ab}$ we define one parameter family of
  metrics $h_t=(1-t)I+t\hat{h}$, $0\leq t\leq 1$, and the following
  associated operators with respect to these metrics: $(D_a)_t$ the
  covariant derivative, $\mathcal{L}_t$ the Killing operator and
  $L_t=\Delta_{L_{h_t}}=(D)_t\cdot \mathcal{L}_t$ the Lichnerowicz
  Laplacian. We want to show that $L_t$ is injective. We recall that
  $-2 \mathcal{L}_t$ is the formal adjoint of $D_t$ (see
  e.~g.~\cite{BES}), in addition, if $L_t(W)=0$, then by Lemma
  \ref{lem:a_priori_estimates_weighted:5} implies $W\in H_{s,-1}$.
  Thus we may use integration by parts and get
  \begin{equation}
    \label{eq:Solutions_to_the_elliptic_systems:15}
    \begin{split} 0=\left(W,L_t W\right)_{ h_t} &=
      \int\left( {h_t}\right)_{ab} W^a L_t( W)^b d\mu_{{h_t}}
      =
      \int\left( {h_t}\right)_{ab}W^a{(D_{c})_t}\cdot
\left(\mathcal{L}_{t}
        W\right)^{cb} d\mu_{{h_t}}\\ &= -2 \int
      \left( {h_t}\right)_{ab}\left(
{h_t}\right)_{dc}(\mathcal{L}_{t}W)^{ad}
      (\mathcal{L}_{t}W)^{cb} \mu_{{h_t}} =
      -2\int|\mathcal{L}_{t}W)|_{{h_t}}^2\mu_{{h_t}}
    \end{split}
  \end{equation}

  Now, if let $\widetilde{h}=|{h_t}|^{-\frac 1 3}{h_t}$, then
  \begin{equation}
    \label{eq:elliptic part:87}
    \pounds_W{\widetilde{h}}=|{h_t}|^{-\frac 1
      3}\left(\pounds_W{{h_t}}-{h_t}\frac 2 3 ({D_a)_t}W^a \right)=
    |{h_t}|^{-\frac 1 3}\mathcal{{L}}_{t}(W).
  \end{equation}

  Thus $L_t(W)=\Delta_{L_{h_t}}(W)=0$ implies $W\equiv 0$ if and only
  if there are no non-trivial Killing vector fields $W$ in $H_{s,-1}$.
  This fact has been proved by G. Choquet and Y.Choquet-Bruhat
  \cite{choquet78:_sur} for $s>\frac{7}{2}$,
  D. Christodoulou and N. O'Murchadha for $s>3+\frac{3}{2}$
  \cite{OMC},
  and Bartnik for $s\geq 2$ \cite{bartnik05:_phase_einst}
  (See also Maxwell \cite{maxwell06:_rough_einst}, where he obtained
  the minimum regularity $s>\frac{3}{2}$ ). Now $L_0=\Delta_{L_{I}}$
  is an operator with constant coefficients, so by Lemma
  \ref{lem:a_priori_estimates_weighted:3} is an isomorphism.
  \BeweisEnde
\end{proof}



\section{Local Existence for Hyperbolic Equations}
\label{sec:local-exist-hyperb}

In this section we prove an existence theorem (locally in time) for
quasi linear symmetric hyperbolic system in the $H_{s,\delta}$ spaces.
The known existence results in the $H^s$ space of Fisher and Marsden
\cite{FMA} and Kato \cite{KATO} (see also \cite{Taylor91},
\cite{taylor97c}), cannot be applied to the $H_{s,\delta}$ spaces.
The main difficulty here is the establishment of energy estimates for
linear hyperbolic systems. In order to achieve it we have defined a
specific inner-product in $H_{s,\delta}$ (see Definition
\ref{def:inner_product}) and in addition the Kato-Ponce commutator
estimate \cite{kato88:_commut_euler_navier_stokes}, \cite{Taylor91},
has an essential role in our approach. Once the energy estimates have
been established in the $H_{s,\delta}$ space, we follow Majda's
\cite{majda84:_compr_fluid_flow_system_conser} iteration procedure and
show existence, uniqueness and continuity in that norm.

We consider the  Cauchy problem for a quasi linear (uniform)
symmetric hyperbolic system of the form
\begin{equation}
  \label{eq:neu-existence:21}
  \begin{cases}
      & \displaystyle{A^0(u;t,x)\partial_t u + \sum_{a=1}^3 A^a(u;t,x)\partial_a u +
      B(u;t,x)u+ F(u;t,x)=0,}\\
      & u(0,x)=u_0(x).
    \end{cases}
  \end{equation}
under the following assumptions:

\begin{enumerate}
  \item[(H1)]
\label{item:wellposs-einstein-euler-hyper:1}
  $A^\alpha$ are symmetric matrices for
$\alpha=0,1,2,3$;
  \item[(H2)]
  \label{item:wellposs-einstein-euler-hyper:2}
$A^\alpha(u;t,x),B(u;t,x),F(u;t,x) $ are smooth in their arguments;
  \item[(H3)] 
\label{item:wellposs-einstein-euler-hyper:3}
$ \left( A^0(0;t,\cdot)-I\right), A^a(0;t,\cdot),
  B(0;t,\cdot), F(0;t,\cdot)\in H_{s,\delta}$;
  \item[(H4)]
  \label{item:wellposs-einstein-euler-hyper:4}
  $\partial_t A^0(u;t,\cdot)\in L^\infty$.
\end{enumerate}

The main result of this section is the well posedness of the system
(\ref{eq:neu-existence:21}) in $H_{s,\delta}$ spaces:
\begin{thm}[Well posedness of first order hyperbolic symmetric systems
  in $H_{s,\delta}$]
  \label{thr:publ-broken:1} Let $s>\frac{5}{2}$, $\delta \geq
-\frac{3}{2}$ and assume hypotheses
(H1)-(H4) 
holds. If the initial condition $u_0$ belongs to $
H_{s,\delta}$ and satisfies
  \begin{equation}
    \label{eq:publ-broken:10} \frac{1}{\mu}
    \delta_{\alpha\beta}u_0^\alpha  u_0^\beta \leq A^0_{\alpha\beta}
    u_0^\alpha u_0^\beta\leq
    \mu \delta_{\alpha\beta}u_0^\alpha u_0^\beta, \qquad \mu\in
    \setR^+
  \end{equation}
  then there exits a positive $T$ which depends on the
  $H_{s,\delta}$-norm of the initial data and there exists a unique
  $u(t,x)$ a solution to (\ref{thr:publ-broken:1}) which in addition
  satisfies
  \begin{equation}
    \label{eq:appendix:11}
    u\in C([0,T],H_{s,\delta})\cap C^1([0,T],H_{s-1,\delta+1}).
  \end{equation}
\end{thm}

\begin{rem}
  Condition (H3) is sometime too restrictive for applications.  We may
  replace it by
  \begin{enumerate}
    \item[(H3')] $ \left( A^0(U^0;t,\cdot)-I\right), A^a(U^0;t,\cdot),
    B(U^0;t,\cdot), F(U^0;t,\cdot)\in H_{s,\delta}$,
  \end{enumerate}
  where $U^0$ is a constant vector. Setting $u=U^0+v$, then $v$
  satisfies
  \begin{equation}
    \label{eq:formulation:4}
    \begin{cases}
        &\tilde{A}^0(v;t,x)\partial_t v = \sum_{a=1}^3
        \tilde{A}^a(u;t,x)\partial_a v +
        \tilde{B}(v;t,x)v+\tilde{F}(v;t,x) \\
        &v(0,x)=u_0(x)-U^0
    \end{cases}
  \end{equation}
  where $\tilde{A}^\alpha(v;t,x)={A}^\alpha(U^0+v;t,x)$,
  $\tilde{B}(v;t,x)={B}(U^0+v;t,x)$ and
  $\tilde{F}(v;t,x)=F(U^0+v;t,x)+\tilde{B}(U^0+v;t,x)$. The Moser type
  estimates are valid under assumptions (H3') (see Remark
  \ref{rem:Properties:2}).
\end{rem}

\subsection{Strategy}
\label{sec:strategy-1}

We will proceed with the following strategy:
\begin{enumerate}
  \item{ The establishment of energy estimates for linear systems in
    the fractional weighted spaces $H_{s,\delta}$.}
  \item We approximate the initial data by a $C_0^\infty$ sequence and
  then construct an iteration process which consists of solutions to a
  linear system having a $C_0^\infty$ initial data.
  \item We show that the sequence which is constructed by the
  iteration process is bounded in $H_{s,\delta}$-norm and weakly
  converges to a solution.

  \item At the final stage we prove uniqueness and continuity in
  $H_{s,\delta}$-norm.
\end{enumerate}

\subsection{Energy estimates in the fractional weighted spaces}
\label{subsec:Energy_estimates}

The energy estimates are indispensable means for the proof of well
posedness of hyperbolic systems. In order to achieve it we introduce
an inner product which depends on a matrix $A$. We assume $A=A(t,x)$
is $m\times m $ symmetric matrix which satisfies
\begin{equation}
  \label{eq:energy-estimates:3}
  \frac{1}{\mu} U^TU\leq U^TAU\leq\mu U^TU
\end{equation}
for some positive $\mu$. Here $B^T$ denotes the transpose matrix.  We
recall that $f_\epsilon(x)=f(\epsilon x)$, the sequence $\{\psi_j\}$
is a dyadic resolution of the unity in $\mathbb{R}^3$ which is defined
the Appendix \ref{sec:constr-spac-h_s} and that
$\Lambda^su={\cal{F}}^{-1}\left( (1+|\xi|^2)^{\frac{s}{2}}
  {\cal{F}}u\right) $, where ${\cal{F}}$ denotes the Fourier
transform. In this section the expression
(\ref{eq:energy-estimates:1}) below will serve as a norm of the space
$H_{s,\delta}$:
\begin{equation}
  \label{eq:energy-estimates:1}
  \|u\|_{H_{s,\delta}}^2
  :=\sum_{j=0}^\infty 2^{( \frac{3}{2} + \delta)2j}
  \|(\psi_j^2
  u)_{(2^j)}\|_{H^s}^2.
\end{equation}

Corollary \ref{cor:app11} implies that (\ref{eq:energy-estimates:1})
is equivalent to the norm of Definition \ref{def:weighted:3}.

\begin{defn}[Inner Product]
  \label{def:inner_product}
  For a symmetric matrix $A=A(t,x)$ which satisfies
  (\ref{eq:energy-estimates:3}) we let
  \begin{eqnarray} \nonumber \langle u,v\rangle_{s,\delta,A}&:=&
    \sum_{j=0}^\infty 2^{( \frac{3}{2} + \delta)2j} \left\langle
      \Lambda^s\left((\psi_j^2 u)_{(2^j)}\right),\left(A\right)_{2^j}
      \Lambda^s\left((\psi_j^2 v)_{(2^j)}\right)\right\rangle_{L^2}\\
    \label{eq:energy-estimates:4}
    &=&\sum_{j=0}^\infty 2^{( \frac{3}{2} + \delta)2j}
    \int \left[\Lambda^s\left((\psi_j^2
        u)_{(2^j)}\right)\right]^T\left(A\right)_{2^j}
    \left[\Lambda^s\left((\psi_j^2
        v)_{(2^j)}\right)\right]dx
  \end{eqnarray}
  and its associated norm $\|u\|_{H_{s,\delta,A}}^2=\langle
  u,u\rangle_{s,\delta,A}$.
\end{defn}

Obviously $\langle u,v\rangle_{s,\delta,A}=\langle
v,u\rangle_{s,\delta,A} $ and from (\ref{eq:energy-estimates:3}) we
obtain the equivalence,
\begin{equation}
  \label{eq:energy-estimates:5}
  \frac{1}{\mu}\|u\|_{H_{s,\delta}}^2\leq  \|u\|_{H_{s,\delta,A}}^2
  \leq {\mu}\|u\|_{H_{s,\delta}}^2.
\end{equation}

We come now to the crucial estimate of this section.
\begin{lem}[An energy estimate]
  \label{lem:energy-estimates:1}
  Let $s>\frac{5}{2}$, $\delta \geq -\frac{3}{2}$,
  $A^\alpha=A^\alpha(t,x)$ be $m\times m$ symmetric matrices such that
  $(A^0(t,\cdot)-I), A^a(t,\cdot)\in H_{s,\delta}$ and $A^0$ satisfies
  (\ref{eq:energy-estimates:3}). If $u(t)=u(t,\cdot)$ is a
  $C^{\infty}_0$ solution of the linear hyperbolic system
  \begin{equation}
    \label{eq:energy-estimates:7} A^0(t,x)\partial_t u= \sum_{a=1}^3
    A^a(t,x)\partial_a u,
  \end{equation}
  then
  \begin{equation}
    \label{eq:energy-estimates:8}
    \frac{d}{dt}\|u(t)\|_{H_{s,\delta,A^0}}^2\leq C \left(
      \mu\|u(t)\|_{H_{s,\delta,A^0}}^2  +1 \right),
  \end{equation}

  where
  $C=C(\left\|A^0-I\right\|_{H_{s,\delta}},\left\|A^a\right\|_{H_{s,\delta}},
  \left\|\partial_t u\right\|_{H_{s-1,\delta}},\|\partial_t
  A^0\|_{L^\infty})$.
\end{lem}

An essential tool for deriving these estimates is the Kato \& Ponce
Commutator Estimate \cite{kato88:_commut_euler_navier_stokes},
\cite{Taylor91}.

\begin{thm}[Kato and Ponce]
  \label{thm:energy-estimates:1}
  Let $s>0$, $f\in H^s\cap C^1$, $g\in H^{s-1}\cap L^{\infty}$, then
  \begin{equation}
    \label{eq:energy-estimates:9} \| \Lambda^s(fg)-f
    \Lambda^{s}g\|_{L^2} \leq C \left\{ \|\nabla f\|_{L^{\infty}}
      \|g\|_{H^{s-1}} + \|f\|_{H^s}\|g\|_{L^{\infty}}
    \right\}.
  \end{equation}
\end{thm}

This estimate will be used term wise in the inner product
(\ref{eq:energy-estimates:4}).

\begin{proof}[of Lemma \ref{lem:energy-estimates:1}]
  \label{sec:energy-estim-fract-1}
  Since $u$ is $C_0^\infty$ we may interchange the derivation with
  respect to $t$ with the inner-product (\ref{eq:energy-estimates:4})
  and get
  \begin{eqnarray}
\nonumber
    \frac{d}{dt}  \left\langle u,u
    \right\rangle_{s,\delta,A^0} &=& 2 \left\langle u,\partial_t u
    \right\rangle_{s,\delta,A^0} \\
    \nonumber  &+&\sum_{j=0}^\infty 2^{( \frac{3}{2} + \delta)2j}
    \int \left[\Lambda^s\left((\psi_j^2
        u)_{(2^j)}\right)\right]^T\left(\partial_t A^0\right)_{2^j}
    \left[\Lambda^s\left((\psi_j^2
        u)_{(2^j)}\right)\right]dx \\
    \nonumber
    &\leq&
    2 \left\langle u,\partial_t u
    \right\rangle_{s,\delta,A^0}+\|\partial_t A^0\|_{L^\infty}\left(
      \sum_{j=0}^\infty 2^{( \frac{3}{2} + \delta)2j}
      \|(\psi_j^2
      u)_{(2^j)}\|_{H^s}^2\right)\\ \label{eq:energy-estimates:10} &=&
    2\left\langle u,\partial_t u
    \right\rangle_{s,\delta,A^0}+\|\partial_t A^0\|_{L^\infty}
    \|u\|_{H_{s,\delta}}^2
  \end{eqnarray}

  We turn now to the hard task of the proof, namely, the estimation of
  $\left\langle u,\partial_t u \right\rangle_{s,\delta,A^0}$. Put
  \begin{equation}
    \label{eq:energy-estimates:11} E(j)=
    \left\langle\Lambda^s\left(\left(\psi_j^2u\right)_{2^j}\right),
      \left((A^0)_{2^j}\right)
      \Lambda^s\left(\left(\psi_j^2\partial_t u\right)_{2^j}\right)
    \right\rangle_{L^2}
  \end{equation}
  and let $\{\Psi_k\}=\frac{1}{\sum \psi_j(x)}\psi_k(x)$,
where $\{\psi_j\}$  is defined in  Appendix
  \ref{sec:constr-spac-h_s}. It follows from the properties of this
sequence that
  \begin{equation}
    \label{eq:energy-estimates:12} \Psi_k\psi^2_j\not=0\qquad \text{
      only when } k=j-3,...,j+4.
  \end{equation}
  Hence,
  \begin{eqnarray*} E(j) &=&
\left\langle\Lambda^s\left(\left(\psi_j^2u\right)_{2^j}\right),
\left((A^0)_{2^j}\right)
\Lambda^s\left(\left(\sum_{k=0}^\infty\Psi_k\right)_{2^j}
\left(\psi_j^2\partial_t u\right)_{2^j}\right)\right\rangle_{L^2} \\ &
= &\sum_{k=j-3}^{j+4}\left\langle
\Lambda^s\left(\left(\psi_j^2u\right)_{2^j}\right),
\left((A^0)_{2^j}\right) \Lambda^s\left(\left(\Psi_k\right)_{2^j}
\left(\psi_j^2\partial_t u\right)_{2^j}\right)\right\rangle_{L^2} \\ &
= &\sum_{k=j-3}^{j+4}\left\langle\Lambda^s
\left(\left(\psi_j^2u\right)_{2^j}\right), (A^0)_{2^j}
\left[\Lambda^s\left(\left(\Psi_k\right)_{2^j}
\left(\psi_j^2\partial_t u\right)_{2^j}\right)
-\left(\Psi_k\right)_{2^j} \Lambda^s\left(\psi_j^2\partial_t
u\right)_{2^j}\right]\right\rangle_{L^2} \\ & + &
\sum_{k=j-3}^{j+4}\left\langle\Lambda^s\left(\left(\psi_j^2u\right)_{
2^j} \right), (\Psi_kA^0)_{2^j} \Lambda^s\left(\psi_j^2\partial_t
u\right)_{2^j}\right\rangle_{L^2}\\ \\ &= &E_{1}(j,k)+E_{2}(j,k).
\end{eqnarray*} This splitting will enable us to estimate $
E_{2}(j,k)$ in terms of the $H_{s,\delta}$ norm of $A^0-I$ while by
Theorem \ref{thm:energy-estimates:1}, 
\begin{eqnarray} 
\nonumber &
&|E_{1}(j,k)|\\ \nonumber
&\leq&\left\|\Lambda^s\left(\left(\psi_j^2u\right)_{2^j}
\right)\right\|_{L^2} \left\|A^0\right\|_{L^\infty}\left\|
\Lambda^s\left(\left(\Psi_k\right)_{2^j}\left(\psi_j^2\partial_t
u\right)_{2^j}\right)
-\left(\Psi_k\right)_{2^j}\Lambda^s\left(\psi_j^2\partial_t
u\right)_{2^j} \right\|_{L^2} \\ \nonumber
&\leq&\left\|\left(\psi_j^2u\right)_{2^j}\right\|_{H^s}
\left\|A^0\right\|_{L^\infty}\left\{\left\|
\nabla\left(\Psi_k\right)_{2^j}\right\|_{L^\infty}\left\|
\left(\psi_j^2\partial_t u\right)_{2^j}\right\|_{H^{s-1}} +
\left\|\left(\Psi_k\right)_{2^j}\right\|_{H^s}
\left\|\left(\psi_j^2\partial_t u\right)_{2^j}
\right\|_{L^\infty}\right\} \\ & \leq& C
\left\|A^0\right\|_{L^\infty}\left(
\left\|\left(\psi_j^2u\right)_{2^j}\right\|_{H^s}
\left\|\left(\psi_j^2\partial_t
u\right)_{2^j}\right\|_{H^{s-1}}\right). \end{eqnarray}

  In the last step above we have used the below useful estimates.
  First, by (\ref{eq:const:4}) and (\ref{eq:energy-estimates:12}),
  \begin{equation}
    \label{eq:energy-estimates:13}
    \left\|\nabla\left(\Psi_k\right)_{2^j}\right\|_{L^\infty}=2^j\left\|
      \nabla\Psi_k\right\|_{L^\infty}\leq C2^j 2^{-k}\leq  8 C.
  \end{equation}
  Secondly, it is well known that for any smooth function $f$,
\begin{equation}
 \label{eq:multiplication:1} 
\|fu\|_{H^s}\leq C(\|f\|_{C^N})\|u\|_{H^s},
\end{equation} 
where the integer $N$ is not less than $s$. In addition, 
 from (\ref{eq:const:9}) we see that
  \begin{equation}
    \label{eq:energy-estimates:26}
    \|f_\epsilon\|_{H^s}^2\lesssim\left\{\begin{array}{cc}
        \epsilon^{-3} \|f\|_{H^s}^2, & \epsilon\leq 1 \\
        \epsilon^{2s-3} \|f\|_{H^s}^2, & \epsilon  \geq  1 \\
      \end{array}\right..
  \end{equation}
  Recalling that $\psi_k(x)=\psi_1(2^{-k}x)$ and
  $\left(\psi_k(x)\right)_{2^j}=\left(\psi_1(x)\right)_{2^{j-k}}$,
  applying the above and combing this with
  (\ref{eq:energy-estimates:12}) and (\ref{eq:multiplication:1}), we have
  \begin{eqnarray}\nonumber
    \left\|\left(\Psi_k\right)_{2^j}\right\|_{H^s}&=& \left\|
      \left(\sum_j
        \psi_j\right)_{2^j}^{-1}\left(\psi_k\right)_{2^j}\right\|_{H^s}\leq
    C\left\|
      \left(\psi_k\right)_{2^j}\right\|_{H^s}\\
    \label{eq:energy-estimates:14}&=&C\left\|
      \left(\psi_1\right)_{2^(j-k)}\right\|_{H^s}\leq C
    2^{(s-\frac{3}{2})3}\left\|\psi_1\right\|_{H^s}.
  \end{eqnarray}

  Finally, by the Sobolev embedding
  \begin{equation}
    \label{eq:energy-estimates:15} \left\|v\right\|_{L^\infty}\leq C
    \left\|v\right\|_{H^{s}},
  \end{equation}
  we obtain \begin{math} \left\|\left(\psi_j^2 \partial_t
        u\right)_{2^j}\right\|_{L^\infty}\leq C \left\|\left(\psi_j^2
        \partial_t u\right)_{2^j}\right\|_{H^{s-1}}
  \end{math}.

  In order to use equation (\ref{eq:energy-estimates:7}) we split
  $E_{2}(j,k)$ as follows:
  \begin{eqnarray*}
    E_{2}(j,k)
    &=&
    \left\langle\Lambda^s\left(\left(\psi_j^2u\right)_{2^j}\right),
      \left((\Psi_k A^0)_{2^j}\right) \Lambda^s\left(
        \left(\psi_j^2\partial_t u\right)_{2^j}\right)\right\rangle_{L^2} \\
    & = &\left\langle\Lambda^s\left(\left(\psi_j^2u\right)_{2^j}\right),
      \left[\left(\Psi_k
          A^0\right)_{2^j}\Lambda^s\left(\left(\psi_j^2\partial_t
            u\right)_{2^j}\right)
        -\Lambda^s\left(\left(\Psi_kA^0\right)_{2^j}
          \left(\psi_j^2\partial_t u\right)_{2^j}\right)\right]\right\rangle_{L^2}
    \\
    & + &\left\langle\Lambda^s\left(\left(\psi_j^2u\right)_{2^j}\right),
      \Lambda^s\left(\left(\Psi_kA^0\right)_{2^j}
        \left(\psi_j^2\partial_t u\right)_{2^j}\right)\right\rangle_{L^2}
    \\
    &= &E_{3}(j,k)+E_{4}(j,k).
  \end{eqnarray*}

  In the estimation of the first term $E_{3}(j,k)$, the Kato-Ponce
  commutator estimate (\ref{eq:energy-estimates:9}) is being used
  again:
  \begin{eqnarray*}
    & &|E_{3}(j,k)|\\
    &\leq&C\left\|\left(\psi_j^2u\right)_{2^j}\right\|_{H^s}
    \left\{\left\|
        \nabla\left(\Psi_kA^0\right)_{2^j}\right\|_{L^\infty}\left\|
        \left(\psi_j^2\partial_t u\right)_{2^j}\right\|_{H^{s-1}}+
      \left\|\left(\Psi_kA^0\right)_{2^j}\right\|_{H^s}
      \left\|\left(\psi_j^2\partial_t u\right)_{2^j}
      \right\|_{L^\infty}\right\}.
  \end{eqnarray*}

  From (\ref{eq:energy-estimates:13}) and the embedding
  (\ref{eq:energy-estimates:15}), we have
  \begin{eqnarray*}
    \left\|
      \nabla\left(\Psi_kA^0\right)_{2^j}\right\|_{L^\infty}&=& 2^j
    \left\|\left(\nabla\left(\Psi_kA^0-I\right)\right)_{2^j}\right\|_{L^\infty}
    +2^j\left\|  \nabla(\Psi_k)_{2^j}\right\|_{L^\infty}\\
    & \leq&
    C\left\{ 2^j \left\|
        \left(\nabla\Psi_k\left(A^0-I\right)\right)_{2^j}\right\|_{H^{s-1}}+
      1\right\}
  \end{eqnarray*}
  and from (\ref{eq:energy-estimates:14})
  \begin{eqnarray*}
    \left\| \left(\Psi_kA^0\right)_{2^j}\right\|_{H^s}&\leq&
    \left\|
      \left(\Psi_k\left(A^0-I\right)\right)_{2^j}\right\|_{H^s} +\left\|
      \nabla\left(\Psi_k\right)_{2^j}\right\|_{H^s}\leq \left\|
      \left(\Psi_k\left(A^0-I\right)\right)_{2^j}\right\|_{H^s} + C.
  \end{eqnarray*}
  Thus
  \begin{eqnarray}\nonumber
    & & |E_{3}(j,k)|\\ \nonumber
    &\leq&C\left\|\left(\psi_j^2u\right)_{2^j}\right\|_{H^s}\quad\left\|
      \left(\psi_j^2\partial_t u\right)_{2^j}\right\|_{H^{s-1}}
    \quad\left\{2^j\left\|
        \left(\nabla\Psi_k\left(A^0-I\right)\right)_{2^j}
      \right\|_{H^{s-1}}+1\right\}  
    \\  \nonumber & + &
    C\left\|\left(\psi_j^2u\right)_{2^j}\right\|_{H^s}
    \left\|\left(\psi_j^2\partial_t u\right)_{2^j}\right\|_{L^\infty}\left\{
      \left\|\left(\Psi_k\left(A^0-I\right)\right)_{2^j}\right\|_{H^s}+1\right\}
    \\    &\leq&
    C\left\|\left(\psi_j^2u\right)_{2^j}\right\|_{H^s}\quad\left\|
      \left(\psi_j^2\partial_t u\right)_{2^j}\right\|_{H^{s-1}}\left\{ 2^j\left\|
        \left(\nabla\Psi_k\left(A^0-I
          \right)\right)_{2^j}\right\|_{H^{s-1}}+1\right\} 
    \\ \nonumber & + &
    C\left\|\left(\psi_j^2u\right)_{2^j}\right\|_{H^s}
    \left\|\left(\Psi_k\left(A^0-I\right)\right)_{2^j}\right\|_{H^s}
    \left\|\partial_t u\right\|_{L^\infty}
    \\   \nonumber  & + &
    C\left\|\left(\psi_j^2u\right)_{2^j}\right\|_{H^s}
    \left\|\left(\psi_j^2\partial_t u\right)_{2^j}\right\|_{H^{s-1}}.\nonumber
  \end{eqnarray}
  Now equation (\ref{eq:energy-estimates:7}) is being utilized and
  \begin{eqnarray}\nonumber
    E_{4}(j,k)
    &=&
    \left\langle\Lambda^s\left(\left(\psi_j^2u\right)_{2^j}\right),
      \Lambda^s\left(
        \left(\Psi_k\psi_j^2\right)_{2^j}
        \left(A^0\partial_t u\right)\right)_{2^j}\right\rangle_{L^2} \\
    \nonumber &=&
    \left\langle\Lambda^s\left(\left(\psi_j^2u\right)_{2^j}\right),
      \Lambda^s\left(
        \left(\Psi_k\psi_j^2\right)_{2^j}
        \left(\sum_{a=1}^3A^a\partial_au\right)_{2^j}\right)
    \right\rangle_{L^2} \\
    \label{eq:energy-estimates:16} &=&\sum_{a=1}^3
    \left\langle\Lambda^s\left(\left(\psi_j^2u\right)_{2^j}\right),
      \Lambda^s\left(
        \left(\Psi_kA^a\right)_{2^j}\left(\psi_j^2\partial_au
        \right)_{2^j}\right)\right\rangle_{L^2}\\ 
    \nonumber
    & = &
    \sum_{a=1}^3\left\langle\Lambda^s
      \left(\left(\psi_j^2u\right)_{2^j}\right), 
      \left[\Lambda^s\left(
          \left(\Psi_kA^a\right)_{2^j}
          \left(\psi_j^2\partial_au\right)_{2^j}\right)
        -
        \left(\Psi_kA^a\right)_{2^j}
        \Lambda^s\left(\left(\psi_j^2\partial_au\right)_{2^j}
        \right)\right]\right\rangle_{L^2}
    \\\nonumber
    & + &
    \sum_{a=1}^3\left\langle\Lambda^s\left(\left(\psi_j^2u\right)_{2^j}\right),
      \left[\left(\Psi_kA^a\right)_{2^j}
        \Lambda^s\left(\left(\psi_j^2\partial_au\right)_{2^j}
\right)\right]\right\rangle_{L^2}\\
    \nonumber  &= &E_{5}(j,k,a)+E_{6}(j,k,a).
  \end{eqnarray}
  Again, by Kato-Ponce commutator estimate
  (\ref{eq:energy-estimates:9}),
  \begin{eqnarray}\nonumber
    & & |E_{5}(j,k,a)|\\ \nonumber
    &\leq&C\left\|\left(\psi_j^2u\right)_{2^j}\right\|_{H^s}
    \left\{\left\|
        \nabla\left(\Psi_kA^a\right)_{2^j}\right\|_{L^\infty}\left\|
        \left(\psi_j^2\partial_a u\right)_{2^j}\right\|_{H^{s-1}}
      +
      \left\|\left(\Psi_kA^a\right)_{2^j}\right\|_{H^s}
      \left\|\left(\psi_j^2\partial_au\right)_{2^j}
      \right\|_{L^\infty}\right\}
    \\  & \leq& 
    C \left\|\left(\psi_j^2u\right)_{2^j}\right\|_{H^s}
    \left\{\left\|
        \nabla A^a\right\|_{L^\infty} +\left\|
        A^a\right\|_{L^\infty}\right\}2^j\left\|
      \left(\psi_j^2\partial_a u\right)_{2^j}\right\|_{H^{s-1}}\\
    & +& C \left\|\left(\psi_j^2u\right)_{2^j}\right\|_{H^s}
    \left\|\left(\Psi_kA^a\right)_{2^j}\right\|_{H^s}\quad\left\|
      \partial_a u\right\|_{L^\infty}.\nonumber
  \end{eqnarray}  
Using the commutation
$\partial_a\Lambda^s=\Lambda^s\partial_a$ and the fact that $\Lambda^s
\left(\psi_j^2u
  \right)$ is rapidly decreasing, we see that
  \begin{eqnarray*}
    0 &=&  \int \partial_a \left\{
      \left[ \Lambda^s  \left( \left(\psi_j^2u
          \right)_{2^j}\right)\right]^T \left(\Psi_k A^a\right)_{2^j}\left[
        \Lambda^s \left(\left( \psi_j^2u \right)_{2^j} \right)
      \right]\right\} dx \\& =& 2^j\int\left\{ \left[ \Lambda^s  \left(
          \left(\psi_j^2\partial_a u \right)_{2^j}\right)\right]^T \left(\Psi_k
        A^a\right)_{2^j}\left[ \Lambda^s \left(\left( \psi_j^2u
          \right)_{2^j} \right) \right]\right\} dx \\& +& 2^j\int\left\{
      \left[ \Lambda^s  \left( \left(\psi_j^2 u
          \right)_{2^j}\right)\right]^T \left(\Psi_k A^a\right)_{2^j}\left[
        \Lambda^s \left(\left( \psi_j^2\partial_au \right)_{2^j} \right)
      \right]\right\} dx\\& +& 2^j2\int\left\{ \left[ \Lambda^s  \left(
          \left(\left(\partial_a\psi_j\right)\psi_j u \right)_{2^j}\right)\right]^T
      \left(\Psi_k A^a\right)_{2^j}\left[ \Lambda^s \left(\left(
            \psi_j^2u \right)_{2^j} \right) \right]\right\} dx\\& +&
    2^j2\int\left\{ \left[ \Lambda^s  \left( \left(\psi_j^2 u
          \right)_{2^j}\right)\right]^T \left(\Psi_k A^a\right)_{2^j}\left[
        \Lambda^s \left(\left(
            (\partial_a\psi_j)\psi_ju \right)_{2^j} \right) \right]\right\} dx
    \\& +& 2^j\int\left\{ \left[ \Lambda^s  \left(
          \left(\psi_j^2 u \right)_{2^j}\right)\right]^T
      \left(\partial_a\left(\Psi_k A^a\right)\right)_{2^j}\left[ \Lambda^s
        \left(\left( \psi_j^2u \right)_{2^j} \right) \right]\right\} dx.
  \end{eqnarray*}
Now  $E_{6}(j,k,l)$ is equal to the
second term, but since $A^a$ is a
symmetric matrix, the first and the second terms are equal and also
the third and the forth. Hence by (\ref{eq:multiplication:1}) and
Cauchy Schwarz inequality,

  \begin{eqnarray*}
    |2E_{6}(j,k,a)|
    &\leq&2\left\|\left(\Psi_kA^a\right)_{2^j}\right\|_{L^\infty}
    \left\|\left(\partial_a\psi_j\psi_j u\right)_{2^j}\right\|_{H^s}\quad\left\|
      \left(\psi_j^2 u\right)_{2^j}\right\|_{H^{s}}
    \\&+&
    \left\|\left(\partial_a\left(\Psi_kA^a\right)\right)_{2^j}
    \right\|_{L^\infty}
    \left\| \left(\psi_j^2 u\right)_{2^j}\right\|_{H^{s}}^2\\
    &\leq&C\left\|A^a\right\|_{L^\infty}
    \left\|\left(\psi_ju\right)_{2^j}\right\|_{H^s}\quad\left\|
      \left(\psi_j^2 u\right)_{2^j}\right\|_{H^{s}}
    \\&+&
    \left\{\left\|\partial_aA^a\right\|_{L^\infty}+C\left\|A^a
      \right\|_{L^\infty}\right\}
    \left\| \left(\psi_j^2 u\right)_{2^j}\right\|_{H^{s}}^2.
  \end{eqnarray*}

  Taking the sum $\sum 2^{(\frac{3}{2}+\delta)2j} E(j)$ we are coming
  across three types of summations:

  \begin{enumerate}
    \item Given $v\in H_{s_1,\delta}$, $w\in H_{s_2,\delta}$ and
    $\gamma_i$ equals $1$ or $2$, then
    \begin{eqnarray*}
      & &\sum_{j=0}^\infty2^{(\frac{3}{2}
        +\delta)2j}\left\|(\psi_j^{\gamma_1}v)_{2^j}\right\|_{H^{s_1}}
      \left\|(\psi_j^{\gamma_2}w)_{2^j}\right\|_{H^{s_2}}\\&\leq&
      \frac{1}{2}\left(\sum_{j=0}^\infty2^{(
          \frac{3}{2}+\delta)2j}\left\|(\psi_j^{\gamma_1}v)_{2^j}
        \right\|_{H^{s_1}}^2+
        2^{(\frac{3}{2} +
          \delta)2j}\left\|(\psi_j^{\gamma_2}w)_{2^j}\right\|_{H^{s_2}}^2\right)\\
      &\leq& C\left(
        \|v\|_{H_{s_1,\delta}}^2+\|w\|_{H_{s_2,\delta}}^2\right),
    \end{eqnarray*}
    where in the last inequality the equivalence of the norms, (see
    Corollary \ref{cor:app11}), was involved.

    \item Given $v\in H_{s,\delta}$ and $w\in H_{s,\delta}$, then from
    the scaling property (\ref{eq:energy-estimates:26}) and
    (\ref{eq:multiplication:1}) we have 
    \begin{eqnarray*}
      & &\sum_{j=0}^\infty\sum_{k=j-3}^{j+4}
      2^{(\frac{3}{2}+\delta)2j}\left\|(\psi_j^2v)_{2^j}\right\|_{H^{s}}
      \left\|(\Psi_kw)_{2^j}\right\|_{H^{s}}\\&\leq&\frac{1}{2}
      \sum_{j=0}^\infty \sum_{k=j-3}^{j+4} 2^{(\frac{3}{2}+\delta)2j}
      \left\|(\psi_j^{2}v)_{2^j}\right\|_{H^{s}}^2+
      \frac{1}{2}\sum_{j=0}^\infty
      \sum_{k=j-3}^{j+4} 2^{(\frac{3}{2}+\delta)2j}\left\|(\Psi_k
        w)_{2^j}
      \right\|_{H^{s}}^2\\
      &\leq&\frac{7}{2} \|v\|_{H_{s,\delta}}^2+ C \sum_{j=0}^\infty
      \sum_{k=j-3}^{j+4} 2^{(\frac{3}{2}+\delta)2j}\left\|(\Psi_k
        w)_{2^k}\right\|_{H^{s}}^2\\&\leq &\frac{7}{2}
      \|v\|_{H_{s,\delta}}^2+ C \sum_{j=0}^\infty \sum_{k=j-3}^{j+4}
      2^{(\frac{3}{2}+\delta)2k}\left\|(\psi_k
        w)_{2^k}\right\|_{H^{s}}^2\\&\leq &
      C\left(
        7\|v\|_{H_{s,\delta}}^2+7\|w\|_{H_{s,\delta}}^2\right).
    \end{eqnarray*}
    \item Given $v\in H_{s_1,\delta}$, $w\in H_{s_2,\delta}$, $z\in
    H_{s_3,\delta}$ and $\gamma_i$ equals $1$ or $2$, then by H\"older
    inequality, Corollary \ref{cor:const:1} and the same arguments as
in type 2, we get
    \begin{eqnarray*}
      & &\sum_{j=0}^\infty\sum_{k=j-3}^{j+4}2^{(\frac{3}{2} +
        \delta)2j}\left\|(\psi_j^{\gamma_1}v)_{2^j}\right\|_{H^{s_1}}
      \left\|(\psi_j^{\gamma_2}w)_{2^j}\right\|_{H^{s_2}}2^j
      \left\|\left(\nabla
          \left(\Psi_k z\right)\right)_{2^j}\right\|_{H^{s_3-1}}\\
      &\leq& \sum_{j=0}^\infty\sum_{k=j-3}^{j+4}2^{(\frac{3}{2} +
        \delta)j}\left\|(\psi_j^{\gamma_1}v)_{2^j}\right\|_{H^{s_1}} 
      2^{(\frac{3}{2} + \delta)j}
      \left\|(\psi_j^{\gamma_2}w)_{2^j}\right\|_{H^{s_2}}2^{(\frac{3}{2}
        +\delta+1)j} \left\|\left(\nabla \left(\Psi_k
            z\right)\right)_{2^j}\right\|_{H^{s_3-1}}\\ &\leq&
      \left(\left(\sum\limits_{j=0}^{\infty} \sum\limits_{k=j-3}^{j+4}
          \left(2^{\left(\frac{3}{2}+\delta \right)2j}
            \left
              \|\left(\psi_j^{\gamma_1}v\right)_{2^j}
            \right\|_{H^{s_1}}^2\right)^2\right)
        ^{\frac{1}{2}}\right)^{\frac{1}{2}}\\
      & \times& \left(\left(\sum\limits_{j=0}^{\infty}
          \sum\limits_{k=j-3}^{j+4}
          \left(2^{\left(\frac{3}{2}+\delta \right)2j}
            \left
              \|\left(\psi_j^{\gamma_2}w\right)_{2^j}
            \right\|_{H^{s_2}}^2\right)^2\right)
        ^{\frac{1}{2}}\right)^{\frac{1}{2}}
      \\
      & \times & C\left(\sum\limits_{j=0}^{\infty}
        \sum\limits_{k=j-3}^{j+4}
        2^{\left(\frac{3}{2}+\delta + 1\right)2j}
        \left\|\left(\nabla ({\psi_{k}}z
            )\right)_{2^{k}}\right\|_{H^{s_3-1}}^2\right) ^{\frac{1}{2}}\\
      &\leq& \left(\sum\limits_{j=0}^{\infty}
        \sum\limits_{k=j-3}^{j+4} 2^{\left(\frac{3}{2}+\delta
          \right)2j} \left
          \|\left(\psi_j^{\gamma_1}v\right)_{2^j}\right\|_{H^{s_1}}^2
      \right)^{\frac{1}{2}}\\
      & \times& \left(\sum\limits_{j=0}^{\infty}
        \sum\limits_{k=j-3}^{j+4} 2^{\left(\frac{3}{2}+\delta
          \right)2j} \left \|\left(\psi_j^{\gamma_2}
            w\right)_{2^j}\right\|_{H^{s_2}}^2\right)^{\frac{1}{2}} \\
      & \times & C\left(\sum\limits_{j=0}^{\infty}
        \sum\limits_{k=j-3}^{j+4} 2^{\left(\frac{3}{2}+\delta +
            1\right)2k} \left\|\left(\nabla ({\psi_{k}}z
            )\right)_{2^{k}}\right\|_{H^{s_3-1}}^2\right)
      ^{\frac{1}{2}}\\
      &\leq& C \|v\|_{H_{s_1,\delta}}\quad\|w\|_{H_{s_2,\delta}}\quad
      \|\nabla z\|_{H_{s_3-1,\delta+1}}\\
      &\leq& C
      \|v\|_{H_{s_1,\delta}}\quad\|w\|_{H_{s_2,\delta}}\quad
      \|z\|_{H_{s_3,\delta}}\\ 
      &\leq&C \left(\|v\|_{H_{s_1,\delta}}^2
        +\left(\|w\|_{H_{s_2,\delta}}
          \|z\|_{H_{s_3,\delta}}\right)^2\right).
    \end{eqnarray*}

  \end{enumerate}
  Applying these three types of inequalities we have,
  \begin{eqnarray}
    \label{eq:energy-estimates:17}
    \sum_{j=0}^\infty\sum_{k=j-3}^{j+4}2^{(\frac{3}{2}+\delta)2j}
    |2E_{6}(j,k,a)|\leq C
    \left(\left\|A^a\right\|_{L^\infty}+\left\|\partial_a
        A^a\right\|_{L^\infty}\right)\left\|u\right\|_{H_{s,\delta}}^2,
  \end{eqnarray}
  \begin{eqnarray}\nonumber
    & &
    \sum_{j=0}^\infty\sum_{k=j-3}^{j+4}2^{(\frac{3}{2}+\delta)2j}
    |E_{5}(j,k,a)|  \\ \nonumber  & \leq  & C\left(\left\|\nabla
        A^a\right\|_{L^\infty}+\left\|
        A^a\right\|_{L^\infty}\right)\left\{\left\|u\right\|_{H_{s,\delta}}^2
      + \left\|\partial_a u\right\|_{H_{s-1,\delta+1}}^2\right\}+ C
\left\{\left\|u\right\|_{H_{s,\delta}}^2
      + \left\|A^a\right\|_{H_{s,\delta}}^2\left\|\partial_au
      \right\|_{L^\infty}^2\right\}\\ \nonumber &\leq & C
\left(\left\|\nabla
        A^a\right\|_{L^\infty}+\left\|
        A^a\right\|_{L^\infty}\right)\left\{\left\|u\right\|_{H_{s,\delta}}^2
      +\left\|u\right\|_{H_{s,\delta}}^2\right\}+
    C\left\{\left\|u\right\|_{H_{s,\delta}}^2
      +\left\|A^a\right\|_{H_{s,\delta}}^2\left\|\partial_au
      \right\|_{H_{s-1,\delta+1}}^2\right\}
   \\
    \label{eq:energy-estimates:18}  &\leq& C
    \left\{2\left\|\nabla
        A^a\right\|_{L^\infty}+2\left\|
        A^a\right\|_{L^\infty}+\left\|A^a\right\|_{H_{s,\delta}}^2+1
    \right\}\left\|u\right\|_{H_{s,\delta}}^2, 
  \end{eqnarray}
  here we have applied Corollary \ref{cor:const:1} and Theorem
  \ref{thr:Appendix:2} to $\partial_a u$. Applying again Theorem
  \ref{thr:Appendix:2} to $\|\partial_t u\|_{L^\infty}$ we have
  \begin{eqnarray}\nonumber
    & & \sum_{j=0}^\infty\sum_{k=j-3}^{j+4}2^{(\frac{3}{2}+\delta)2j}
    |E_{3}(j,k)|\nonumber \\ &\leq & C
    \left\{\left\|u\right\|_{H_{s,\delta}}^2 +\left\|
        \partial_t u\right\|_{H_{s-1,\delta}}^2
      \left\|\nabla\left(A^0-I\right)\right\|_{H_{s-1,\delta+1}}^2\right\}
    + 2C\left\{\left\|u\right\|_{H_{s,\delta}}^2
      +\left\| \partial_t u\right\|_{H_{s-1,\delta}}^2\right\} \nonumber \\
    &+& C \left\{\left\|u\right\|_{H_{s,\delta}}^2 +
      \left\|\left(A^0-I\right)\right\|_{H_{s,\delta}}^2\left\|
        \partial_t u\right\|_{L^\infty}^2\right\}\\
    \label{eq:energy-estimates:19}
    & \leq & 2C\left\{\left\|u\right\|_{H_{s,\delta}}^2+\left\|
        \partial_t u\right\|_{H_{s-1,\delta}}^2\left(1+
        \left\|A^0-I\right\|_{H_{s,\delta}}^2\right)\right\}\nonumber
  \end{eqnarray}
  and finally
  \begin{eqnarray}
    \label{eq:energy-estimates:20}
    \sum_{j=0}^\infty2^{(\frac{3}{2}+\delta)2j} |E_{1}(j)| \leq C
    \left\| A^0\right\|_{L^\infty} 
    \left\{\left\|u\right\|_{H_{s,\delta}}^2 +\left\|
        \partial_t u\right\|_{H_{s-1,\delta}}^2\right\}.
  \end{eqnarray}
  Recalling that
  \begin{equation*}
    \langle
    u,\partial_t u\rangle_{s,\delta,A^0} =
    \sum_{j=0}^\infty2^{(\frac{3}{2}+\delta)2j}E(j) 
    =\sum_{j=0}^\infty2^{(\frac{3}{2}+\delta)2j}
    \left\langle\Lambda^s\left(\left(\psi_j^2u\right)_{2^j}\right),
      \left((A^0)_{2^j}\right)
      \Lambda^s\left(\left(\psi_j^2\partial_t u\right)_{2^j}\right)
    \right\rangle_{L^2},
  \end{equation*}
  then inequalities (\ref{eq:energy-estimates:17}),
  (\ref{eq:energy-estimates:18}), (\ref{eq:energy-estimates:19}) and
  (\ref{eq:energy-estimates:20}) imply that

  \begin{equation*}
    \langle u,\partial_t u\rangle_{s,\delta,A^0}\leq
    C\left(\|A^\alpha\|_{L^\infty},\|\nabla
      A^a\|_{L^\infty},\|A^a\|_{H_{s,\delta}},
      \|A^0-I\|_{H_{s,\delta}},\|\partial_t u\|_{H_{s-1,\delta}}\right)
    \left\{\|u\|_{H_{s,\delta}}^2+1\right\}.
  \end{equation*}

  Since $s>\frac{5}{2}$ and $\delta\geq-\frac{3}{2}$ we can use
  Theorem \ref{thr:Appendix:2} (of the Appendix
  \ref{sec:some-properties-h_s}) and bound the norms
  $\|A^\alpha\|_{L^\infty}$ and $\|\nabla A^\alpha\|_{L^\infty}$ by 
  the norms $\|A^0-I\|_{H_{s,\delta}}$ and $\|A^a\|_{H_{s,\delta}}$.
  Thus, combining these bounds with above inequality and inequality
  (\ref{eq:energy-estimates:10}), we have obtained
  \begin{equation}
    \label{eq:energy-estimates:21} \frac{d }{dt} \left\langle u(t),u(t)
    \right\rangle_{s,\delta,A^0}\leq C \left(
      \left\|u(t)\right\|_{H_{s,\delta}}^2+1\right),
  \end{equation}

  where
  $C=C(\|A^a\|_{H_{s,\delta}},\|A^0-I\|_{H_{s,\delta}},\|\partial_t
  u\|_{H_{s-1,\delta}},\|\partial_t A^0\|_{L^\infty})$.  Inserting the
  equivalence of norms $\left\|u\right\|_{H_{s,\delta}}^2\leq \mu
  \left\|u\right\|_{H_{s,\delta,A^0}}^2$ in
  (\ref{eq:energy-estimates:21}), we obtain
  (\ref{eq:energy-estimates:8}) which completes the proof of Lemma
  \ref{lem:energy-estimates:1}.  \BeweisEnde
\end{proof}

We may extend the energy estimate (\ref{eq:energy-estimates:8}) to a
non-homogeneous symmetric hyperbolic systems.
\begin{lem}[An energy estimate]
  \label{lem:energy-estimates:2}
  Let $s>\frac{5}{2}$, $\delta \geq -\frac{3}{2}$,
  $A^\alpha=A^\alpha(t,x)$ be $m\times m$ symmetric matrices such that
  $(A^0(t,\cdot)-I), A^a(t,\cdot)\in H_{s,\delta}$ and $A^0$ satisfies
  (\ref{eq:energy-estimates:3}).  Let $B(t,\cdot)$, $F(t,\cdot)\in
  H_{s,\delta}$. If $u(t,\cdot)$ is a $C^{\infty}_0$ solution of the
  linear hyperbolic system
  \begin{equation}
    \label{eq:energy-estimates:22} A^0(t,x)\partial_t u= \sum_{a=1}^3
    A^a(t,x)\partial_a u +B(t,x)u+F(t,x),
  \end{equation}
  then
  \begin{equation}
    \label{eq:energy-estimates:23}
    \frac{d}{dt}\|u(t)\|_{H_{s,\delta,A^0}}^2\leq C \left(
      \mu\|u(t)\|_{H_{s,\delta,A^0}}^2  +1 \right),
  \end{equation}
  where the constant $C$ depends on $\|A^a\|_{H_{s,\delta}}$,
  $\|A^0-I\|_{H_{s,\delta}}$, $\|\partial_t u\|_{H_{s-1,\delta}}$,
  $\|\partial_t A^0\|_{L^\infty}$, $\left\|B\right\|_{H_{s,\delta}}$
  and $\left\|F\right\|_{H_{s,\delta}}$.
\end{lem}

\begin{proof}[of Lemma \ref{lem:energy-estimates:2}]
  This proof is precisely as the previous one expect the two terms
  \begin{equation}
    \label{eq:energy-estimates:24}
    \left\langle\Lambda^s\left(\left(\psi_j^2u\right)_{2^j}\right),
      \Lambda^s\left(
        \left(\Psi_kB\right)_{2^j}
        \left(\psi_j^2u\right)_{2^j}\right)\right\rangle_{L^2} 
  \end{equation}
  and
  \begin{equation}
    \label{eq:energy-estimates:25}
    \left\langle\Lambda^s\left(\left(\psi_j^2u\right)_{2^j}\right),
      \Lambda^s\left(
        \left(\Psi_k\psi_j^2F\right)_{2^j}\right)\right\rangle_{L^2}
  \end{equation}
  which are added to (\ref{eq:energy-estimates:16}). Using the algebra
  properties of $H^s$ spaces, we see that
  (\ref{eq:energy-estimates:24}) is less than
  \begin{equation*}
    C\left\|\left(\psi_j^2u\right)_{2^j}\right\|_{H^s}^2 \left\|
      \left(\Psi_kB\right)_{2^j}\right\|_{H^s}\leq C
    \left\|\left(\psi_j^2u\right)_{2^j}\right\|_{H^s}^2 \left\|
      B\right\|_{H_{s,\delta}};
  \end{equation*}
  and by Cauchy Schwarz inequality (\ref{eq:energy-estimates:25}) is
  less than
  \begin{equation*}
    C\left\|\left(\psi_j^2u\right)_{2^j}\right\|_{H^s} \left\|
      \left(\Psi_k\psi_j^2F\right)_{2^j}\right\|_{H^s}
    \leq C\frac{1}{2}\left\{ \left\|\left(\psi_j^2u\right)_{2^j}\right\|_{H^s}^2
      +\left\| \left(\psi_j^2F\right)_{2^j}\right\|_{H^s}^2\right\}.
  \end{equation*}
  Multiplying (\ref{eq:energy-estimates:24}) and
  (\ref{eq:energy-estimates:25}) by $2^{(\frac{3}{2}+\delta)2j}$ and
  taking the sum, it results with two quantities less than
  $\|u\|_{H_{s,\delta}}^2\|B\|_{H_{s,\delta}}$ and
  $\left(\|u\|_{H_{s,\delta}}^2+\|F\|_{H_{s,\delta}}^2\right)$
  respectively.  \BeweisEnde
\end{proof}

\subsection{Construction of the iteration}
\label{sec:constr-iter}

We assume $u_0(x)$, the initial value of (\ref{eq:neu-existence:21}),
is contained in $G_1$, where the origin belongs to $G_1$ and $G_1$ is
a compact subset of an open set $G$ of $\mathbb{R}^m$.  In addition we
assume, 
\begin{equation}
  \label{eq:formulation:3}
  \frac{1}{\mu} U^TU\leq U^TA^0(U;t,x)U\leq \mu U^TU \qquad \text{for
all}\quad
  U \in G_2,
\end{equation}
where $G_2$ is a compact set of $G$ such that $G_1\Subset G_2$ and
$\mu>0$.

\begin{rem}
  Since the matrix $A^0$ is continuous, the initial condition
  (\ref{eq:publ-broken:10}) guarantees the existence of a domain
  $G_2$.
\end{rem}

The initial data $u_0$ will be approximated by a sequence
$\{u_0^{k}\}$ of smooth functions with compact support, which
converges to $u_0$ in $H_{s,\delta}(\setR^3)$. It follows from the
embedding $\|v\|_{L^\infty} \leq C
\|v\|_{H_{s,\delta}}$ (see Theorem
    \ref{thr:Appendix:2}, and the density Theorem
\ref{thm:density}, that there is a positive $R$, $u_0^0\in
C_0^\infty(\mathbb{R}^3)$ and $ \{u_0^k\}_{k=1}^\infty\subset
C_0^\infty(\mathbb{R}^3)$ such that
\begin{gather}
  \label{eq:constr-iter:1} 
\|u_0^0\|_{H_{s+1,\delta}}\leq C
  \|u_0\|_{H_{s,\delta}},\\
\label{eq:constr-iter:2}
  \|u_0^0- u_0\|_{H_{s,\delta}}\leq \frac{R}{{\mu 8}},\\
\label{eq:constr-iter:3} \|u- u_0^0\|_{H_{s,\delta}}\leq  {R}
  \Rightarrow u\in G_2
\end{gather}
and
\begin{equation}
  \label{eq:constr-iter:4} \|u_0^k- u_0\|_{H_{s,\delta}} \leq
  2^{-k}\frac{R}{{\mu 8}}.
\end{equation}
The iteration procedure is defined as follows: $u^0(t,x)=u_0^0(x)$ and
$u^{k+1}(t,x)$ is a solution to the linear initial value problem
\begin{equation}
  \label{eq:constr-iter:5} 
  \begin{cases}
      &A^0(u^k;t,x)\partial_t u^{k+1} =  \sum_{a=1}^3
      A^a(u^k;t,x)\partial_a  u^{k+1} + 
      B(u^k;t,x)u^{k+1}+F(u^k;t,x), \\
      &u^{k+1}(0,x)=u_0^{k+1}(x).
\end{cases}
\end{equation}

The existence of $\{u^{k}(t,x)\}\subset C_0^\infty(\mathbb{R}^3)$
follows from:
\begin{thm}[Existence of classical solutions of a linear symmetric
  hyperbolic system] \label{thr:hyperbolic-weighted:1} Let $A^\alpha$,
  $B$ and $F$ be $C^\infty$ functions and $v_0\in
  C_0^\infty(\mathbb{R}^3)$ be an initial datum. Then the linear
  system
  \begin{equation}
    \label{eq:constr-iter:6} 
    \begin{cases}
        & A^0(t,x)\partial_t v =  \sum_{a=1}^3
        A^a(t,x)\partial_a v + B(t,x) v+F(t,x)\\
        &v(0,x)=v_0(x)
    \end{cases}
  \end{equation}
  has a unique solution $v(t,x)$ such that $v(t,x)\in C^\infty$ and it
  has compact support in $\mathbb{R}^3$ for each fixed $t$.
\end{thm}
For the proof we refer to John
\cite{john86:_partial
}. It is evident from Theorem \ref{thr:hyperbolic-weighted:1}
and inequalities (\ref{eq:formulation:3}) and (\ref{eq:constr-iter:3})
that for each $k$: $u^{k}(t,x)$ is well defined, $u^k(t,x)\in
C^\infty$, $u^k(t,x)$
has compact support in $\mathbb{R}^3$ and $u^k(t,x)\in G_2$ for some
positive $T$. We put
\begin{equation}
  \label{eq:constr-iter:7}  T_k=\sup \{T: \sup_{0<t<T}\|u^{k}(t)
  -u_0^0\|_{H_{s,\delta}}\leq R\}.
\end{equation}

Our next issue is to show the existence of $T^*>0 $ such that $T_k
\geq T^*$ for $k=1,2,3,...$

\subsection{Boundedness in the $H_{s,\delta} $ norm}
\label{sec:boundness-norm}

We introduce the following notations: $u(t):=u(t,x)$ and
\begin{equation}
  \label{eq:boundness:1} \n
  u\n_{s,\delta,T}:=\sup\{\|u(t)\|_{H_{s,\delta}}:0\leq t\leq T\}.
\end{equation}

The main result of this subsection is:
\begin{lem}[Boundedness in the $H_{s,\delta}$ norm]
  \label{lem:boundness:1}
  There are positive constants $T^{*}$ and $L$ such that
  \begin{enumerate}
    \item[{\rm(A)}]
    \begin{math}
      \n u^k -u_0^0\n_{s,\delta,T^*} \leq R
    \end{math}
    \item[{\rm(B)}]
    \begin{math}
      \n \partial_t u^k\n_{s-1,\delta+1,T^*}\leq L.
    \end{math}
  \end{enumerate}
\end{lem}

\begin{proof}[of Lemma \ref{lem:boundness:1}]
  We first prove $(B)$. Let
  \begin{equation*}
    G^{k+1} =  \sum_{a=1}^3 A^a(u^k;t,x)\partial_a  u^{k+1} +
    B(u^k;t,x)u^{k+1}+F(u^k;t,x),
  \end{equation*}
  then by Proposition \ref{prop:Properties:3}
and Moser
  type estimate, Theorem \ref{thr:Appendix:5} with Remark
\ref{rem:Properties:2},
  \begin{eqnarray}\nonumber
    & &  \|G^{k+1}\|_{H_{s-1,\delta+1}}\\\nonumber &\leq & \sum_{a=1}^3
    \|A^a(u^k)\|_{H_{s,\delta}} \|\partial_a u^{k}\|_{H_{s-1,\delta+1}}+
    \|B(u^k)\|_{H_{s,\delta}} \| u^{k}\|_{H_{s,\delta}}
    +\|F(u^k)\|_{H_{s,\delta} }\\ \nonumber &\leq &  \sum_{a=1}^3
    \left( C\| u^{k}\|_{H_{s,\delta}}+ \|
      A^a(0)\|_{H_{s,\delta}}\right)\| u^{k}\|_{H_{s,\delta}}+ \left(C\|
      u^{k}\|_{H_{s,\delta}}+ \| B(0)\|_{H_{s,\delta}}\right)\|
    u^{k}\|_{H_{s,\delta}}\\ \label{eq:boundness:2} &+&
    C\|u^k\|_{H_{s,\delta}}+\| F(0)\|_{H_{s,\delta}}.
  \end{eqnarray}
  The constant $C$ here depends on $\|A^a\|_{C^{N+1}(G_2)}$,
  $\|B\|_{C^{N+1}(G_2)}$, $\|F\|_{C^{N+1}(G_2)}$ and
  $\|u^k\|_{L^\infty}$ (see (\ref{eq:13})).  Since
  \begin{equation}
    \label{eq:boundness:3}
    \|u^k(t)\|_{H_{s,\delta}}\leq
    \|u^k(t)-u_0^0\|_{H_{s,\delta}}+ \|u^0_0\|_{H_{s,\delta}},
  \end{equation}
  the induction assumption (A) and  inequality
  (\ref{eq:constr-iter:1}) imply that $\|u^k\|_{H_{s,\delta}}\leq R
  +C\|u_0\|_{H_{s,\delta}}$. Using the embedding
  $\|u^k\|_{L^\infty}\leq C \|u^k\|_{H_{s,\delta}}$, we see that
  $\|G^{k+1}\|_{H_{s-1,\delta+1}}\leq C_1(R)$, where the constant
  $C_1(R)$ depends upon $R$, condition (H3) and the initial data, but
  it is independent of $k$. From (\ref{eq:constr-iter:5}) we have
  \begin{equation*}
    \partial_t u^{k+1}=\left(A^0(u^k;t,x)\right)^{-1}G^{k+1}=
    \left(\left(A^0(u^k;t,x)\right)^{-1}-I\right)G^{k+1}+G^{k+1}.
  \end{equation*}
  Repeating same arguments as above, we conclude that
  \begin{equation*}
    \|\left(\left(A^0(u^k;t,x)\right)^{-1} -
      I\right)G^{k+1}\|_{H_{s-1,\delta+1}}\leq 
    C_2(R)
  \end{equation*}
  and the constant $C_{2}(R)$ does not depend on $k$. We take
  $L=C_1(R)+C_2(R)$.  Here we have used Moser estimate with
  $F(u)=A^{-1}(u)-I$, and the formula $
  \frac{\partial{A^{-1}(u)}}{\partial u}={A^{-1}(u)} \frac{\partial
    {A(u)}}{\partial u}{A^{-1}(u)}$.  Thus the constant $C(R)$ depends
  on $\|A^0\|_{C^{N+2}(G_2)}$ and $\mu$.

  We turn now to show (A). Let $V^{k+1} = u^{k+1}-u^0_0$, then
  inserting it in the equation (\ref{eq:constr-iter:5}) we have
  obtained
  \begin{eqnarray}\nonumber
    A^0(u^{k};t,x)\partial_t V^{k+1}&=&
    A^0(u^{k};t,x)u^{k+1}_t = \sum_{a=1}^3
    A^a(u^{k};t,x)\partial_a u^{k+1} + B(u^{k};t,x)u^k+F(u^{k};t,x)\\
    \label{eq:boundness:4} &=&  \sum_{a=1}^3 A^a(u^{k};t,x)\partial_a
    V^{k+1} +B(u^{k};t,x)V^{k+1}+F(u^{k};t,x)\\\nonumber&+&
    \sum_{a=1}^3 A^a(u^{k};t,x)\partial_a u_0^0 +B(u^{k};t,x)u_0^0
  \end{eqnarray}
  and $V^{k+1}(0,x)=u_0^{k+1}(0,x)-u_0^0(x)$. At this stage we would
  like employ the energy estimate Lemma \ref{lem:energy-estimates:2}.
  Due the  fact that the coefficients of (\ref{eq:boundness:4})
  depend on $u^k$, it is obligatory to control the constant of
  (\ref{eq:energy-estimates:23}) in terms of $\|u^k\|_{H_{s,\delta}}$.
  Therefore we need to bound $\|\left(
    A^0(u^{k};t,x)-I\right)\|_{H_{s,\delta}}$, $ \|
  A^a(u^{k};t,x)\|_{H_{s,\delta}}$, $\|B(u^{k};t,x)\|_{H_{s,\delta}}$,
  $\|F(u^{k};t,x)\|_{H_{s,\delta}}$ and $\|\frac{\partial}{\partial
    t}A^0(u^{k};t,x)\|_{L^\infty}$ by $\|u^k\|_{H_{s,\delta}}$. The
  first four are similar, so take for example $ A^a(u^{k};t,x)$: We
  use assumption (H2), Moser type estimate, Theorem
\ref{thr:Appendix:5} with Remark
  \ref{rem:Properties:2} and get that
  \begin{equation}
    \label{eq:boundness:6}
    \|A^a(u^{k};t,x)\|_{H_{s,\delta}}\leq C \left\{\|A^a\|_{C^{N+1}(G_2)}
      \left(1+\|u^k\|_{L^\infty}^N\right)\right\}\|
    u^k\|_{H_{s,\delta}}+\|A^a(0;t,\cdot)\|_{H_{s,\delta}}.
  \end{equation}
  For the last one we have
  \begin{equation}
    \label{eq:boundness:5}
    \begin{split}
      & \|\frac{\partial}{\partial
        t}A^0(u^{k};t,x)\|_{L^\infty}=\|\frac{\partial}{\partial
        u}A^0(u^{k};t,x)\partial_t u^k(t,x)+ \partial_t
      A^0(u^{k};t,x)\|_{L^\infty}\\
      \leq & \|\frac{\partial}{\partial u}A^0(u^{k};t,x)\|_{L^\infty}
      \|\partial_t u^k(t,x)\|_{L^\infty}+ \|\partial_t
      A^0(u^{k};t,x)\|_{L^\infty} \\
      \leq C & \|\frac{\partial}{\partial
        u}A^0(u^{k};t,x)\|_{L^\infty} \|\partial_t
      u^k(t,x)\|_{H_{s-1,\delta+1}}+ \|\partial_t
      A^0(u^{k};t,x)\|_{L^\infty}.
    \end{split}
  \end{equation}
  We conclude from inequalities (\ref{eq:boundness:6}) and
  (\ref{eq:boundness:5}), Theorem \ref{thr:Appendix:2}, the inductions
hypothesis (A) and (B),
  (\ref{eq:constr-iter:3}) and (H4) that the constant of
  (\ref{eq:energy-estimates:23}) depends on $R$, $L$,
  $\|u_0\|_{H_{s,\delta}}$ and the $H_{s,\delta}$-norm of the
  coefficients, but it is independent of $k$.  Hence, the energy
  estimate Lemma \ref{lem:energy-estimates:2} implies that
  \begin{equation}
    \label{eq:wellposs-einstein-euler:1}
    \frac{d}{dt}\|V^{k+1}(t)\|_{H_{s,\delta,A^0}}^2\leq
    C(R,L) \left( \mu\|V^{k+1}(t)\|_{H_{s,\delta,A^0}}^2  +1 \right),
  \end{equation}

  Applying Gronwall's inequality, (\ref{eq:constr-iter:2}),
  (\ref{eq:constr-iter:4}) and the equivalence
  (\ref{eq:energy-estimates:5}) results in
  \begin{eqnarray}\nonumber
    \n V^{k+1}\n_{s,\delta,T}^{2} &\leq &
    \mu e^{C(R,L)\mu T} \left( \mu \|V^{k+1}(0)\|^{2}_{H_{s,\delta}}
      +C(R,L)T\right) \\\nonumber & = &\mu e^{C(R,L)\mu T} \left( \mu
      \|u^{k+1}_0-u_0^0\|^{2}_{H_{s,\delta}}
+C(R,L)T\right)\\\nonumber
    &\leq & \mu e^{C(R,L)\mu T} \left( \mu\left(
        \|u^{k+1}_0-u_0\|^{2}_{H_{s,\delta}}+
        \|u_0^0-u_0\|^{2}_{H_{s,\delta}}\right) +C(R,L)T \right)
    \\\label{eq:boundness:7}
    &\leq & e^{C(R,L)\mu T} \left( 2\mu^2 \left(
        \frac{R}{\mu 8}\right)^2 +\mu C(R,L) T \right).
  \end{eqnarray}
  Therefore $ \n V^{k+1}\n_{s,\delta,T}^{2} \leq R^2$, if
$\mu C(R,L)T\leq \min\{\log 2,\frac{15}{32}R^2\}$.
  Thus taking $T^*=(\mu  C(R,L))^{-1}\min\{\log
2,\frac{15}{32}R^2\}$ proves $(A)$ and
  completes the proof of Lemma \ref{lem:boundness:1}.  \BeweisEnde
\end{proof}

\subsection{Contraction in the lower norm}
\label{sec:contr-lower-norm}

We show here that $\left\{ u^k \right\}$ has a contraction property in
$\|\cdot \|_{0,\delta,T^{**}} $ for a positive $T^{**}$ (see
(\ref{eq:boundness:1})). In order to
achieved it we need an energy estimate in $H_{0,\delta}\backsim
L_\delta^2$.  For that purpose we introduce the below inner-product in
$L_\delta^2$: for two vectors $u $ and $v$ in $L_\delta^2$, we set
\begin{equation}
  \label{eq:contr-lower-norm:1}
  \langle u,v\rangle_{L^2_\delta,A^0}=\int\left(1+|x|\right)^{2\delta}\left(
    u^TA^0v\right)dx,
\end{equation}
and its associated norm $\|u\|^2_{L_\delta^2,A^0}=\langle
u,u\rangle_{L^2_\delta,A^0}$. The ordinary norm is denoted by
$\|u\|^2_{L_\delta^2}=\langle u,u\rangle_{L^2_\delta,I}$.  Since $A^0$
satisfies (\ref{eq:formulation:3}),
\begin{equation}
  \label{eq:contr-lower-norm:9}
  \frac{1}{\mu }\|u\|_{L_\delta}^2\leq \langle u,u\rangle_{L^2_\delta,A^0}
  \leq \mu \|u\|_{L_\delta}^2,
\end{equation}
and hence by Theorem we obtain \ref{thm:a1}, 
$\|u\|^2_{L_\delta^2,A^0}\simeq
\|u\|_{H_{0,\delta}}$.

\begin{prop}[Energy estimate in $L_\delta^2$]
  \label{prop:contr-lower-norm:1} Suppose $u$ satisfies the linear
  hyperbolic system (\ref{eq:constr-iter:6}), then
  \begin{equation}
    \label{eq:contr-lower-norm:3} \frac{d}{dt}  \langle
    u(t),u(t)\rangle_{L^2_\delta,A^0}\leq \mu C\langle
    u(t),u(t)\rangle_{L^2_\delta,A^0}+\|F\|^2_{L_\delta^2} ,
  \end{equation}
  where $C=C(\|\partial_t A^0\|_{L^\infty}, \|A^a\|_{L^\infty},
  \|B\|_{L^\infty}, \|\partial_a A^a\|_{L^\infty})$.
\end{prop}
\begin{proof}[of Proposition \ref{prop:contr-lower-norm:1}]
  Taking the derivative of (\ref{eq:contr-lower-norm:1}) with respect
  to $t$, we get
  \begin{eqnarray}\nonumber
    \frac{d}{dt}  \langle u,u\rangle_{L^2_\delta,A^0} &=&  2\langle
    u,\partial_t u\rangle_{L^2_\delta,A^0}+
    \int\left(1+|x|\right)^{2\delta}\left(
      u^T\partial_t A^0 u\right)dx\\\nonumber & = &
    2\sum_{a=1}^3\int\left(1+|x|\right)^{2\delta}\left(
      u^TA^a \partial_a u\right)dx+ 2\int\left(1+|x|\right)^{2\delta}\left(
      u^TB u\right)dx\\\nonumber & + &
    2\int\left(1+|x|\right)^{2\delta}\left(
      u^TF\right)dx+ \int\left(1+|x|\right)^{2\delta}\left(
      u^T\partial_t A^0 u\right)dx\\
    \nonumber
    & = & 2\sum_{a=1}^3 L_{1,a}+2L_2+2L_3+L_4.
  \end{eqnarray}
  Clearly,
  \begin{equation*}
    |L_2|\leq\| B\|_{L^\infty}\int\left(1+|x|\right)^{2\delta}
    |u|^2dx\leq \|B\|_{L^\infty}\|u\|^2_{L_\delta^2}
  \end{equation*}
  and in a similar way we obtain the estimates of $L_4$ while by
  Cauchy-Schwarz inequality,
  \begin{equation*}
    |L_3|\leq \|u\|_{L_{\delta}^2}\|F\|_{L_{\delta}^2}\leq
    \frac{1}{2}\left(\|u\|_{L_{\delta}^2}^2+\|F\|_{L_{\delta}^2}^2\right).
  \end{equation*}
  Now,
  \begin{eqnarray}\nonumber
    0 &=&  \int\partial_a\left(\left(1+|x|\right)^{2\delta}\left(
        u^TA^a u\right)\right)dx\\\nonumber & = &
    2\delta\int\left(1+|x|\right)^{2\delta-1}\frac{x_a}{|x|}\left(
      u^TA^a  u\right)dx+ \int\left(1+|x|\right)^{2\delta}\left((\partial_a
      u)^TA^au\right)dx\\\nonumber & + &
    \int\left(1+|x|\right)^{2\delta}\left(
      u^T\partial_a A^a u\right)dx+ \int\left(1+|x|\right)^{2\delta}\left(
      u^TA^a\partial_a u\right)dx,
  \end{eqnarray}
  and since $A^0$ is symmetric, the second and the fourth terms are
  equal to $L_{1,a}$. Hence,
  \begin{eqnarray}\nonumber
    2|L_{1,a}| &\leq& 2\delta\int\left(1+|x|\right)^{2\delta}\frac
    {|A^0|} {1+|x|}\left( |u|^2\right)dx+
    \int\left(1+|x|\right)^{2\delta}|\partial_a A^a| |u|^2dx\\\nonumber
    & \leq &\left( \|A^a\|_{L^\infty}+\|\partial_a
      A\|_{L^\infty}\right)\|u\|^2_{L_\delta^2}.
  \end{eqnarray}
  \BeweisEnde
\end{proof}

In order to proof the contraction we shall also need the following
proposition.

\begin{prop}[Difference estimate in $L_\delta^2$]
  \label{prop:contr-lower-norm:2} Let $G:\mathbb{R}^m\to\mathbb{R}^m$
  be a $C^1$ mapping. Then
  \begin{equation}
    \|G(u)-G(v)\|^2_{L_\delta^2}\leq \|\nabla
    G\|^2_{L^\infty}\|u-v\|^2_{L_\delta^2}.
  \end{equation}
\end{prop}

\begin{proof}[of Proposition \ref{prop:contr-lower-norm:2}]
  \begin{eqnarray}\nonumber
    & & \|G(u)-G(v)\|^2_{L_\delta^2} =
    \int\left(1+|x|\right)^{2\delta}\left(G(u)-G(v)\right)^2dx 
    \\\nonumber & = &
    \int\left(1+|x|\right)^{2\delta}\left(\int_0^1  \nabla G\left(s
        u+(1-s)v\right)
      (u-v)ds\right)^2dx\leq\|\nabla G\|^2_{L^\infty}\|u-v\|^2_{L_\delta^2}.
  \end{eqnarray}
  \BeweisEnde
\end{proof}

\begin{lem}[Contraction in a lower norm]
  \label{lem:contr-lower-norm1}
  There is a positive $T^{**}$, $0<\Lambda <1$ and a positive sequence
  $\left\{ \beta_k \right\}$ with $\sum \beta_k<\infty$ such that
  \begin{equation}
    \label{eq:neu-existence:15}
    \n u^{k+1} -u^k\n_{0,\delta,T^{**}} \leq \Lambda
    \n u^{k} -u^{k-1}\n_{0,\delta,T^{**}} + \beta_k.
  \end{equation}
  Here $ \n u\n_{0,\delta,T^{**}}=\sup\{\|u(t)\|_{H_{0,\delta}}: 0\leq
  t\leq T^{**}\}$.
\end{lem}

\begin{proof}[of Lemma \ref{lem:contr-lower-norm1}]
  Since $u^k$ satisfies equation (\ref{eq:constr-iter:5}), the
  difference $\left[u^{k+1}- u^k\right]$ will satisfy
  \begin{equation}
    \label{eq:contr-lower-norm:2}
\begin{split}
   & A^0(u^k;t,x)\partial_t\left[u^{k+1}- u^k\right]  = \sum_{a=1}^3
    A^a(u^k;t,x)\partial_a\left[u^{k+1}- u^k\right] \\ &+
B(u^k;t,x)\left[u^{k+1}-
      u^k\right]+F^k,
\end{split}
  \end{equation}
  where
  \begin{eqnarray}\nonumber
    F^k & = & -\left[A^0(u^k;t,x)-A^0(u^{k-1};t,x)\right]\partial_t
u^k+\sum_{a=1}^3
    \left[A^a(u^k;t,x)-A^a(u^{k-1};t,x)\right]\partial_a
u^k\\\nonumber &+ &
    \left[B(u^k;t,x)-B(u^{k-1};t,x)\right] u^k
    +\left[F(u^k;t,x)-F(u^{k-1};t,x)\right].
  \end{eqnarray}
  Applying Proposition \ref{prop:contr-lower-norm:1} to equation
  (\ref{eq:contr-lower-norm:2}) above we have
  \begin{equation}
    \label{eq:contr-lower-norm:4} \frac{d}{dt}  \langle
    \left[u^{k+1}-u^k\right],
    \left[u^{k+1}-u^k\right]\rangle_{L^2_\delta,A^0}\leq
    \mu C\langle
    \left[u^{k+1}-u^k\right],
    \left[u^{k+1}-u^k\right]\rangle_{L^2_\delta,A^0}+\|F^k\|^2_{L_\delta^2}. 
  \end{equation}
  Thus Gronwall's inequality yields,
  \begin{equation}
    \label{eq:contr-lower-norm:5}
    \|\left[u^{k+1}(t)-u^k(t)\right]\|_{L_\delta^2,A^0}^2
    \leq e^{\mu C
      t}\left[\|\left[u^{k+1}(0)-u^k(0)\right]\|_{L_\delta^2,A^0}^2+\int_0^t
      \|F^k(s)\|_{L_\delta^2}ds\right].
  \end{equation}
  The constant $C$ in inequalities (\ref{eq:contr-lower-norm:4}) and
  (\ref{eq:contr-lower-norm:5}) depends on
  $\|A^a(u^k;t,x)\|_{L^\infty}$, $\|B(u^k;t,x)\|_{L^\infty}$,
  $\|\partial_t(A^0(u^k;t,x))\|_{L^\infty}$ and $\|\partial_a
  \left(A^a(u^k;t,x)\right)\|_{L^\infty}$. The first two of them are
  bounded by a constant independent of $k$, since it follows from (A)
  of Lemma \ref{lem:boundness:1} that $u^k\in G_2$. The estimation of
  $\|\partial_t(A^0(u^k;t,x))\|_{L^\infty}$ is done in
  (\ref{eq:boundness:5}) and for the last one, since
  $s-1>\frac{3}{2}$, we can use Theorem \ref{thr:Appendix:2} and
Corollary \ref{cor:const:1}  and get 
  \begin{equation*}
    \begin{split}
      & \|\partial_a \left(A^a(u^k;t,x)\right)\|_{L^\infty}\leq
      \|\frac{\partial}{\partial u} A^a(u^k;t,x)\partial_a
      u^k\|_{L^\infty}+\|\partial_a A^a(u^k;t,x))\|_{L^\infty} \\
      \leq C & \|\frac{\partial}{\partial u} A^a(u^k;t,x)\|_{L^\infty}
      \|\partial_a u^k\|_{H_{s-1,\delta+1}}+\|\partial_a
      A^a(u^k;t,x))\|_{L^\infty} \\
      \leq C & \|\frac{\partial}{\partial u} A^a(u^k;t,x)\|_{L^\infty}
      \|u^k\|_{H_{s,\delta}}+\|\partial_a A^a(u^k;t,x))\|_{L^\infty}
    \end{split}
  \end{equation*} 
  Lemma \ref{lem:boundness:1} (A) implies that
  $\|u^k\|_{H_{s,\delta}}$ is bounded and $u^k\in G_2$, therefore the
  above inequality shows that $\|\partial_a
  \left(A^a(u^k;t,x)\right)\|_{L^\infty}$ is bounded by a constant
  independent of $k$. From Proposition \ref{prop:contr-lower-norm:2}
  we obtain
  \begin{equation}
 \label{eq:contr-lower-norm:6}
\begin{split}
    \left\|F^k\right\|_{L^2_{\delta}}^2 
      & \leq  2\left\{  \|\nabla
      A^0\|_{L^\infty(G_2)}^2\left\|\partial_t
        u^k\right\|_{L^\infty}^2 + \sum_{a=1}^3
      \|\nabla A^a\|_{L^\infty(G_2)}^2
      \left\|\partial_a u^k\right\|_{L^\infty}^2\right. \\
     &+ \left. \left\|\nabla
        B\right\|_{L^\infty(G_2)}^2\left\|u^k\right\|_{L^\infty}^2
      + \|\nabla
      F\|_{L^\infty(G_2)}^2\right\}
    \left\|[u^k-u^{k-1}]\right\|_{L^2_{\delta}}^2,
\end{split} 
 \end{equation}
  here $\nabla$ is the gradient with respect to $u$.  Since by Theorem
\ref{thr:Appendix:2} and Corollary \ref{cor:const:1}, 
  $\left\|\partial_t u^k\right\|_{L^\infty}\leq C \left\|\partial_t
    u^k\right\|_{H_{s-1,\delta+1}}$, $\left\|\partial_a
    u^k\right\|_{L^\infty}\leq C
  \left\|u^k\right\|_{H_{s-1,\delta+1}}\leq C
  \left\|u^k\right\|_{H_{s,\delta}}$ and $\left\|
    u^k\right\|_{L^\infty}\leq C \left\|u^k\right\|_{H_{s,\delta}}$,
  it follows from (\ref{eq:contr-lower-norm:6}) and Lemma
  \ref{lem:boundness:1} that
  \begin{equation}
    \label{eq:contr-lower-norm:7}
    \left\|F^k(s)\right\|_{L^2_{\delta}}^2\leq C_1
    \left\|[u^k(s)-u^{k-1}(s)]\right\|_{L^2_{\delta}}^2,
  \end{equation}
  where the constant $C_1 $ depends upon $R$ and $L$ of Lemma
  \ref{lem:boundness:1}, but it is independent of $k$. By the
  equivalence $\|u\|^2_{L_\delta^2,A^0}\simeq \|u\|_{H_{0,\delta}}$,
  (\ref{eq:contr-lower-norm:7}) and (\ref{eq:contr-lower-norm:5})
  above, we conclude that
  \begin{eqnarray}
    \nonumber &
    &\left\|[u^{k+1}(t)-u^{k}(t)\right\|_{H_{0,\delta}}^2
    \\\nonumber  &\leq& C_2e^{\mu C t}\left[
      \left\|[u^{k+1}(0)-u^{k}(0)]\right\|_{H_{0,\delta}}^2+C_1\int_0^t
      \left\|[u^{k}(s)-u^{k-1}(s)\right\|_{H_{0,\delta}}^2\right]\\
    \nonumber & \leq& C_2e^{\mu C
      t}\left[\left\|[u^{k+1}(0)-u^{k}(0)]\right\|_{H_{0,\delta}}^2+C_1t
      \sup_{0\leq s\leq t}
      \left\|[u^{k}(s)-u^{k-1}(s)\right\|_{H_{0,\delta}}^2\right],
  \end{eqnarray}
  where $C_1$, $C_2$ and $C$ do not depend on $k$. Hence 
  \begin{eqnarray}
    \nonumber & &\n[u^{k+1}-u^{k}\n_{0,\delta,T^{**}}
    \\\nonumber  &\leq& \sqrt{\frac{C_2 e^{\mu C T^{**}}}{2}}
  \left\|[u^{k+1}(0)-u^{k}(0)]\right\|_{H_{0,\delta}}+\sqrt{\frac{C_1
e^{\mu C T^{**}}}{2}} \n[u^{k}-u^{k-1}]\n_{_0,\delta,T^{**}}.
  \end{eqnarray}
  Thus, taking $T^{**}$ sufficiently small so that
  $\Lambda:=\sqrt{\frac{C_1 e^{\mu C T^{**}}}{2}}<1$ and
  putting $\beta_k=\sqrt{\frac{C_2 e^{\mu C T^{**}}}{2}}
  \left\|[u^{k+1}(0)-u^{k}(0)]\right\|_{H_{0,\delta}}$ completes the
  proof of the Lemma.  \BeweisEnde
\end{proof}

Lemma \ref{lem:contr-lower-norm1} implies that $\{u^k\}$ is a Cauchy
sequence in $C([0,T^{**}],H_{0,\delta})$. Combing this with the
intermediate estimates $\|u\|_{H_{s',\delta}}\leq \|u
\|_{H_{s,\delta}}^{\frac{s'}{s}}\| u\|_{H_{0,\delta}}^{1-\frac
  {s'}{s}} $ (see Proposition \ref{prop:Properties:5}) and
Lemma \ref{lem:boundness:1} (A), we conclude that $\{u^k\}$ is a
Cauchy sequence in $C([0,T^{**}],H_{s',\delta})$ for any $s'<s$.
Therefore there is a unique $u\in C([0,T^{**}],H_{s',\delta})$ such
that
\begin{equation}
  \label{eq:contr-lower-norm:8} \n u^k-u\n_{s',\delta,T^{**}}\to 0
  \qquad\text{for any }\qquad s'<s.
\end{equation}
Taking $\frac{5}{2}<s'<s$ and utilizing the embedding Theorem
\ref{thr:Appendix:2}, we have
\begin{equation*}
  u^k \to u  \qquad\text{in }\qquad
  C\left([0,T^{**}],C^1_\beta(\mathbb{R}^3)\right) 
  \qquad\text{for any  }\quad\beta\leq\delta+\frac{3}{2},
\end{equation*}
where $C^1_\beta(\mathbb{R}^3)$ is the class for which the norm
\begin{equation*}
  \sup_{\mathbb{R}^3}\left( (1+|x|)^\beta |u(x)|+\sum_{a=1}^3
    (1+|x|)^{\beta+1}|\partial_a u(x)|\right)  
\end{equation*} 
is finite.  From (\ref{eq:constr-iter:5})
\begin{equation*}
  \partial_t u^{k+1} =
  \left(A^0(u^k;t,x)\right)^{-1}\left[\sum_{a=1}^3
    A^a(u^k;t,x)\partial_a  u^{k+1} + 
    B(u^k;t,x)u^{k+1}+F(u^k;t,x)\right],
\end{equation*}
therefore by Corollary \ref{cor:appendix:3}  and Proposition
\ref{prop:Properties:5}, we obtain $\partial_t u^k\to
\partial_t u$ in $H_{s-1,\delta+1}$. Hence
\begin{equation*}
  \partial_t u^k \to \partial_t u  \qquad\text{in }\qquad
  C\left([0,T^{**}],C_{\beta+1}(\mathbb{R}^3)\right) 
  \qquad\text{for any  }\quad\beta\leq\delta+\frac{3}{2}.
\end{equation*}
Thus $u\in C^1\left(\mathbb{R}^3\times[0,T^{**}]\right)$ is a
classical solution of the nonlinear system
(\ref{eq:neu-existence:21}).  Moreover, it follows from Lemma
\ref{lem:boundness:1} (B) that $u\in {\rm Lip}\left([0,T^{**}],
  H_{s-1,\delta+1}\right)$. Our next task is to show that $u^k$
converges weakly to $u$ in $H_{s,\delta}$.

\subsection{Weak Convergence}
\label{sec:Weak_Convergence}

We first define the standard inner-product on $H_{s,\delta}$. For
two vector valued functions $v,\phi\in H^s$, the expression
\begin{equation*}
    \left\langle v,\phi
    \right\rangle_{s} =\int \left(\Lambda^s v\right)^T\Lambda^s
    (\phi)dx=\int (1+|\xi|^2)^s
    \hat{v}^T\overline{\hat{\phi}}d\xi.
  \end{equation*}
is an inner-product on $H^s$. Utilizing Definition
\ref{def:weighted:3} and Corollary \ref{cor:app11} we see that 
\begin{equation*}
  \langle v,\phi\rangle_{s,\delta}= \sum_j 2^{( \frac{3}{2} +
\delta)2j}
  \left\langle(\psi_j^2v)_{(2^j)}
    ,(\psi_j^2 \phi)_{(2^j)}\right\rangle_s
\end{equation*}
is inner-product on $H_{s,\delta}$. This definition coincides with
Definition \ref{def:inner_product} in the case where $A$ is the
identity matrix.
\begin{lem}[Weak Convergence]
  \label{lem:Weak_Convergence:1}
  For any $\phi\in H_{s,\delta}$, we have
  \begin{equation}
    \label{eq:Weak_Convergence:1}
    \lim_k \left\langle u^{k}(t) ,\phi\right\rangle_{s,\delta} = \left\langle
      u(t) ,\phi\right\rangle_{s,\delta}
  \end{equation}
  uniformly for $0\leq t\leq T^{**}$. Consequently
  \begin{equation}
    \label{eq:Weak_Convergence:2}
    \|u(t) \|_{H_{s,\delta}} \leq \lim_k\inf \|u^k(t)\|_{H_{s,\delta}}
  \end{equation}
  and hence the solution $u $ of the initial value problem
  (\ref{eq:neu-existence:21}) belongs to $L^\infty\left([0,T^{**}],
    H_{s,\delta}\right)$.
\end{lem}

 In order to show Lemma \ref{lem:Weak_Convergence:1} we need the below
property.

\begin{prop}
  \label{prop:Weak_Convergence:2}
  Let $s< \frac{s^{\prime}+s^{\prime\prime}}{2}$, $v\in
  H_{s^{\prime},\delta }$, $\phi\in H_{{s^{\prime\prime}},\delta}$.
Then
  we have
  \begin{equation}
    \label{eq:Weak_Convergence:3} \left| \left\langle v,\phi
      \right\rangle_{s,\delta}\right| \leq
    \|v\|_{H_{{s^{\prime}},\delta}}
    \|\phi\|_{H_{{s^{\prime\prime}},\delta}}.
  \end{equation}

\end{prop}
\begin{proof}[of Proposition \ref{prop:Weak_Convergence:2}]
  Elementary arguments show that
  \begin{equation*}
    \left|   \left\langle v,\phi
      \right\rangle_{s}\right| \leq \|v\|_{H^{{s^{\prime}}}}
    \|\phi\|_{H{{s^{\prime\prime}}}}.
  \end{equation*}
  Applying it term-wise and using the Cauchy-Schwarz inequality we
  have
  \begin{eqnarray}
    \nonumber  \left|   \left\langle v,\phi
      \right\rangle_{s,\delta}\right|& \leq   & \sum_{j=0}^\infty
    2^{\left( \frac{3}{2} +\delta\right)2j} \left| \left\langle \left(
          \psi_j^2 v \right)_{2^j},\left( \psi_j^2 \phi
        \right)_{2^j} \right\rangle_s \right|\\ \nonumber
    &\leq & \sum_{j=0}^\infty \left(2^{\left(  \frac{3}{2}
          +\delta\right)j} \left\| \left( \psi_j^2 v
        \right)_{2^j}\right\|_{H^{s^{\prime}}}\right) \left(2^{\left(
          \frac{3}{2} +\delta\right)j}
      \left\|  \left( \psi_j^2 \phi  \right)_{2^j}
      \right\|_{H^{s^{\prime\prime}}}\right)\\
    \nonumber  &\leq & \left( \sum_{j=0}^\infty 2^{\left(  \frac{3}{2}
          +\delta\right)2j} \left\|  \left( \psi_j^2 v
        \right)_{2^j}\right\|_{H^{s^{\prime}}}^{2}
    \right)^{\frac{1}{2}}
    \left( \sum_{j=0}^\infty 2^{\left(  \frac{3}{2} +\delta\right)2j}
      \left\| \left( \psi_j^2 \phi
        \right)_{2^j}\right\|_{H^{s^{\prime\prime}}}^{2}
    \right)^{\frac{1}{2}}\\\nonumber
    &=& \|v\|_{H_{{s^{\prime}},\delta}}
    \|\phi\|_{H_{{s^{\prime\prime}},\delta}}
  \end{eqnarray}
  \BeweisEnde
\end{proof}

\begin{proof}[of Lemma \ref{lem:Weak_Convergence:1}]
  Take $s^{\prime}$ and $s^{\prime\prime}$ such that
  $s^{\prime}<s<s^{\prime\prime}$ and $s<\frac{s^{\prime} +
    s^{\prime\prime}}{2}$.  For a given $\phi\in H_{s,\delta}$ and
  positive $\epsilon$, we may find by Theorem \ref{thm:density} (b),
  $\tilde{\phi}\in H_{s^{\prime\prime},\delta}$ such that
  \begin{equation}
    \label{eq:Weak_Convergence:4} \|\phi
    -\tilde{\phi}\|_{H_{s,\delta}}\leq \frac{\epsilon}{2R}\quad
    \text{and}\quad \|\tilde{\phi}\|_{H_{s^{\prime\prime},\delta}}\leq
    C(\epsilon) \|\phi\|_{H_{s,\delta}},
  \end{equation}
  where $R$ is the positive number appearing in
  (\ref{eq:constr-iter:3}). Now,
  \begin{eqnarray*}
    \label{eq:appendix:39} \left\langle u^k(t)-u(t),\phi
    \right\rangle_{s,\delta} & = & \left\langle u^k(t)-u(t),\tilde
      \phi \right\rangle_{s,\delta} + \left\langle  u^k(t)-u(t),\left(
        \phi - \tilde \phi
      \right)\right\rangle_{s,\delta} \\
    &=& I_k + II_k.
  \end{eqnarray*}
  Therefore Proposition \ref{prop:Weak_Convergence:2},
  (\ref{eq:Weak_Convergence:4}) and (\ref{eq:contr-lower-norm:8})
  imply that
  \begin{eqnarray*}
    \label{eq:appendix:42} |I_k| & \leq  & \|u^k(t) -u(t)
    \|_{H_{s^{\prime},\delta}} \|\tilde
    \phi\|_{H_{s^{\prime\prime},\delta}} \leq \|u^k(t) -u(t)
    \|_{H_{s^{\prime},\delta}} C(\epsilon)
    \|\phi\|_{H_{s,\delta}}\to 0.
  \end{eqnarray*}
  While in the second estimate we use Lemma \ref{lem:boundness:1} (A)
  and get
  \begin{eqnarray*}
    \label{eq:appendix:43}
    |II_k| & \leq  & \|u^k(t) -u(t) \|_{H_{s,\delta}} \|\phi- \tilde
    \phi\|_{H_{s,\delta}}\\ 
    &\leq & \left(  \|u^k(t) -u^0_0 \|_{H_{{{s}},\delta}}
      +\|u(t) -u^0_0
      \|_{H_{s,\delta}} \right)
    \|\phi- \tilde \phi\|_{H_{s,\delta}}\leq \frac{2 R \epsilon }{2R}={\epsilon}.
  \end{eqnarray*}
  Thus,
  \begin{equation*}
    \limsup_k\left|\left\langle u^k(t)-u(t),\phi
      \right\rangle_{s,\delta}\right|\leq\epsilon
  \end{equation*}
  which completes the proof of the limit
  (\ref{eq:Weak_Convergence:1}).  \BeweisEnde
\end{proof}

For each $k$, $\left\langle u^{k}(t) ,\phi\right\rangle_{s,\delta}$ is
continuous for $ t\in[0,T^{**}]$ and by Lemma
\ref{lem:Weak_Convergence:1} it convergences uniformly to
$\left\langle u(t) ,\phi\right\rangle_{s,\delta}$, hence $\left\langle
  u(t) ,\phi\right\rangle_{s,\delta}$ is a continuous function of $t$
for any $\phi\in H_{s,\delta}$ and we have obtained the following:
\begin{thm}[Existence]
  \label{existence} Under conditions 
(H1)-(H4) and
  (\ref{eq:formulation:3}) there is $u\in
  C^1\left(\mathbb{R}^3\times[0,T^{**}]\right)$ a classical solution
  to the hyperbolic system (\ref{eq:neu-existence:21}) such that
  $u(t,x)\in \overline{G_2}$ and
  \begin{equation}
    \label{reg} u \in L^\infty\left([0,T^{**}],H_{s,\delta}\right)\cap
    C_w\left([0,T^{**}],H_{s,\delta}\right)\cap {\rm
      Lip}\left([0,T^{**}],H_{s-1,\delta+1}\right),
  \end{equation}
  where $C_w$ means continuous in the weak topology of $H_{s,\delta}$.

\end{thm}

\subsection{Well-posedness }
\label{sec:Well-posedness}

In this subsection we well prove continuity in ${H_{s,\delta}}$-norm
and
uniqueness.
\begin{thm}[Uniqueness]
  \label{Uniqueness} Assume conditions (H1)-(H4) and
  (\ref{eq:formulation:3}) hold.  If $u_1(t,x)$ and $ u_2(t,x)$ are
  classical solutions to the hyperbolic system
  (\ref{eq:neu-existence:21}) such that $u_1,u_2\in \overline{G_2}$,
  then $u_1\equiv u_2$.
\end{thm}

\begin{proof}[of Theorem \ref{Uniqueness}]
  Let $u_1$ and $u_2$ be a solutions to the hyperbolic system
  hyperbolic system (\ref{eq:neu-existence:21}) with the same initial
  data and let $V(t,x)=u_1(t,x)-u_2(t,x)$. Then $V$ satisfies the
  equation
\begin{equation}
\label{eq:Well-posedness:1}
    A^0(u_1;t,x)\partial_t V =\sum_{a=1}^3 A^a(u_1;,x)\partial_a
V+ B(u_1;t,x)  V + G
\end{equation}
with the initial condition $V(0,x)=0$ and where
\begin{equation*}
\begin{split}
  G=
  &\left[A^0(u_1;t,x)-A^0(u_2)\right]\partial_t u_1+\sum_{a=1}^3
  \left[A^a(u_1;t,x)-A^a(u_2;t,x)\right)]\partial_a u_1\\
 + & \left[B(u_1;t,x)-B(u_2;t,x)\right] u_1
  +\left[F(u_1;t,x)-F(u_2;t,x)\right)].
\end{split}
\end{equation*}
Applying Proposition \ref{prop:contr-lower-norm:1} to
(\ref{eq:Well-posedness:1}), we have
\begin{equation*}
  \frac{d}{dt}  \langle V,V\rangle_{L^2_\delta,A^0(u_1)}\leq \mu C\langle
  V,V\rangle_{L^2_\delta,A^0(u_1)}+\|G\|^2_{L_\delta^2}.
\end{equation*}
Let $T\leq T^*$, then Gronwall's inequality and the equivalence
(\ref{eq:contr-lower-norm:9}) imply
\begin{equation*}
  \n V\n_{0,\delta,T}^2\leq C_1e^{C\mu T}\int_0^T
  \|G(t)\|_{L^2_\delta}^2dt.
\end{equation*}
Similar estimation as done in (\ref{eq:contr-lower-norm:6}) yield that
$\|G(t)\|_{L^2_\delta}^2\leq C_2 \|V(t)\|_{L^2_\delta}^2$. Hence,
\begin{equation}\label{22}
  \n V\n_{0,\delta,T}^2\leq C_3e^{C\mu T}T \quad\n V\n_{0,\delta,T}^2.
\end{equation}
Thus, if $T$ is sufficiently small, then (\ref{22}) leads to a
contradiction unless $V\equiv 0$.  \BeweisEnde
\end{proof}

\begin{thm}[Continuation in norm]
  \label{Continuation} Under conditions (H1)-(H4) and
  (\ref{eq:formulation:3}), any solutions $u$ to the hyperbolic system
  (\ref{eq:neu-existence:21}) which satisfies $u(t,x)\in
  \overline{G_2}$ and the regularity condition (\ref{reg}), satisfies
  in addition
  \begin{equation}
    \label{eq:Well-posedness:2} u \in
    C\left([0,T^{**}],H_{s,\delta}\right)\cap
    C^1\left([0,T^{**}],H_{s-1,\delta+1}\right).
  \end{equation}
\end{thm}

\begin{proof}[of Theorem \ref{Continuation}]
  We first treat the continuity
  $C\left([0,T^{**}],H_{s,\delta}\right)$. Since $u$ is a solution of
  initial value problem (\ref{eq:neu-existence:21}) which is
  reversible in time, is sufficient to show that
  \begin{equation}
    \lim_{t\downarrow 0}\| u(t)-u(0)\|_{H_{s,\delta}}=
    \lim_{t\downarrow 0}\| u(t)-u_0\|_{H_{s,\delta}}=0.
  \end{equation}
  We shall use the following known argument: suppose $\{w_n\}$ is a
  sequence in Hilbert space which converge weakly to $w_0$ and
  $\limsup_n \|w_n\|\leq \|w_0\|$, then $\lim_n \|w_n-w_0\|=0$. We are
  going to use the equivalence norm $\|\cdot
  \|_{H_{s,\delta,A^0(u(0))}}$, so we need to show
  \begin{equation}
    \label{eq:Well-posedness:3} \limsup_{t\downarrow 0}\|
    u(t)\|_{H_{s,\delta,A^0(u(0))}}\leq \|
    u_0\|_H{_{s,\delta,A^0(u(0))}}.
  \end{equation}

  Let $\{u^{k}(t)\}$ be the sequence which is defined by the iteration
  process (\ref{eq:constr-iter:5}).  It follows from the uniqueness
  Theorem \ref{Uniqueness} and (\ref{eq:Weak_Convergence:2}) that
  \begin{equation}\label{eq:Well-posedness:4}
    \| u(t)\|_{H_{s,\delta,A^0(u(t))}}\leq \liminf_k \|
    u^{k}(t)\|_H{_{s,\delta,A^0(u(t))}},
  \end{equation}
  where the limit above is uniformly in $t$.  Applying the  energy
  estimate (\ref{eq:energy-estimates:23}), we have
  \begin{equation*}
    \frac{d}{dt}\|u^{k+1}(t)\|_{H_{s,\delta,A^0(u^k(t))}}^2\leq C
    \left( \mu\|u^{k+1}(t)\|_{H_{s,\delta,A^0(u^k(t))}}^2  +1 \right).
  \end{equation*}
  So Gronwall's inequality yields
  \begin{equation}
    \label{eq:Well-posedness:5}
    \|u^{k+1}(t)\|_{H_{s,\delta,A^0(u^k(t))}}^2\leq e^{C\mu t} \left[
      \|u^{k+1}(0)\|_{H_{s,\delta,A^0(u^k(0))}}^2  +Ct \right].
  \end{equation}
  Take now arbitrary $\epsilon>0$, since $u^k(t)\to u(t)$ uniformly in
  $[0,T^{**}]$, we see from the inner-product
  (\ref{eq:energy-estimates:4}) that there is $k_0$ such that
  \begin{equation}
    \label{eq:Well-posedness:6} \| v(t)\|_{H_{s,\delta,A^0(u(t))}}\leq
    (1+\epsilon)\| v(t)\|_{H_{s,\delta,A^0(u^k(t))}}\qquad\qquad\text{
      for}\qquad k \geq  k_0.
  \end{equation}
  Combing (\ref{eq:Well-posedness:4}), (\ref{eq:Well-posedness:5}),
  (\ref{eq:Well-posedness:6}) and (\ref{eq:constr-iter:4}) with the
  fact that $u^k(t)\to u(t)$ uniformly in $[0,T^{**}]$, we obtain
  \begin{equation*}
    \begin{split}  \limsup_{t \downarrow
        0}\|u(t)\|_{H_{s,\delta,A^0(0)}}^2=&\limsup_{t \downarrow
        0}\|u(t)\|_{H_{s,\delta,A^0(u(t))}}^2\\
      \leq & \limsup_{t \downarrow 0}\left(\liminf_k
        \|u^{k+1}(t)\|_{H_{s,\delta,A^0(u(t))}}^2\right)
      \\\leq &  \limsup_{t \downarrow 0}\left(\liminf_k
        (1+\epsilon)^2\|u^{k+1}(t)\|_{H_{s,\delta,A^0(u^k(t))}}^2\right)\\
      \leq &\limsup_{t \downarrow 0}\left(\liminf_k e^{C\mu t} \left[
          (1+\epsilon)^2\|u^{k+1}(0)\|_{H_{s,\delta,A^0(u^k(0))}}^2
+Ct
        \right]\right) \\
      = &\limsup_{t \downarrow 0}\left( e^{C\mu t} \left[
          (1+\epsilon)^2\|u_0\|_{H_{s,\delta,A^0(u(0))}}^2 +Ct
        \right]\right)\\= &
      (1+\epsilon)^2\|u_0\|_{H_{s,\delta,A^0(u(0))}}^2
    \end{split}
  \end{equation*}
  which proves (\ref{eq:Well-posedness:3}).

  It remains to show that $\lim_{t\to t_0}\left(\|\partial_t
    u(t)-\partial_t u(t_0)\|_{H_{s-1,\delta+1}}\right)=0$.  Now,
  \begin{equation}
    \label{eq:Well-posedness:7} \partial_t u=
    \left(A^0(u;t,x)\right)^{-1}\left\{\sum_{a=1}^3
      A^a(u;t,x)\partial_a u + B(u;t,x)u+F(u;t,x)\right\}.
  \end{equation}
  By the first step of the proof, $\|\partial_a u(t)-\partial_a
  u(t_0)\|_{H_{s-1,\delta+1}}\to 0 $ and $\|u(t)-
  u(t_0)\|_{H_{s,\delta}}\to 0 $. At this stage we apply Corollary
  \ref{cor:appendix:3} to the right hand of
  (\ref{eq:Well-posedness:7}) and this completes the proof of Theorem
  \ref{Continuation}.

  \BeweisEnde
\end{proof}

\subsection{Local existence for the evolution equations of
  Einstein-Euler system}
\label{subsec:Local_Existence_Einstein-Euler_system}

In the previous subsections we have established the well-posedness of
the  first order symmetric hyperbolic system in the $H_{s,\delta}$ spaces. We
would like to apply this result to the coupled 
system (\ref{eq:publ-broken:19}) and (\ref{eq:Initial:2}).

The unknowns of the evolution equations are the gravitational field
$g_{\alpha\beta}$ and its first order partial derivatives
$\partial_\alpha g_{\gamma\delta}$, the Makino variable $ w$ and the
velocity vector $u^\alpha$. We represent them by the vector
\begin{equation}
  \label{eq:hyper_cor:7}
  U=\left(g_{\alpha\beta}-\eta_{\alpha\beta},
    \partial_a g_{\gamma\delta}, \partial_0 g_{\gamma\delta},w,u^{a},u^0-1\right),
\end{equation}  
here $\eta_{\alpha\beta}$ denotes the Minkowski metric.
The initial data for equation
(\ref{eq:publ-broken:19}) are 
\begin{equation*}
g_{ab}{_{\mid_M}=h_{ab}}, \ \ g_{0b}{_{\mid_M}}=0, \ \
g_{00}{_{\mid_M}}=-1, \ \ -\frac{1}{2}\partial_0
g_{ab}{_{\mid_M}}=K_{ab}.
\end{equation*}
where $(h_{ab},K_{ab})$ are given by (\ref{eq:Functions_spaces:5}),
and (\ref{eq:Functions_spaces:3}) for equation (\ref{eq:Initial:2}).
 Hence  $g_{ab}{_{\mid_M}}-\eta_{ab}=h_{ab}-I\in
H_{s,\delta}$ and 
$(w_{\mid_{M}},u^{a}{}_{\mid_{M}},u^0{}_{\mid_{M}}-1)\in
H_{s-1,\delta+2}$.. 
Therefore we conclude that
\begin{equation}
  \label{eq:cor:5} 
  U(0,\cdot)\in
  H_{s,\delta}\times H_{s-1,\delta+1}\times H_{s-1,\delta+2}.
\end{equation} 

In this situation we cannot apply directly Theorem
\ref{thr:publ-broken:1}.  The idea to overcome this obstacle is the
following. We first introduce some more convenience notations:
$\mathord{\bf g}=g_{\alpha\beta}-\eta_{\alpha,\beta}$,
$\partial\mathord{\bf g}=\partial_{\alpha}g_{\gamma\delta}$ (that is,
$\partial\mathord{\bf g}$ is the set of all first order partial
derivatives), $\mathord{\bf v}=(w,u^a,u^0-1)$ and $U=(\mathord{\bf
  g},\partial\mathord{\bf g},\mathord{\bf v})$. Since
$H_{s,\delta}\subset H_{s-1,\delta}$, it follows from (\ref{eq:cor:5})
that
\begin{equation}
  \label{eq:cor:9} 
  U(0,\cdot)\in
  H_{s-1,\delta}\times H_{s-1,\delta+1}\times H_{s-1,\delta+2}.
\end{equation}

If we prove the existence of $U(t,x)$ which is a solution to the
coupled systems (\ref{eq:publ-broken:19}) and (\ref{eq:Initial:2})
with initial data in the form of (\ref{eq:cor:9}) and such that
$U(t,\cdot)\in H_{s-1,\delta}\times H_{s-1,\delta+1}\times
H_{s-1,\delta+2}$ and it is continuous with respect to this norm, then
from inequality
\begin{equation}
  \label{eq:hyper-cor:1} 
  \|\mathord{\bf g}\|_{H_{s,\delta}} \lesssim \left( \|\mathord{\bf
      g}\|_{H_{s-1,\delta} }+\|\partial \mathord{\bf
      g}\|_{H_{s-1,\delta+1}}\right), 
\end{equation}
we will get that $U(t,\cdot) \in H_{s,\delta}\times
H_{s-1,\delta+1}\times H_{s-1,\delta+2}$ and it will be continuous
with respect to the norm of $H_{s,\delta}\times H_{s-1,\delta+1}\times
H_{s-1,\delta+2}$. Note that (\ref{eq:hyper-cor:1})  certainly
holds for  the integral representation of the norm
(\ref{eq:const:3}), and by Theorem \ref{thm:a1} it holds also for
the
$H_{s,\delta}$ norm.

In order to achieve this we carefully examine the structure of the
coupled systems (\ref{eq:publ-broken:19}) and (\ref{eq:Initial:2}).
According to Conclusion \ref{thm:hyperbolic_reduction}, we can write
Einstein-Euler system in the form:
\begin{equation}
  \label{eq:cor:2}
  A^0(U)\partial_t U=\sum_{a=1}^3 A^a(U)\partial_aU +B(U)U, 
\end{equation} 
where $A^\alpha$ and $B$ are $55\times 55$ matrices such that
\begin{equation}
 \label{eq:cor:6}
      {A}^\alpha=\left(\begin{array}{c|c|c} I_{10} & {\bf
            0}_{10\times 40} & {\bf 0}_{10\times 5} \\ \hline & & \\
          {\bf 0}_{40\times 10} &\widetilde{A^\alpha}(\mathord{\bf g})
&
          {\bf 0}_{40\times 5}\\ \hline & & 
          \\
          {\bf 0}_{5\times 10} & {\bf 0}_{5\times 40}&
          \widehat{A}^\alpha(\mathord{\bf g},\partial\mathord{\bf
g},\mathord{\bf v})\\ 
        \end{array}\right) 
\end{equation}


and
\begin{equation}
  \label{eq:hyperbolic-cor:3}
  B=\left(\begin{array}{c|c|c} 
      {\bf 0}_{10}   &  {\bf b}_{10\times 40} & {\bf 0}_{10\times 5}  \\ \hline 
      &                         &                        \\ 
      \quad &\widetilde{B}(\mathord{\bf g},\partial\mathord{\bf
        g},\mathord{\bf v}) & \quad\\ \hline 
      &                      &                       \\
      {\bf 0}_{5\times 10} & {\bf 0}_{5\times 40}&
      {\bf 0}_{5\times 5}\\
    \end{array}\right). 
\end{equation}
Here $\widetilde{A^\alpha}(\mathord{\bf g})$ is $40\times 40$
matrix which represents system (\ref{eq:publ-broken:19}),
$\widehat{A^\alpha}(\mathord{\bf g},\partial\mathord{\bf
g}, \mathord{\bf v})$ is 
$5\times 5$ matrix of system (\ref{eq:Initial:2}). Both of
them are symmetric and both $\widetilde{A^0}(\mathord{\bf g})$ and
$\widehat{A^0}(\mathord{\bf g},\partial\mathord{\bf g},\mathord{\bf
v})$ are positive definite
matrices; $\widetilde{B}(\mathord{\bf g},\partial\mathord{\bf
  g},\mathord{\bf v})$ is $40\times 55$ matrix,
 and ${\bf b}_{10\times 40}$ is a constant matrix.

A natural norm of $U=(\mathord{\bf g},\partial\mathord{\bf
g},\mathord{\bf v})$ on the product space $H_{s-1,\delta}\times
H_{s-1,\delta+1}\times H_{s-1,\delta+2}$ is
\begin{equation}
  \label{eq:cor:3} 
  \n U\n_{H_{s-1,\delta}}^2=\|\mathord{\bf g} \|_{H_{s-1,\delta}}^2+
\|\partial
  \mathord{\bf g}\|_{H_{s-1,\delta+1}}^2+\|\mathord{\bf v}\|_{H_{s-1,\delta+2}}^2. 
\end{equation}

Note that from Proposition (\ref{prop:Properties:3}) and
Theorem \ref{thr:Appendix:5} we have that $A^\alpha U, BU\in
H_{s-1,\delta}\times H_{s-1,\delta+1}\times H_{s-1,\delta+2}$,
whenever $U\in H_{s-1,\delta}\times H_{s-1,\delta+1}\times
H_{s-1,\delta+2}$.

We formulate an inner-product in accordance with the norm
(\ref{eq:cor:3}) and the structure of $A^0$.  Let $U_1=(\mathord{\bf
  g}_1,\partial\mathord{\bf g}_1,\mathord{\bf v}_1)$ and
$U_2=(\mathord{\bf g}_2,\partial\mathord{\bf g}_2,\mathord{\bf v}_2)$,
similarly to (\ref{eq:energy-estimates:4}) we set
\begin{eqnarray}
  \nonumber & &\langle U_1,U_2\rangle_{s-1,\delta,A^0}\\ \nonumber&:=&
  \sum_{j=0}^\infty 2^{( \frac{3}{2} + \delta)2j} \int\left[ 
    \Lambda^{s-1}\left((\psi_j^2 \mathord{\bf
        g}_1)_{(2^j)}\right)\right]^T \left[  
    \Lambda^{s-1}\left((\psi_j^2
      \mathord{\bf g}_2)_{(2^j)}\right)\right] dx\\ \nonumber &+&
  \sum_{j=0}^\infty 2^{( \frac{3}{2} + \delta+1)2j} \left[ 
    \Lambda^{s-1}\left((\psi_j^2 \partial
      \mathord{\bf g}_1)_{(2^j)}\right)\right]^T (\widetilde{A}^0)_{(2^j)}
  \left[\Lambda^{s-1}\left((\psi_j^2 \partial
      \mathord{\bf g}_2)_{(2^j)}\right)\right] dx\\
  &+& \sum_{j=0}^\infty 2^{( \frac{3}{2} + \delta+2)2j} \left[ 
    \Lambda^{s-1}\left((\psi_j^2 \mathord{\bf
        v}_1)_{(2^j)}\right)\right]  (\widehat{A}^0)_{(2^j)} 
  \left[  \Lambda^{s-1}\left((\psi_j^2
      \mathord{\bf v}_2)_{(2^j)}\right)\right] dx\ \ \ \label{eq:cor:4}
\end{eqnarray}
and $\n U \n^2_{H_{s-1,\delta,A^0}}=\left\langle
  U,U\right\rangle_{s-1,\delta,A^0}$.  Since $A^0$ is positive
definite, $\n U\n_{H_{s-1,\delta,A^0}}\sim \n U\n_{H_{s-1,\delta}}$

We can now repeat all the arguments and estimations of subsections
\ref{subsec:Energy_estimates}-\ref{sec:Well-posedness}, which are
applied term-wise to the norm (\ref{eq:cor:3}) and inner-product
(\ref{eq:cor:4}), and in this way we extend Theorem
\ref{thr:publ-broken:1} to the product space:

\begin{thm}[Well posedness of hyperbolic systems in product spaces]
  \label{thm:cor:1} Let $s-1>\frac{5}{2}$, $\delta \geq -\frac{3}{2}$
  and assume the coefficient of (\ref{eq:cor:2}) are of the form
  (\ref{eq:cor:6}) and (\ref{eq:hyperbolic-cor:3}). If $U_0\in
  H_{s-1,\delta}\times H_{s-1,\delta+1}\times H_{s-1,\delta+2}$ and
  satisfies
  \begin{equation}
    \label{eq:cor:10} \frac{1}{\mu}
    U_0^TU_0 \leq U_0^T A^0U_0
    \leq
    \mu U_0^TU_0, \qquad \mu\in
    \setR^+
  \end{equation}
  then there exits a positive $T$ which depends on
  $\n U_0\n_{H_{s-1,\delta}}$ and a unique $U(t,x)$ a solution to
  (\ref{eq:cor:2}) such that $U(0,x)=U_0(x)$ and in addition it
  satisfies
  \begin{equation}
    \label{eq:appendix:12}
    U\in C([0,T],H_{s-1,\delta}\times  H_{s-1,\delta+1}\times
    H_{s-1,\delta+2})\cap C^1([0,T],H_{s-2,\delta+1}\times
    H_{s-2,\delta+2}\times H_{s-2,\delta+3}). 
  \end{equation}
\end{thm}

\begin{cor}[Solution to the gravitational field and the fluid]
  \label{cor:hyper_cor:1} Let
  $\frac{7}{2}<s<\frac{2}{\gamma-1}+\frac{3}{2}$ and
  $\delta>-\frac{3}{2}$.  Then there exists a positive $T$, a unique
  gravitational field $g_{\alpha\beta}$ solution to
  (\ref{eq:publ-broken:19}) and a unique $(w,u^\alpha)$ solution to
  Euler equation (\ref{eq:Initial:2}) such that
  \begin{equation}
    \label{eq:hyper_cor:5}
    g_{\alpha\beta}-\eta_{\alpha\beta}\in C([0,T],H_{s,\delta})\cap
    C^1([0,T],H_{s-1,\delta+1}) 
  \end{equation}
  and
  \begin{equation}
    \label{eq:hyper_cor:6}
    (w,u^a,u^0-1)\in C([0,T],H_{s-1,\delta+2})\cap C^1([0,T],H_{s-2,\delta+3}).
  \end{equation}
\end{cor}

\begin{proof}[of Corollary \ref{cor:hyper_cor:1}]
  Theorem \ref{thm:Function_space:1}  implies that the initial
data
  for $g_{\alpha\beta}$ belong to $H_{s,\delta}$ and initial data for
  $(w,u^\alpha)$ are in $H_{s-1,\delta+2}$. Thus $U(0,\cdot)\in
  H_{s-1,\delta}\times H_{s-1,\delta+1}\times H_{s-1,\delta+2}$, where
  $U$ is given by (\ref{eq:hyper_cor:7}). In addition, the continuity
  of $A^0$ implies that the
  vector $U(0,\cdot)$ satisfies (\ref{eq:cor:10}). Therefore Theorem
  \ref{thm:cor:1} with inequality (\ref{eq:hyper-cor:1}) give the
  desired result.  \BeweisEnde
\end{proof}


\section{Quasi Linear Elliptic Equations in $H_{s,\delta}$}
\label{sec:elliptic}

In this section we will establish the elliptic theory in
$H_{s,\delta}$ which is essential for the solution of the constraint
equations.  We will extend earlier results in weighted Sobolev
spaces of integer order  which were obtained by Cantor
\cite{cantor79}, Choquet-Bruhat and Christodoulou 
\cite{choquet--bruhat81:_ellip_system_h_spaces_manif_euclid_infin%
} and
Christodoulou and O'Murchadha \cite{OMC} to the fractional ordered
spaces. The essential tool is the a priori estimate
(\ref{eq:a_priori_estimates_weighted:6}) and proving it requires first
to
establish  an analogous a priori estimate in Bessel potential
spaces. Our approach is based on the techniques of
Pseudodifferential Operators which have symbols with limited
regularity and we are adopting ideas being presented in Taylor's
books  \cite{Taylor91
}
and 
\cite{taylor00
}.  A different method
was derived recently by Maxwell \cite{maxwell06:_rough_einst}.

\subsection{A priori estimates for linear elliptic systems in  $H^s$}
\label{subsec:a_priori_estimate_Bessel} In this section we
consider a second order homogeneous elliptic system
\begin{equation}
    \label{eq:a_priori_estimate_Bessel:1}
    (Lu)^i=\sum_{\alpha,\beta,j} a^{\alpha \beta}_{ij}(x)
    \partial_\alpha\partial_\beta u^j,
\end{equation}
where the indexes $i,j=1,,,N$ and $\alpha,\beta=1,2,3$ (since only
$ \mathbb{R}^3$ is being discussed in this paper). We will use the
convention
\begin{equation}
    \label{eq:a_priori_estimate_Bessel:2}
    Lu = A(x) D^2u,
\end{equation}
 where  $A(x)=a_{ij}^{\alpha\beta}(x)$,
$(D^2u)^j_{\alpha\beta}=\partial_\alpha
\partial_\beta u^j$ and
$(A(x)D^2u)_i=a_{ij}^{\alpha\beta}(x)(D^2u)^j_{\alpha\beta}$. 
The symbol of (\ref{eq:a_priori_estimate_Bessel:1}) is $N\times N$
matrix $A(x,\xi)$, defined for all $\xi\in \setC^3$ as follows:
\begin{equation}
 \label{eq:a_priori_estimate_Bessel:3}
 A(x,\xi)_{ij}:=
 -\sum_{\alpha,\beta}
  a^{\alpha\beta}_{ij}(x)i\xi_\alpha i\xi_\beta.
\end{equation}

The following definitions are due to Morrey \cite{morrey54:_secon%
}.
\begin{defn}

\quad
\begin{enumerate}
 \item The system
(\ref{eq:a_priori_estimate_Bessel:1}) is {\bf elliptic} provided
that
\begin{equation}
    \label{eq:a_priori_estimate_Bessel:4}
 \det\left(A(x,\xi)\right)=  {\det}\left(\sum_{\alpha,\beta} a^{\alpha
    \beta}_{ij}(x)\xi_\alpha\xi_\beta\right)\not=0, \qquad
    \text{for all\qquad} 0\not=\xi\in \mathbb{R}^3;
\end{equation}

\item The system (\ref{eq:a_priori_estimate_Bessel:1}) is {\bf
strongly elliptic} provided that for some positive $\lambda$
\begin{equation}
    \label{eq:a_priori_estimate_Bessel:5}
  \langle A(x,\xi)\eta,\eta\rangle=  \sum_{\alpha,\beta,i,j}
  a^{\alpha\beta}_{ij}(x)\xi_\alpha\xi_\beta \eta^i\eta^j \geq
  \lambda 
    |\xi|^2|\eta|^2.
\end{equation}
\end{enumerate}
\end{defn}

Our main task is to obtain a priori estimate in the Bessel potential
spaces $H^s$ for the operator (\ref{eq:a_priori_estimate_Bessel:1})
whose coefficients $ a_{\alpha \beta}^{ij}$ belong to $ H^{s_2}$.  In
case $s$ and $ s_2$ are integers, then one may prove
(\ref{eq:a_priori_estimate_Bessel:8}) below by means of induction and
the classical results of Douglas and Nirenberg
\cite{douglis55:_inter
},
and Morrey
\cite{morrey54:_secon
}. We will employ techniques of Pseudodifferential calculus.

If the coefficients of the matrix $A$ belongs to $H^{s_2}$, then
$A(x,\xi)\in H^{s_2}S_{1,0}^2$, that is, $\|\partial_\xi^\alpha
A(\cdot,\xi)\|_{H^{s_2}}\leq C_\alpha
(1+|\xi|^{2})^{(2-|\alpha|)/2}$). We follow Taylor and decompose
\begin{equation}
    \label{eq:a_priori_estimate_Bessel:6}
    A(x,\xi)=A^{\#}(x,\xi)+A^b(x,\xi)
\end{equation}
in such way that a good parametrix can be constructed for
$A^{\#}(x,\xi)$, while  $A^b(x,\xi)$ will have  order less than 2.
According to Proposition 8.2 in \cite{taylor00}, for $s_2>
\frac{3}{2}$ there is $0<\delta<1$ such that
\begin{equation*}
    A^{\#}(x,\xi)\in S_{1,\delta}^2,\qquad\quad A^b(x,\xi)\in
    H^{s_2}S_{1,\delta}^{2-\sigma\delta},
    \qquad\quad\sigma=s_2-\frac{3}{2}
\end{equation*}
where  \begin{math}
  A^{\#}(x,\xi)=  \sum_{k=0}^\infty
    J_{\epsilon_k}A(x,\xi)\phi_k(\xi)
  \end{math}, $\epsilon_k=c2^{-k\delta}$.  Here $\{\phi_{k}\}$ is the
  Littlewood-Paley partition of unity, that is, ${\phi_0\in
    C_0^\infty(\mathbb{R}^3)}$, $\phi_0(0)=1$,
  $\phi_k(\xi)=\phi_0(2^{-k}\xi)-\phi_0(2^{-k+1}\xi)$ and
  $\sum_{k=0}^\infty \phi_k(\xi)=1$. The smoothing operator
  $J_\epsilon$ is defined as follows:
\begin{equation*}
  J_{\epsilon}f(x)=\phi_0(\epsilon D)f(x)=\left(
    \frac{1}{2\pi}\right)^{\frac{3}{2}}\int \epsilon^{-3} 
  \widehat{  \phi_0}(\frac{y}{\epsilon})f(x-y)dy,
\end{equation*}
where $\widehat{\phi}_0$ is the inverse Fourier transform. In order
that $A^{\#}$ will have a good parametrix we need to verify that it is
a strongly elliptic. Since the original operator is strongly elliptic,
\begin{equation*}
   \begin{split}
   &\sum_{\alpha,\beta,i,j} J_{\epsilon_k}a^{\alpha\beta}_{ij}(x)
   \phi_{k}(\xi)
    \xi_\alpha\xi_\beta \eta^i\eta^j\\= &
  \left( \frac{1}{2\pi}\right)^{\frac{3}{2}} \int
  \left(\sum_{\alpha,\beta,i,j}  \epsilon_k^{-3} 
  \widehat{  \phi_0}(\frac{y}{\epsilon_k})(y)a^{\alpha\beta}_{ij}(y-x)
\phi_{k}(\xi) \xi_\alpha\xi_\beta \eta^i\eta^j\right)dy \\
    \geq  &  \left( \frac{1}{2\pi}\right)^{\frac{3}{2}}\quad\lambda
    \phi_{k}(\xi)|\xi|^2|\eta|^2\int \epsilon_k^{-3} 
  \widehat{  \phi_0}(\frac{y}{\epsilon_k})dy
  = \lambda \phi_{k}(\xi)|\xi|^2|\eta|^2{\phi_{0}}(0)
  \\=\quad& \lambda \phi_{k}(\xi)|\xi|^2|\eta|^2
    \end{split}
\end{equation*}
 for each fixed $k$. Summing over the $k$ we have,
\begin{eqnarray*}
\langle A^{\#}(x,\xi)\eta,\eta\rangle =
\sum_{k=0}^\infty \sum_{\alpha,\beta,i,j} \left(J_{\epsilon_k}
a^{\alpha\beta}_{ij}\right)(x)
   \phi_{k}(\xi)
    \xi_\alpha\xi_\beta \eta^i\eta^j
  \geq
    \sum_{k=0}^\infty\lambda \phi_{k}(\xi)|\xi|^2|\eta|^2=
    \lambda |\xi|^2|\eta|^2,
\end{eqnarray*}
thus  (\ref{eq:a_priori_estimate_Bessel:5}) holds for $A^{\#}$. The
last step
assures that
\begin{math}
\|A^{\#}(x,\xi)^{-1} \| \leq \frac{1}{\lambda |\xi|^2}
 \end{math}
 and then it follows from the identity
 $\partial(A^{-1})=A^{-1}(\partial(A))A^{-1}$ that
\begin{equation*}
 \|\partial_x^\beta
\partial_\xi^\alpha(A^{\#}(x,\xi))^{-1} \| \leq C_{\alpha\beta}
(1+|\xi|^2)^{(-2-|\alpha|+\delta|\beta|)/2},
 \end{equation*}
that is,  $(A^{\#}(x,\xi))^{-1}\in S_{1,\delta}^{-2}$. Hence, the
operator $A^{\#}(x,D)$ has a parametrix  $E^{\#}(x,D)\in
OPS_{1,\delta}^{-2}$ satisfying
\begin{equation}
\label{eq:a_priori_estimate_Bessel:7}
 E^{\#}(x,D)A^{\#}(x,D)=I + S,
 \end{equation}
 where $S\in OPS^{-\infty}$ (See e.~g.~\cite{Taylor91
}  Section 0.4).

\begin{lem}[An a priori estimates in $H^s$]
  \label{lem:a_priori_estimate_Bessel:2}
 Let $Lu=A(x) D^2u$ be a strongly elliptic system  and assume
 $A\in H^{s_2}$, $s_2>\frac{3}{2} $ and $0\leq s-2\leq s_2$. Then
there is
 a constant $C$ such that
\begin{equation}
 \label{eq:a_priori_estimate_Bessel:8}
   \|u\|_{H^s}\leq C\left\{ \|Lu\|_{H^{s-2}}+  \|u\|_{H^{s-2}}
  \right\}.
\end{equation}
\end{lem}

\begin{proof}[of Lemma \ref{lem:a_priori_estimate_Bessel:2}]
  We decompose $A(x,D)$ as in (\ref{eq:a_priori_estimate_Bessel:2})
  and let $E^{\#}(x,D)$ be the above parametrix, then by
  (\ref{eq:a_priori_estimate_Bessel:7})
\begin{equation}
 \label{eq:a_priori_estimate-Bessel:9}
E^{\#}(x,D)A(x,D)u=u+Su+E^{\#}(x,D)A^b(x,D)u.
\end{equation}
Since $ E^{\#}(x,D), S: H^{s-2} \to H^s$ are bounded,  
\begin{equation}
 \label{eq:a_priori_estimate_Bessel:10}
\|E^{\#}(x,D)A(x,D)u\|_{H^s}=\|E^{\#}(x,D)Lu\|_{H^s}\leq C
\|Lu\|_{H^{s-2}}
\end{equation}
and
\begin{equation}
 \label{eq:a_priori_estimate_Bessel:11}
\|Su\|_{H^s}\leq C \|u\|_{H^{s-2}}.
\end{equation}
According to \cite{taylor00} Proposition 8.1, (see also
\cite{Taylor91
} 
Proposition 2.1.J)
\begin{equation*}
A^b(x,D): H^{s-\sigma\delta}\to H^{s-2}.
\end{equation*}
Hence,
\begin{equation}
 \label{eq:a_priori_estimate_Bessel:12}
\|E^{\#}(x,D)A^b(x,D)u\|_{H^s}\leq C\|A^b(x,D)u\|_{H^{s-2}}\leq C
\|u\|_{H^{s-\sigma\delta}}.
\end{equation}
Using the intermediate estimate
\begin{math}
\|u\|_{H^{s-\sigma\delta}}\leq \epsilon \|u\|_{H^{s}}+
C(\epsilon)\|u\|_{H^{s-2}},
\end{math}
and combining it with(\ref{eq:a_priori_estimate-Bessel:9})-
(\ref{eq:a_priori_estimate_Bessel:11}),
we obtain the estimate (\ref{eq:a_priori_estimate_Bessel:8}).
\BeweisEnde
\end{proof}

\subsection{A priori estimates in  $H_{s,\delta}$}
\label{subsec:a_priori_estimates_weighted}
Our main task here is to extend the a priori estimate
(\ref{eq:a_priori_estimate_Bessel:8}) to $H_{s,\delta}$-spaces and for
 second order elliptic systems of the form:

\begin{equation}
 \label{eq:a_priori_estimates_weighted:1}
 \begin{split}
 (Lu)^i& =\sum_{\alpha,\beta,j} a^{\alpha
   \beta}_{ij}(x)\partial_\alpha\partial_\beta u^j 
    +\sum_{\alpha,j}b^{\alpha}_{ij}(x)\partial_\alpha u^j+
    \sum_{j} c_{ij}(x) u^j\\& =A(x)D^2u+B(x)(Du)+C(x)u.
    \end{split}
\end{equation}
Here $A(x)$ is as in the previous subsection,
$B(x)=b_{ij}^\alpha(x)$, $(Du)_\alpha^j=\partial_\alpha u^j$,
$(B(x)Du)_i=b_{ij}^\alpha(x)\partial_\alpha u^j$ and 
$C(x)=c_{ij}(x)$ is $N\times N$. We introduce
the following hypotheses:

\medskip\noindent
\textbf{Hypotheses} \begin{math}  (H)\end{math}
\begin{enumerate}
  \item[{\rm (H1)}]
  \label{eq:a_priori_estimates_weighted:H} $\sum
  a^{\alpha,\beta}_{i,j}(x)\eta^i\eta^j\xi_\alpha \xi\beta \geq
  \lambda|\eta|^2|\xi|^2 $ \quad (i.e. $L$ is strongly elliptic);
  \item[{\rm (H2)}] $\left(A(\cdot)-A_\infty\right)\in
  H_{s_2,\delta_2}, \quad B\in H_{s_1,\delta_1},\quad C\in
  H_{s_0,\delta_0}\\
  s_i \geq s-2, i=0,1,2,\quad s_2>\frac{3}{2}, s_1>\frac{1}{2}, s_0
  \geq 0\quad {\rm and}\quad \delta_i>\frac{1}{2} -i, \quad i=0,1,2$,
  the matrix $A_\infty$ has constant coefficients and $A_\infty D^2u$
  is an elliptic system, that is,
  \begin{math}{\rm det}\left(\sum     
\left(a_\infty\right)^{\alpha,\beta}_{ij}
\xi_\alpha\xi_\beta\right)\not
    =0\end{math}.
\end{enumerate}

We shall first derive an a priori estimate for a second order
homogeneous operator
\begin{equation*}
   { L_2u} = A(x) D^2u.
\end{equation*}

\begin{lem}[An a priori estimate for homogeneous operator in
$H_{s,\delta}$]
  \label{lem:a_priori_estimates_weighted:1}
Assume the operator ${L_2}$ satisfies hypotheses $(H)$ and $s \geq
2 $.
 Then
\begin{equation}
\label{eq:a_priori _estimates _weighted:2}
 \| u\|_{H_{s,\delta}} \leq C\left\{ \|{L_2}u\|_{H_{s-2,\delta+2}} +
\|
 u\|_{H_{s-2,\delta}}\right\},
\end{equation}
where the constant $C$ depends on $s,\delta$ and $\|
A-A_\infty\|_{H_{s_2,\delta_2}}$.
\end{lem}

\begin{proof}[of Lemma \ref{lem:a_priori_estimates_weighted:1}]
According to Corollary \ref{cor:app11}, 
\begin{equation*}
  \sum_{j=0}^\infty 2^{\left( \frac{3}{2} +\delta \right)2j}
  \left\| \left( \psi_j^4u \right)_{2^j} \right\|_{H^s}^2
\end{equation*}
is an equivalent norm in $H_{s,\delta}$.  The main idea of the proof
is to apply Lemma \ref{lem:a_priori_estimate_Bessel:2} to each term of
the equivalent norm above.  We use the convention
(\ref{eq:a_priori_estimate_Bessel:2}) and compute
\begin{equation*}
\begin{split}
  L_2(\psi^4 u)&= \psi^4L_2\left(D^2u\right)
+\psi A(x)R(u,\psi),
   \end{split}
\end{equation*}
where
\begin{equation*}
R(u,\psi)_{\alpha\beta}=8\psi\left(D\psi\right)_\alpha\left(\psi
Du\right)_\beta+12\left(D\psi\right)_\alpha\left(D\psi\right)_\beta
(\psi u)+4\psi\left(D^2\psi\right)_{\alpha\beta}(\psi u).
\end{equation*}
Applying the a priori estimate (\ref{eq:a_priori_estimate_Bessel:8}),
we have
\begin{eqnarray}
\nonumber
 \| u \|_{H_{s,\delta}}^2
 &\lesssim&
  \sum_{j=0}^\infty 2^{\left( \frac{3}{2} +\delta \right)2j}
  \left\| \left( \psi_j^4u \right)_{2^j} \right\|_{H^s}^2\\\nonumber
  &\lesssim &
  \sum_{j=0}^\infty 2^{\left( \frac{3}{2} +\delta \right)2j}
  \left\{\left\|  L_2\left((\psi_j^4u) \right)_{2^j}
  \right\|_{H^{s-2}}^2 +
  \left\| \left( \psi_j^4u \right)_{2^j}
  \right\|_{H^{s-2}}^2\right\}\\\nonumber
  &\lesssim &
  \sum_{j=0}^\infty 2^{\left( \frac{3}{2} +\delta \right)2j}
  \left\{2^{4j}\left\|  \left(\psi_j^4 L_2(u) \right)_{2^j}
  \right\|_{H^{s-2}}^2 +
  \left\| \left( \psi_j^4u \right)_{2^j}
  \right\|_{H^{s-2}}^2\right\}\\\nonumber &  + &
\sum_{j=0}^\infty 2^{\left( \frac{3}{2} +\delta+2 \right)2j}
  \left\| \left(\psi_j A R(u,\psi_j)\right)_{2^j}
  \right\|_{H^{s-2}}^2\\\label{eq:a_priori_estimates_weighted:3}
 &\lesssim &
 \left\| L_2(u) \right\|_{H_{s-2,\delta+2}}^2+
  \left\| u \right\|_{H_{s-2,\delta}}^2
 + \| AR\|_{H_{s-2,\delta+2}}^2.
\end{eqnarray}
The assumption on $s_2$ and $\delta_2$ enable us to use  Proposition
\ref{prop:Properties:3}  and get
\begin{equation}
\label{eq:a_priori_estimates_weighted:4}
\begin{split} \| AR\|_{H_{s-2,\delta+2}} & \leq C\left(\|
(A-A_\infty)R\|_{H_{s-2,\delta+2}}+ \| A_\infty
R\|_{H_{s-2,\delta+2}}\right) \\&\leq
C\left(\|(A-A_\infty)\|_{H_{s_2,\delta_2}}+\|A_\infty\|\right)\|
R\|_{H_{s-2,\delta+2}}.
\end{split}
\end{equation}
Property (\ref{eq:const:4}) of $\psi_j$ and inequality
(\ref{eq:Properties:2}) imply
\begin{equation*}
 \| (\psi_j R)_{2^j}\|_{H^{s-2}}  \leq C\left(2^{-j}\| (\psi_j
 Du)_{2^j}\|_{H^{s-2}}+ 2^{-2j}\| (\psi_j
 u)_{2^j}\|_{H^{s-2}}\right)
\end{equation*}
and hence $\|R\|_{H_{s-2,\delta+2}}\leq C\left(
\|u\|_{H_{s-1,\delta}} + \|u\|_{H_{s-2,\delta}}\right)$. Thus,
inequalities (\ref{eq:a_priori_estimates_weighted:3}) and
(\ref{eq:a_priori_estimates_weighted:4}) yields
\begin{equation}
\label{eq:a_priori_estimtes_weighted:5}
 \|u\|_{H_{s,\delta}}\leq C\left\{
 \|L_2u\|_{H_{s-2,\delta+2}}+\left(
 \|A-A_\infty\|_{H_{s_2,\delta_2}}+1\right)\left(
 \|u\|_H{_{s-1,\delta}}+ \|u\|_{H_{s-2,\delta}}\right)\right\}.
\end{equation}
Invoking the intermediate  estimate 
\begin{math}
\|u\|_H{_{s-1,\delta}}\leq \sqrt{2\epsilon}\|u\|_H{_{s,\delta}}+
C(\epsilon) \|u\|_H{_{s-2,\delta}} 
\end{math} (see Proposition \ref{prop:Properties:4}) and  taking
$\epsilon$ so that $
C\left(
 \|A-A_\infty\|_{H_{s_2,\delta_2}}+1\right)\sqrt{2\epsilon}\leq
\frac{1}{2}$, we obtain from (\ref{eq:a_priori_estimtes_weighted:5})
  the desired estimate (\ref{eq:a_priori _estimates _weighted:2}).
\BeweisEnde
\end{proof}

\begin{lem}[An a priori estimate in $H_{s,\delta}$]
  \label{lem:a_priori_estimates_weighted:2}
  Assume the operator ${L}$ of the form
  (\ref{eq:a_priori_estimates_weighted:1}) satisfies hypotheses $(H)$
  and $s \geq 2 $.  Then
\begin{equation}
\label{eq:a_priori_estimates_weighted:6}
 \| u\|_{H_{s,\delta}} \leq C\left\{ \|{L}u\|_{H_{s-2,\delta+2}} + \|
 u\|_{H_{s-2,\delta}}\right\},
\end{equation}
where the constant $C$ depends on $s,\delta$ and the coefficients
of $L$.
\end{lem}

\begin{proof}[of Lemma \ref{lem:a_priori_estimates_weighted:2}] By
Lemma
\ref{lem:a_priori_estimates_weighted:1},
\begin{equation*}
\begin{split}
 \| u\|_{H_{s,\delta}} &\leq C\left\{ \|{L_2}u\|_{H_{s-2,\delta+2}} +
\|
 u\|_{H_{s-2,\delta}}\right\}\\&
 \leq C\left\{ \|{L}u\|_{H_{s-2,\delta+2}} + \|
 u\|_{H_{s-2,\delta}}+ \|(L-L_2)u\|_{H_{s-2,\delta+2}}\right\},
 \end{split}
\end{equation*}
where
\begin{math}
  (L-L_2)u=B(x)(Du)+C(x)u.
\end{math}
Hypothesis (H2) together with Corollary \ref{cor:const:1} and
Proposition \ref{prop:Properties:3} give
\begin{equation}
\label{eq:a_priori_estimates_weighted:7}
\begin{split}
 \| B(x)(Du)\|_{H_{s-2,\delta+2}} \lesssim \|B\|_{H_{s_1,\delta_1}} \|
 D u\|_{H_{s-2,\delta+1}}  \lesssim \|B\|_{H_{s_1,\delta_1}} \|
 u\|_{H_{s-1,\delta}},
 \end{split}
\end{equation}
and
\begin{equation*}
 \| C(x)u\|_{H_{s-2,\delta+2}} \lesssim \|C|\|_{H_{s_0,\delta_0}}
 \|u\|_{H_{s-2,\delta}}. 
\end{equation*}
Finally, we apply the intermediate estimate, Proposition
\ref{prop:Properties:4}, to the right hand side of
(\ref{eq:a_priori_estimates_weighted:7}) and by taking $\epsilon$
sufficiently small we obtain (\ref{eq:a_priori_estimates_weighted:6}).
\BeweisEnde
\end{proof}

\begin{lem}[Isomorphism of an operator with constant coefficients]
  \label{lem:a_priori_estimates_weighted:3}
  Let $A_\infty u:=A_\infty D^2u$ be a homogeneous elliptic system
  with constant coefficients. Then for any $s \geq 2$ and
  $-\frac{3}{2}<\delta<-\frac{1}{2}$, the operator $A_\infty:
  H_{s,\delta+2}\to H_{s-2,\delta}$ is isomorphism satisfying
\begin{equation}
\label{eq:a_priori_estimates_weighted:8}
 \| u\|_{H_{s,\delta}} \leq C\| A_\infty D^2
 u\|_{H_{s-2,\delta+2}}.
\end{equation}
\end{lem}
\begin{proof}[of  Lemma \ref{lem:a_priori_estimates_weighted:3}]
  Both statements are true when $s$ is an integer and under the norm
  (\ref{eq:const:1}) (see e.~g.~\cite{
    choquet--bruhat81:_ellip_system_h_spaces_manif_euclid_infin},
  Theorem 5.1) and by Theorem \ref{thm:a1} they hold also in the
  $H_{s,\delta}$ norm (\ref{eq:weighted:4}). For $s$ between two
  integers $m_0$ and $m_1$, we have $s=s_\theta=\theta
  m_0+(1-\theta)m_1$ and $s-2=s_\theta-2=\theta(
  m_0-2)+(1-\theta)(m_1-2)$, where $0<\theta<1$. The interpolation
  Theorem \ref{const:thm:1} implies
\begin{equation*}
  H_{s,\delta}=[H_{m_0,\delta}, H_{m_1,\delta}]_\theta\quad \text{and}
\quad
  H_{s-2,\delta}=[H_{m_0-2,\delta}, H_{m_1-2,\delta}]_\theta.
\end{equation*}
Since $A_\infty^{-1}: H_{m_i-2,\delta}\to H_{m_i,\delta+2}$, $i=0,1$,
is continuous, it follows from interpolation theory that
$A_\infty^{-1}: H_{s_\theta-2,\delta}\to H_{s_\theta,\delta+2}$ is
also continuous (see
e.~g.~\cite{triebel95
}).  Hence (\ref{eq:a_priori_estimates_weighted:8}) holds.
\BeweisEnde
\end{proof}

The next lemma improves the a priori estimate
(\ref{eq:a_priori_estimates_weighted:6}).

\begin{lem}[Improved a priori estimate]
  \label{lem:a_priori_estimates_weighted:4}
 Let $ {L}$ be an elliptic operator of the form
 (\ref{eq:a_priori_estimates_weighted:1}) which
  satisfies  hypotheses $(H)$. Assume  $s \geq  2$ and $-\frac{3}{2} <
  \delta<-\frac{1}{2}$. 
Then  for any $\delta'$ there is a constant $C$  such that
\begin{equation}
\label{eq:a_priori_estimates_weighted:9}
 \| u\|_{H_{s,\delta}} \leq C\left\{ \|{L}u\|_{H_{s-2,\delta+2}} + \|
 u\|_{H_{s-1,\delta'}}\right\}.
\end{equation}
The constant $C$ depends on the $H_{s_i,\delta_i}$-norm of the
coefficients of $L$, $s$, $\delta$ and $\delta'$.
\end{lem}

\begin{proof}[of  Lemma \ref{lem:a_priori_estimates_weighted:4}]
 Let $\chi_R\in C_0^\infty( \mathbb{R}^3)$ be a
cut-off function satisfying $\supp(\chi_R)\subset \{|x|\leq 2R\}$,
$\chi_R(x)=1 $ for $|x|\leq R$, $0\leq\chi_R(x)\leq 1$ and
$\|\partial^\alpha
\chi_R\|_\infty\leq C_\alpha R^{-|\alpha|}$. For $u\in H_{s,\delta}$
we write
\begin{equation*}
u=(1-\chi_R)u+\chi_R u
\end{equation*}
and $R$ will be determinate later on. We start with the estimation
of $\|(1-\chi_R)u\|_{H_{s,\delta}}$ and for that purpose we use
the convention  (\ref{eq:a_priori_estimate_Bessel:2}) and compute
\begin{equation}
\label{eq:a_priori_estimates_weighted:10}
\begin{split}
A_\infty(D^2(1-\chi_R)u)  &= (1-\chi_R)A_\infty(D^2u)-
2A_\infty(D\chi_R)(Du)- A_\infty(D^2\chi_R)u \\
& =(1-\chi_R)(Lu) +E_1+E_2,
\end{split}
\end{equation}
where 
\begin{equation*}
 E_1=-(1-\chi_R)\left\{
\left(A-A_\infty\right)(D^2u)+B(x)(Du)+C(x)u\right\}
\end{equation*}
and 
\begin{equation*}
 E_2=- \left\{2A_\infty(D\chi_R)(Du)+
A_\infty(D^2\chi_R)u\right\}
\end{equation*}
Applying  inequality (\ref{eq:a_priori_estimates_weighted:8}) of Lemma
\ref{lem:a_priori_estimates_weighted:3},
\begin{equation}
\label{eq:a_priori_estimates_weighted:11}
\begin{split}
  \|(1-\chi_R)u\|_{H_{s,\delta}}
  & \leq  C\| A_\infty D^2((1-\chi_R)u)\|_{H_{s-2,\delta+2}}\\
  & \leq  C\left\{\| (1-\chi_R)Lu\|_{H_{s-2,\delta+2}}+
   \| E_1\|_{H_{s-2,\delta+2}}+\|
   E_2\|_{H_{s-2,\delta+2}}\right\}.
   \end{split}
\end{equation}
Since $\| (1-\chi_R)Lu\|_{H_{s-2,\delta+2}}\lesssim \|
Lu\|_{H_{s-2,\delta+2}}$ (see Proposition \ref{prop:Properties:1}),
it     
remains to estimate $\| E_1\|_{H_{s-2,\delta+2}}$ and $ \|
E_2\|_{H_{s-2,\delta+2}}$. We may choose $\delta_i'$ so that
$\delta_i>\delta_i'>\frac{1}{2} -i$, $i=0,1,2$ and then we put
$\gamma=\min_{i=0,1,2}(\delta_i-\delta_i')$. Under these conditions
the application of Corollary \ref{cor:const:1}, Propositions
\ref{prop:Properties:11} and \ref{prop:Properties:3} yield
\begin{equation}
\label{eq:a_priori_estimates_weighted:12}
\begin{split}
 \|E_1\|_{H_{s-2,\delta+2}}& \leq C
 \left\|(1-\chi_R)\left\{\left(A-A_\infty\right)(D^2u)+ B(Du) +
C u\right\}\right\|_{H_{s-2,\delta+2}}
\\&\leq C  \{\left\|(1-\chi_R) (A-A_\infty)
\right\|_{H_{s_2,\delta_2'}} \|D^2 u\|_{H_{s-2,\delta+2}} \\
& + \left\|(1-\chi_R) B \right\|_{H_{s_1,\delta_1'}} \|D
u\|_{H_{s-1,\delta+1}}+\left\|(1-\chi_R) C
\right\|_{H_{s_0,\delta'}} \| u\|_{H_{s,\delta}}\}
\\&\leq   \frac{C_1}{R^\gamma}  \left( \left\|(
A-A_\infty)\right\|_{H_{s_2,\delta_2}} + \left\|
B\right\|_{H_{s_1,\delta_1}}+\left\|C
\right\|_{H_{s_0,\delta_0}}\right)
\| u\|_{H_{s,\delta}} \\
&\leq \frac{C_1\Lambda}{R^\gamma} \| u\|_{H_{s,\delta}},
\end{split}
\end{equation}
where  \begin{math}
\Lambda=\left(\|A-A_\infty\|_{H_{s_2,\delta_2}}+\|B\|_{H_{s_1,\delta_2
}}+\|C\|_{H_{s_0,\delta_0}}\right)
\end{math}.

Next, since  $D\chi_R$ has compact support, Remark \ref{rem:constr:1}
 and Corollary \ref{cor:const:1} imply that
\begin{equation}
\label{eq:a_priori_estimates_weighted:13}
\begin{split}
 \|E_2\|_{H_{s-2,\delta+2}} &\leq C(R)\left\{
   \|2A_\infty((D\chi_R)(Du))\|_{H_{s-2,\delta'+1}} 
 + \|A_\infty((D^2\chi_R)u)\|_{H_{s-2,\delta'}}\right\}\\
 & \leq
 C(R)\|A_\infty\|
 \left\{2\|Du\|_{H_{s-2,\delta'+1}}+\|u\|_{H_{s-2,\delta'}}\right\}
\\& \leq  C(R)\|A_\infty\|
\|u\|_{H_{s-1,\delta'}}.
\end{split}
\end{equation}

We turn now to the estimation of $ \|\chi_R u\|_{H_{s,\delta}}$.
Noting that $(\chi_R u)$ has compact support, we have by
Remark \ref{rem:constr:1}, (\ref{eq:a_priori_estimates_weighted:6})
and
Proposition \ref{prop:Properties:1} that 
\begin{equation}
\label{eq:a_priori_estimates_weighted:14}
\begin{split}
\|\chi_R u\|_{H_{s,\delta}} & \leq C(R) \|\chi_R
u\|_{H_{s,\delta'}} \leq  C(R)\{ \|L(\chi_R
u)\|_{H_{s-2,\delta'+2}} + \| u\|_{H_{s-1,\delta'}}\}.
\end{split}
\end{equation}
Similarly to (\ref{eq:a_priori_estimates_weighted:10}) we compute
\begin{equation}
\label{eq:a_priori_estimates_weighted:15}
\begin{split}
& L(\chi_Ru) = \chi_R L(u)+ 2 A(D\chi_R)(Du)+ A(D^2\chi_R)u
+ B( D\chi_R) u.
\end{split}
\end{equation}

We estimate each term of (\ref{eq:a_priori_estimates_weighted:15})
separately. Once again, since $\chi_R L u$ has compact support,
\begin{equation}
\label{eq:elliptic part:46}
\begin{split}
 \|\chi_R(L u)\|_{H_{s-2,\delta'+2}} \leq C(R) \|L
u\|_{H_{s-2,\delta+2}}.
\end{split}
\end{equation}

Next, using the second assumption of $(H)$, Proposition
\ref{prop:Properties:3}  and the compactness of $\supp(\chi_R)$
 we get
\begin{equation}
\label{eq:a_priori_estimates_weighted:16}
\begin{split}
& \ \ \ \ \|2 A(D\chi_R)(Du)\|_{H_{s-2,\delta'+2}}\\ & \leq 2   \|
(A-A_\infty)(D\chi_R)(Du)\|_{H_{s-2,\delta'+2}}+ \|
A_\infty(D\chi_R)(Du)\|_{H_{s-2,\delta'+2}}\\
 &  \leq C
\left(\|(A-A_\infty)\|_{H_{s_2,\delta_2}}+\|A_\infty\|\right)
\|(D\chi_R)(Du)\|_{H_{s-2,\delta'+2}}
\\
 & \leq C(R)
\left(\|(A-A_\infty)\|_{H_{s_2,\delta_2}}+\|A_\infty\|\right)
\|Du\|_{H_{s-2,\delta'+1}}\\
 & \leq C(R)
\left(\|(A-A_\infty)\|_{H_{s_2,\delta_2}}+\|A_\infty\|\right)
\|u\|_{H_{s-1,\delta'}}.
\end{split}
\end{equation}
In a similar manner we estimate the other terms and together with
inequalities
(\ref{eq:a_priori_estimates_weighted:11})--(\ref{eq:a_priori_estimates_weighted:14}),%
(\ref{eq:elliptic part:46}) and
(\ref{eq:a_priori_estimates_weighted:16}) we have
\begin{equation}
\label{eq:a_priori_estimates_weighted:17}
\begin{split}
 \|  u\|_{H_{s,\delta}}& \leq \|(1-\chi_R)  u\|_{H_{s,\delta}}+
\|\chi_R  u\|_{H_{s,\delta}}
 \\& \leq C\left\{ \| L u\|_{H_{s-2,\delta+2}} + C_2\|
u\|_{H_{s-1,\delta'}}
+ \frac{C_1\Lambda}{R^\gamma} \| u\|_{H_{s,\delta}}\right\},
\end{split}
\end{equation}
where $C_1$ and $C_2$ depend on the norms of the coefficients of $L$
and in addition $C_2$ depends in $R$.  We now take $R$ such that
$\frac{C_1\Lambda}{R^\gamma}\leq\frac{1}{2}$, then
(\ref{eq:a_priori_estimates_weighted:9}) follows from
(\ref{eq:a_priori_estimates_weighted:17}).  \BeweisEnde
\end{proof}

The next two theorems are consequence of the compact embedding,
Theorem \ref{thr:Appendix:1}, the a priori estimate
(\ref{eq:a_priori_estimates_weighted:9}) and standard arguments of
Functional Analysis.
\begin{thm}[Semi Fredholm]
  \label{thm:eq:a_priori_estimates_weighted:1}
  Assume the operator $L$ satisfies hypotheses$(H)$, $s \geq 2 $ and
  $-\frac{3}{2}<\delta<-\frac{1}{2}$. Then $ {L}:H_{s,\delta}\to
  H_{s-2,\delta+2}$ is {\it semi Fredholm}, that is,
 \begin{enumerate}
\item[{\rm (i)}]  ${\bf dim}({\bf Ker}L)<\infty$; \item[{\rm
(ii)}] If $L$ is injective, then there is a constant $C$ such that
\begin{equation}
\label{eq:a_priori_estimates_weighted:18} \|u\|_{H_{s,\delta}}
\leq C\|Lu\|_{H_{s-2,\delta+2}};
\end{equation}
\item[{\rm (iii)}] $L$ has a closed range.
 \end{enumerate}
\end{thm}

\begin{thm}[A homotopy argument]
  \label{thm:eq:a_priori_estimates_weighted:2}
  Let  $s \geq 2$ and $-\frac{3}{2}<\delta< -\frac{1}{2}$. Assume $L$
be
  an elliptic operator of the form
  (\ref{eq:a_priori_estimates_weighted:1}) that fulfilled the
  hypotheses $(H)$ and $ {L_t}$ is a continuous family of operators
  which satisfy hypotheses $(H)$ for $t\in [0,1]$, $L_1=L$ and
 \begin{equation*}
 {L}_t : H_{s,\delta}\to H_{s-2,\delta+2} \text{ is injective}.
 \end{equation*}
 If
  \begin{equation*}
  L_0 : H_{s,\delta}\to H_{s-2,\delta+2} \text{ is an isomorphism,}
  \end{equation*}
 then  the same is true for $L$.
\end{thm}

The next Lemma shows that solutions to the homogeneous system have
lower growth at infinity. We follow Christodoulou and
O'Murchadha's proof  \cite{OMC}.

\begin{lem}[Lower growth of homogeneous solutions]
  \label{lem:a_priori_estimates_weighted:5}
Assume $ {L}$ satisfies hypotheses $(H)$, $u\in H_{s,\delta}$,
 $s \geq  2$ and $-\frac{3}{2} < \delta<-\frac{1}{2}$.
If $Lu=0$, then $u\in H_{s,\delta'}$ for any
 $-\frac{3}{2} < \delta'<-\frac{1}{2}$.
\end{lem}

\begin{proof}[of Lemma \ref{lem:a_priori_estimates_weighted:5}]
We recall that if $\delta'<\delta$, then it follows from Remark
\ref{rem:const:1} that  $H_{s,\delta}\hookrightarrow H_{s,\delta'}$.
Hence  it suffices to show the
statement for $\delta'>\delta$. The conditions on $\delta_i$ imply
that we may find $\delta'>\delta$ so that
$\delta_i+\delta+i>\delta'+2-\frac{3}{2}$. Then  applying the
Proposition \ref{prop:Properties:3} to
\begin{equation*}
 f:=A_{\infty} u - L u =
\left(A_\infty-A(x)\right)\left( D^2 u\right)-B(x)(Du) -C(x) u,
\end{equation*}
we obtain that $f$ belongs to $ H_{s-2,\delta'+2}$. Now $L
u=0$, so
\begin{math}
 A_\infty u= f
\end{math}
and  since $ A_\infty: H_{s,\delta'}\to H_{s-2,\delta'+2} $ is
isomorphism by Lemma \ref{lem:a_priori_estimates_weighted:3}, we
conclude that hence $u\in H_{s,\delta'}$. We now replace $\delta$
by $\delta'$ repeat the above arguments. 
\BeweisEnde
\end{proof}

\subsection{Semi Linear Elliptic Equations on Asymptotically Flat
Manifolds} \label{subsection:Semi_linear_elliptic}

A Riemannian 3-manifold $(M,h)$ is asymptotically flat (AF) if there
is a compact subset $K$ such that $M\setminus K$ is diffeomorphic to $
\mathbb{R}^3\setminus B_1(0)$ and the metric $h$ tends to the identity
$I$ at infinity.  A natural definition of the last statement in our
case is $h-I\in H_{s', \delta'}$. We will hence assume here that
$h-I\in H_{s',\delta'}$, $s'>\frac{3}{2}$ and
$\delta'>-\frac 3 2$.

We denote by $\Delta_h$ be the Laplace-Beltrami operator on
$(M,h)$. In the coordinates $(x^1,x^2,x^3)$ it takes the form
\begin{equation}
\label{eq:Semi_linear_elliptic:1}
 \Delta_h=\frac{1}{\sqrt{|h|}}\partial_j\left(\sqrt{|h|}h^{ij}
\partial_i\right),
\end{equation}
where $|h|={\rm det}(h_{ij})$ and $h^{ij}=(h_{ij})^{-1}$.
 Inserting the identity
$\partial_j |h|=|h|{\rm tr}(h^{ij}(\partial_j (h_{ij}))$ into
(\ref{eq:Semi_linear_elliptic:1}), we have
\begin{equation}
\label{eq:Semi_linear_elliptic:2}
 \Delta_h= h^{ij}\partial_{j}\partial_i+\partial_j(h^{ij})\partial_i
 +\frac{1}{2} {\rm tr}(h^{ij}(\partial_j (h_{ij}))h^{ij}\partial_i.
\end{equation}
Hence, by means of Proposition \ref{prop:Properties:3}, Theorem
\ref{thr:Appendix:5} and Remark \ref{rem:Properties:2}, the elliptic
operator
(\ref{eq:Semi_linear_elliptic:2}) satisfies hypothesis $(H)$ of
Section \ref{subsec:a_priori_estimates_weighted} provided that $s\leq
s'$.

Let us introduce some more notations. We denote by
$\mu_h=\sqrt{|h|}dx$ the Lebesgue measure on the manifold $(M,h)$,
\begin{math}
 (Du\cdot Dv)_h= h^{ij}\partial_i u\partial_jv,
\end{math}
and $\|D u\|_h^2=(Du\cdot Du)_h$. Integration by parts yields
\begin{equation}
\label{eq:Semi_linear_elliptic:3}
\begin{split}
& \int \left(\Delta_h u\right) v d\mu_h=\int
\partial_j \left(\sqrt{|h|}h^{ij}\partial_i
 u\right)v dx\\= - &\int h^{ij}\partial_i u \partial_j v\sqrt{|h|}
 dx=-\int (Du\cdot Dv)_hd\mu_h.
\end{split}
\end{equation}
Formula (\ref{eq:Semi_linear_elliptic:3}) holds whenever $v\in
H_0^1(\mathbb{R}^3)$, $u\in H_{s,\delta}$ and $s \geq 1$. Therefore it
enables us to define weak solutions on the manifold $(M,h)$.

\begin{defn}[Weak solutions]
\label{def:Semi_linear_elliptic:1} A function $u$ in $
H_{s,\delta}$  is a weak solution of
\begin{equation*}
 -\Delta_h u +c(x)u=f\in H_{s-2,\delta+2}
\end{equation*}
on $(M,h)$, if
\begin{equation}
\label{eq:Semi_linear_elliptic:4} \int \left((Du\cdot Dv)_h
+cuv\right)d\mu_h= \int fv d\mu_h,
\end{equation}
for all $v\in H_0^1(\setR^3)$.
\end{defn}

\begin{rem}
 In case $u,v\in H_{s,\delta}$, $s \geq 2$ and
$\delta  \geq -1$, then by algebra $h^{ij}\partial_i u,
\sqrt{|h|}\partial_j v\in H_{s-1,0}$ (see subsection
\ref{sec:algebra}). Applying the Cauchy Schwarz
inequality
\begin{equation*}
\begin{split}
&\int |(Du\cdot Dv)_h|d\mu_h = \int
|h^{ij}\partial_iu\partial_jv|\sqrt{|h|}dx\\\leq & \left(\int
(h^{ij}\partial_i
u)^2\right)^{\frac{1}{2}}\left(\int\sqrt{|h|}
\partial_jv)^2\right)^{\frac{1}{2}}\leq
\|h^{ij}\partial_iu\|_{H_{s-1,0}}\|\sqrt{|h|}\partial_jv\|_{H_{s-1,0}}
,
\end{split}
\end{equation*}
we see that $h^{ij}\partial_iu\partial_jv|\sqrt{|h|}\in L^1(\setR^3)$.
Similarly, the integrand of the left hand side of
(\ref{eq:Semi_linear_elliptic:3}) belongs to $L^1(\setR^3)$.  Hence,
approximating $u$ and $v$ by $C_0^\infty$ functions and using
Lebesgue's Dominated Convergence Theorem  we have
 \begin{equation}
 \label{eq:Semi_linear_elliptic:5} \int \left(\Delta_h u\right) v
d\mu_h= -\int
(Du\cdot Dv)_hd\mu_h,\quad u,v\in H_{s,\delta},\, \text{whenever}\,
s \geq 2, \, \text{and}\, \delta \geq -1.
\end{equation}
\end{rem}

In this section we will prove existence and uniqueness for the
semi-linear equation
\begin{equation}
\label{eq:Semi_linear_elliptic:10}
 -\Delta_h u= F(x,u):= \sum_{i=1}^N m_i(x)h_i(u),
\end{equation}
where
 $m_i\in H_{s_0,\delta_0}$, $m_i(x) \geq  0$, $s_0 \geq  0$,
$\delta_0>
\frac{1}{2}$ and for $u>-1$ the functions  $h_i$ are decreasing,
nonnegative and smooth. These conditions  ensure $F(\cdot,u)$ and
$\frac{\partial F}{\partial p}(\cdot, u)$ are in $H_{s-2,\delta+2}$
whenever $u\in H_{s,\delta}$ and $s \geq  2$.

\begin{thm}[Existence and uniqueness]
\label{thm:Semi_linear_elliptic:1}
Let $h-I\in H_{s',\delta'}$, $s'> \frac{3}{2}$,
$\delta'>-\frac{3}{2}$, $2\leq s\leq s'$ and $-\frac{3}{2}
<\delta<-\frac{1}{2}$. Then equation
(\ref{eq:Semi_linear_elliptic:10}) has a unique solution $u$ in
$H_{s,\delta}$.  Furthermore, $0\leq u\leq K$ for a nonnegative
constant $K$.
\end{thm}

In order to show Theorem \ref{thm:Semi_linear_elliptic:1} we need the
weak maximal principle:

\begin{prop}[Weak maximal principle]
  \label{prop:Semi_linear_elliptic:1}
  Assume   $c\in
H_{s'-2,\delta'+2}$ is nonnegative. If $u\in H_{s,\delta}$
satisfies
\begin{equation}
\label{eq:Semi_linear_elliptic:71}
 -\Delta_h u+ cu\leq 0,
\end{equation}
then $u\leq 0$.
\end{prop}

\begin{proof}[of Proposition \ref{prop:Semi_linear_elliptic:1}]
  For $\epsilon>0$ we put $w=\max(u-\epsilon,0)$.  It has compact
  support since $\lim_{x\to\infty}u(x)=0$. Further, $Dw=Du$ a.e. in
  $\{u(x)>\epsilon\}$ (see
  e.~g.~\cite{gilbarg83:_ellip_partial_differ_equat_secon_order
  } 
  or
  \cite{kinderlehrer00:_introd_variat_inequal_their_applic%
  }. 
  Thus, $w\in H_0^1(\setR^3)$ and $w \geq 0$, so by
(\ref{eq:Semi_linear_elliptic:4})
\begin{equation*}
 0 \geq  \int\left( (Du,Dw)_{h}+cuw\right)d\mu_h=
 \int_{\{u \geq \epsilon\}}\left( \|Du\|^2_{h}+cu^2\right)d\mu_h.
\end{equation*}
Therefore $u\equiv\epsilon$  in $\{u(x)\geq\epsilon\}$. Since
$\epsilon$ is
arbitrary, we have $u\leq0$.
\BeweisEnde
\end{proof}

\begin{proof}[of Existence]
The proof will be done in several steps. We define a map
$\Phi:\{H_{s,\delta}\times [0,1], u(x)>-1\}\to H_{s-2,\delta+2}$
by
\begin{equation*}
 \Phi(u,\tau)=-\Delta_h u- \tau F(x,u),
\end{equation*}
   let $u(\tau)$ denotes a solution of $\Phi(u,\tau)=0$ and
 put $J=\{0\leq s\leq 1: \Phi(u(s),s)=0\}$. We will show that $J$
is  both  open and  closed set. Since $0\in J$, $J=[0,1]$ which
yields the existence result.

\medskip\noindent
{\textbf{Step 1.}} {\it The set $J$ is open.}

Let
\begin{equation*}
 Lw:=\left(\frac{\partial \Phi}{\partial
u}(u,\tau)\right)(w)=-\Delta_h
 w-\tau \frac{\partial F}{\partial p}(\cdot, u)w
\end{equation*}
and
\begin{equation*}
 L_t w=-\Delta_{\{th+(1-t)I\}}
 w-t\tau \frac{\partial F}{\partial p}(\cdot, u)w.
\end{equation*}
If $ L_t w=0$, then by Lemma (\ref{lem:a_priori_estimates_weighted:5})
$w\in H_{s,-1}$. So we may use (\ref{eq:Semi_linear_elliptic:5}) and
get
\begin{equation*}
\begin{split} \int\left( L_t w\right) w d\mu_{\{th+(1-t)I\}} =\int
\left(\|Dw\|_{\{th+(1-t)I\}}^2
 -t\tau \frac{\partial F}{\partial p}(\cdot,
 u)w^2\right)d\mu_{\{th+(1-t)I\}}.
\end{split}
\end{equation*}
Since $ \frac{\partial F}{\partial p}\leq0$, the above yields that
$L_t w=0$
implies $w\equiv 0$ for each $t\in [0,1]$. In addition
$L_0=-\Delta_I=-\Delta$  is an isomorphism according to Lemma
\ref{lem:a_priori_estimates_weighted:3}. Therefore  Theorem
\ref{thm:eq:a_priori_estimates_weighted:2} implies that $L_1=L$ is an
isomorphism too. Thus $J$ is open by the Implicit Function
Theorem.

\medskip\noindent
{\textbf{Step 2.}} {\it $\|u(\tau)\|_{H_{s,\delta}}\leq C$ for a
constant $C$ independent of $\tau$.}

We first establish the bound in $H_{2,\delta}$-norm. The weak maximum
principle implies $u(\tau) \geq 0$ and since $F(x,p)$ is decreasing in
$p$,
\begin{equation*}
\|F(\cdot,u(\tau))\|_{H_{0,\delta+2}}\leq
\|F(\cdot,0)\|_{H_{0,\delta+2}}\leq\left( \sum_{i=1}^N
h_i(0)^2\|m_i\|_{H_{0,\delta+2}}^2\right)^{\frac{1}{2}} :=K.
\end{equation*}
We showed in Step 1 that $\Delta_h:H_{s,\delta}\to
H_{s-2,\delta+2}$ is injective, therefore from Theorem
\ref{thm:eq:a_priori_estimates_weighted:1} (ii),
\begin{equation}
\label{eq:Semi_linear_elliptic:6} \|u(\tau)\|_{H_{2,\delta}}\leq C
\|-\Delta_h u(\tau)\|_{H_{0,\delta+2}} \leq
 C\|F(\cdot,0)\|_{H_{0,\delta+2}}\leq C K.
\end{equation}

Now, Theorem \ref{thr:Appendix:5} implies
$\|h_i(u(\tau))\|_{H_{2,\delta}}\leq C \|(u(\tau))\|_{H_{2,\delta}}$
and by Proposition  \ref{prop:Properties:3},
$\|F(\cdot,u(\tau))\|_{H_{2,\delta}}\leq C
\|u(\tau)\|_{H_{2,\delta}}$. In order to improve
(\ref{eq:Semi_linear_elliptic:6}), we take $s^{\prime\prime}$ so that
$ s^{\prime\prime}-2\leq2$ and $s^{\prime\prime}\leq s$. Then we may
apply again (\ref{eq:a_priori_estimates_weighted:18}) and combine it
with
(\ref{eq:Semi_linear_elliptic:6}) and Remark \ref{rem:const:1}, we
have
\begin{equation}
\begin{split}
& \|u(\tau)\|_{H_{s^{\prime\prime},\delta}}\leq C \|-\Delta_h
u(\tau)\|_{H_{s^{\prime\prime}-2,\delta+2}} \leq
 C \|-\Delta_h
u(\tau)\|_{H_{2,\delta+2}}\\=  C&
\|F(\cdot,u(\tau))\|_{H_{2,\delta}}\leq C
\|u(\tau)\|_{H_{2,\delta}}\leq CK.
\end{split}
\end{equation}
We have proved the boundedness in case $s^{\prime\prime}=s$,
otherwise we can repeat the same procedure as above to improve
regularity until we would reach the desired regularity. It is
obvious that the bound on $\|u(\tau)\|_{H_{s,\delta}}$ does not
depend on $\tau$.

\medskip\noindent
{\textbf{Step 3.}} {\it Lipschitz
 continuity with respect to  $\tau$:}

  Differentiation of the equation
$\Phi(u(\tau),\tau)=0$ with respect to $\tau$ gives

\begin{equation*}
-\Delta_h u_\tau(\tau)-\tau\frac{\partial F}{\partial
  p}F(x,u(\tau))u_\tau(\tau)=F(x,u(\tau)). 
\end{equation*}

Now $ \frac{\partial F}{\partial p}F(x,p)\leq 0$, so in the same way
as we did in Step 1 we obtain that the operator $L=-\Delta_h
-\tau\frac{\partial F}{\partial p}F(x,u(\tau)):{H_{s,\delta}}\to
{H_{s-2,\delta+2}}$ is injective. Hence, by Theorem
\ref{thm:eq:a_priori_estimates_weighted:1} (ii),
\begin{equation}
\label{eq:Semi_linear_elliptic:7}
 \|u_\tau \|_{H_{s,\delta}}\leq C
\|L(u_\tau) \|_{H_{s-2,\delta+2}}=C \|F(x,u(\tau))
\|_{H_{s-2,\delta+2}}.
\end{equation}
Next, Step 2 implies
\begin{equation}
\label{eq:Semi_linear_elliptic:8} \|F(x,u(\tau))
\|_{H_{s-2,\delta+2}}\leq C \|u(\tau)
\|_{H_{s,\delta}}\left(\sum_{i=1}^N \|m_i
\|_{H_{s_0,\delta_0}}\right)\leq C.
\end{equation}
Thus, combining (\ref{eq:Semi_linear_elliptic:7}) with
(\ref{eq:Semi_linear_elliptic:8}) we get
\begin{equation}
\label{eq:Semi_linear_elliptic:9}
\|u(\tau_1)-u(\tau_2)\|_{H_{s,\delta}}\leq C|\tau_1-\tau_2|.
\end{equation}

\medskip\noindent
{\textbf{Step 4.}} {\it The set $J$ is closed:}

Take a sequence  $\{\tau_n\}\subset J$ such that $\tau_n\to
\tau_0$. By
 (\ref{eq:Semi_linear_elliptic:9}),
 $\{u(\tau_n)\}$ is Cauchy in  ${H_{s,\delta}}$  and
therefore it converges to $u_0\in {H_{s,\delta}}$.  Since the map
$\Phi$ is continuous, it follows that $\Phi(u_0,\tau_0)=0$, that
is $\tau_0\in J$. This completes the proof of the existence.
\BeweisEnde
\end{proof}

\begin{proof}[of Uniqueness]
  Assume $u_1$ and $u_2$ are solutions to
  (\ref{eq:Semi_linear_elliptic:10}). We conduct the proof by showing
  that $\Omega:=\{x: u_1(x)>u_2(x)\}$ is an empty set. Note that
  $\Omega$ is open since $ u_1$ and $u_2$ are continuous.  Put
  $w=u_1-u_2$, then $-\Delta_h w=F(x,u_1)-F(x,u_2)\leq0$ in $\Omega$.
  So $w\leq 0$ in $\Omega$ by Proposition
  (\ref{prop:Semi_linear_elliptic:1}). That obviously leads to a
  contradiction unless $\Omega$ is empty.  \BeweisEnde
\end{proof}



\textbf{Acknowledgement:} The problem we worked out goes back to Alan
Rendall and we would like to thank him for enlightening discussions.
The second author would like to thank Victor Ostrovsky for many
valuable conversations.

\section*{Appendix}

\section{The construction of the Spaces $H_{s,\delta}$:}
\label{sec:constr-spac-h_s}

The weighted Sobolev spaces of integer order  below were
introduced by Cantor~\cite{cantor75:_spaces_funct_condit_r}    and
independently by Nirenberg and Walker
\cite{nirenberg73:_null_spaces_ellip_differ_operat_r}.
 Nirenberg and Walker initiate the study of elliptic operators
in these spaces, while Cantor used them to solve the constraint
equations on asymptotically flat manifolds. For an  nonnegative
integer $m$ and a real $\delta$  we define a norm
\begin{equation}
\label{eq:const:1} \left(\|u\|_{m,\delta}^{*}\right)^{2}
=\sum_{|\alpha|\leq m}\int\left(\langle
  x\rangle^{\delta+|\alpha|} |\partial^\alpha u|\right)^2dx,
  \end{equation}
where $\left\langle x \right\rangle=1+|x| $. The space $H_{s,m}$
is the completion of $C_0^\infty(\setR^3)$ under the norm
(\ref{eq:const:1}). Note that the weight varies with the
derivatives.

Here we will repeat Triebel's extension of these spaces into a
fractional order, \cite{triebel76:_spaces_kudrj2},
\cite{triebel95
}.  Let $s=m+\lambda$, where $ m$ is a nonnegative
integer and $0<\lambda<1$. One possibility {of
extending} the ordinary integer order Sobolev spaces  is the {\it
Lipschitz-Sobolevskij Spaces}, {having} a norm
\begin{equation}
\label{eq:const:2}
 \|u\|_{m+\lambda,2}^2=\sum_{|\alpha|\leq m}\int |\partial^\alpha u|^2dx+
\sum_{|\alpha|= m}\int\int{ |\partial^\alpha u(x)-\partial^\alpha
u(y)|^2\over |x-y|^{3+\lambda 2}}dxdy.
\end{equation}

Hence, a reasonable definition of {\it weighted fractional Sobolev
norm} is a combination of the norm (\ref{eq:const:1}) with
(\ref{eq:const:2}): 
\begin{equation} 
\label{eq:const:3}
 \left(\|u\|_{s,\delta}^{*}\right)^{2} =\left \{
  \begin{array}{lc}
\displaystyle
 \sum_{|\alpha|\leq m}
 \int|\langle  x\rangle^{\delta+|\alpha|} \partial^\alpha u|^2 dx,&  s=m\\[0.4cm]
\left.
\begin{array}{l}
\displaystyle  \sum_{|\alpha|\leq m}\int|\langle
x\rangle^{\delta+|\alpha|} \partial^\alpha u|^2dx\\  [0.8cm]
\displaystyle + \sum_{|\alpha|= m}\int\int  { |\langle
    x\rangle^{m+\lambda+\delta }\partial^\alpha u(x)-\langle
    y\rangle^{m+\lambda+\delta}\partial^\alpha u(y)|^2\over |x-y|^{3+2
      \lambda }}dxdy\, \quad 
\end{array}
\right\}, &  s=m+\lambda.
\end{array}
\right.
\end{equation}
here $m$ is a nonnegative integer and $0<\lambda<1$. The space
$H_{s,\delta}$ is the completion of $C_0^\infty(\setR^3)$ under
the norm (\ref{eq:const:3}).

The norm (\ref{eq:const:3}) is essential for the understating of the
connections between the integer and the fractional order. But it has a
disadvantage, namely, the double integral makes it almost impossible
to establish {any property} (embedding, a priori estimate, etc.)
needed for PDEs.  We are therefore looking for an equivalent
definition of the norm (\ref{eq:const:3}).

Let $K_j=\{x:2^{j-3}\leq |x|\leq 2^{j+2}\}$, $(j=1,2,...)$ and
$K_0=\{x:  |x|\leq4\}$.  Let $\{\psi_j\}_{j=0}^\infty\subset
C_0^\infty(\setR^3)$ be a sequence  such that $ \psi_j(x)=1$ on
$K_j$, $\supp(\psi_j)\subset \{x:2^{j-4}\leq |x|\leq 2^{j+3}\}$,
for $j\geq 1$, $\supp(\psi_0)\subset\{x: |x|\leq 2^{3}\}$ and
\begin{equation}
\label{eq:const:4} |\partial^\alpha  \psi_j(x)|\leq  C_\alpha
2^{-|\alpha|j},
\end{equation}
where the constant $C_\alpha$ does not depend on $j$.

We define now, 
\begin{equation}
  \label{eq:const:19}
  \left(\|u\|_{s,\delta}^{\bigstar}\right)^{2} =\left \{
    \begin{array}{ll}
      \displaystyle
      \sum_{j=0}^\infty \left(2^{\delta 2j
        }\|\psi_ju\|_{L^2}^2 + 2^{(\delta+m)2 j}
        \sum_{|\alpha|=m}\|\partial^\alpha(\psi_ju)\|_{L^2}^2\right),&
      s=m\\
      \left.
        \begin{array}{l}
          \displaystyle  \sum_{j=0}^\infty\left( 2^{\delta
              2j}\|\psi_ju\|_{L^2}^2 + 2^{(\delta+m) 2j}
            \sum_{|\alpha|=m}\|\partial^\alpha(\psi_ju)\|_{L^2}^2\right) \\
+ \sum_{j=0}^\infty 2^{(\delta+m+\lambda)2
            j}\left(\sum_{|\alpha|= m}\int\int  {
              | \partial^\alpha(\psi_j u)(x)-\partial^\alpha(\psi_j u)(y)|
              ^2\over |x-y|^{3+2 \lambda }}dxdy\right),
        \end{array}
      \right\} &  s=m+\lambda.
    \end{array}
  \right.
\end{equation}

\begin{prop}[Equivalence of norms]
  \label{appendix;prop1}
  There are two positive constant{{s}} \ $c_0$
  and $c_1$ depending only on $s$, $\delta$ and the constants in
  (\ref{eq:const:4}) such that
  \begin{equation}
    \label{eq:const:5}
    c_0\|u\|_{s,\delta}^{\bigstar} \leq\|u\|_{s,\delta}^{\ast}\leq c_1
    \|u\|_{s,\delta}^{\bigstar}.
  \end{equation}
\end{prop}

This equivalence was proved in \cite{triebel76:_spaces_kudrj2} (see also \cite{BK2}).

We express these norms in terms of Fourier transform.  Let
\begin{equation*}
\hat{u}(\xi)=\mathcal{F}(u)(\xi)=\frac{1}{(2\pi)^3}\int
u(x)e^{-ix\cdot\xi} dx
\end{equation*}
denotes the Fourier transform, put
\begin{equation}
\Lambda^s u=\mathcal{F}^{-1}(1+|\xi|^2)^{\frac{s}{2}} \mathcal{F}u),
\end{equation}
and let $H^s$ denotes the  Bessel Potentials space having the norm
\begin{equation}
\label{eq:const:15}
 \|u\|_{H^{s}}^2=\|\Lambda^s u\|_{L^2}^2=\int
(1+|\xi|^2)^s|\hat{u}(\xi)|^2d\xi.
\end{equation}
 We also set
\begin{equation*}
\|u\|_{h^{s}}^2=\|\mathcal{F}^{-1}(|\xi|^s
\mathcal{F}u)\|_{L^2}^2=\int (|\xi|^s|\hat{u}(\xi)|)^2d\xi.
\end{equation*}

It is well known  that (see e.~g.~\cite{HOER
};  p. 240-241)
\begin{equation}
\label{eq:const:6}
 \|u\|_{h^{s}}^2\simeq\left\{ \begin{array}{cc}
  \sum_{|\alpha|=m}\int|\partial^\alpha u|^2dx  \quad  & s=m \\
  \sum_{|\alpha|=m}\int\int{{|\partial^\alpha u(x)-\partial^\alpha
  u(y)|^2} \over{|x-y|^{3+2\lambda}}}dx  \quad  & s=m+\lambda \\
\end{array}\right.
\end{equation}

 and since $(1+|\xi|^2)^s\simeq (1+|\xi|^s)$,
 \begin{equation}
 \label{eq:const:7}
\|u\|_{H^{s}}^2\simeq \left( \|u\|_{L^2}^2+\|u\|_{h^{s}}^2\right).
\end{equation}
Hence, by (\ref{eq:const:19}),
\begin{equation}
\label{eq:const:8}
 \left(\|u\|_{s,\delta}^{\bigstar}\right)^{2} \simeq
 \sum_{j=0}^\infty \left(2^{\delta 2 j
}\|\psi_ju\|_{L^2}^2 +2^{(\delta+s)2
j}\|{{\psi_j}}u\|_{h^{s}}^2\right)
\end{equation}

We invoke now  the scaling $u_\epsilon(x):=u(\epsilon x)$
($\epsilon>0$), then simple calculations yields
\begin{math}
 \|u_\epsilon\|_{L^2}^{2} =\epsilon^{-3} \|u\|_{L^2}^{2}
\end{math}  and 
\begin{math}
 \|u_\epsilon\|_{h^s}^{2} =\epsilon^{2s-3} \|u\|_{h^s}^{2}
\end{math}.
Combining the later one with (\ref{eq:const:7}), we have
\begin{equation}
\label{eq:const:9} \|u_\epsilon\|_{H^s}^{2} \simeq
\epsilon^{-3}\left(\|u\|_{L^2}+ \epsilon^{2s}
\|u\|_{h^s}^{2}\right).
\end{equation}
Setting $\epsilon =2^j$, multiplying (\ref{eq:const:9}) by
$2^{3j}$ and inserting it in (\ref{eq:const:8}), we conclude
\begin{equation}
\label{eq:const:10}
 \left(\|u\|_{s,\delta}^{\bigstar}\right)^{2} \simeq
 \sum_{j=0}^\infty 2^{(\frac{3}{2}+\delta )2 j
}\|(\psi_ju)_{2^j}\|_{H^s}^2.
\end{equation} 
The last one is the most convenience form of norm for
applications and therefore the right hand side of (\ref{eq:const:10})
defines the norm of $ H_{s,\delta}$ space (see Definition
\ref{def:weighted:3}).

Combining Proposition \ref{appendix;prop1} with (\ref{eq:const:8})
and (\ref{eq:const:10}) we get:
\begin{thm}[Equivalence of norms, Triebel]
\label{thm:a1}
There are two positive constant $c_0$ and $c_1$
depending only on $s$, $\delta$ and the constants in
(\ref{eq:const:4}) such that
\begin{equation}
 c_0\|u\|_{H{s,\delta}} \leq\|u\|_{s,\delta}^{\ast}\leq
 c_1 \|u\|_{H{s,\delta}}.
\end{equation}
\end{thm}

\begin{rem} 
  Theorem \ref{thm:a1} enables us to use both sorts of the norms
  (\ref{eq:const:3}) and (\ref{eq:weighted:4}), and for each
  application we will use the suitable type of norm.
\end{rem}

\begin{rem}
\label{rem:const:1}
Let $s'\leq s$ and $\delta'\leq\delta$, then the inclusion
\begin{math}
H_{s,\delta}\hookrightarrow H_{s',\delta'}
\end{math}
follows easily from the representations (\ref{eq:const:15}) and
(\ref{eq:weighted:4}) of the norms.
\end{rem}

\begin{rem}
  The functions $\{\psi_j\}$ are constructed by means of a
  composition of exponential functions. Hence, for any positive
  $\gamma$ there holds
  \begin{equation}
    c_1(\gamma,\alpha)|\partial^\alpha \psi_j^\gamma(x)|\leq
    |\partial^\alpha \psi_j(x)| 
    \leq c_2(\gamma,\alpha)|\partial^\alpha \psi_j^\gamma(x)|.
  \end{equation}

  Therefore the equivalence (\ref{eq:const:5})
  remains valid with $\psi_j^\gamma$ replacing $\psi_j$ and hence
  \begin{equation}
    \label{eq:const:12}
    \sum_j 2^{( \frac{3}{2} + \delta)2j} \| (\psi_j^\gamma u)_{(2^j)}
    \|_{H^{s}}^{2} \simeq
    \left(\|u\|_{s,\delta}^{{\bigstar}}\right)^2\simeq \sum_j 2^{(
      \frac{3}{2} + \delta)2j} \| (\psi_j u)_{(2^j)} \|_{H^{s}}^{2}.
  \end{equation}

\end{rem}

\begin{cor}[{Equivalence of norms}]
  \label{cor:app11}
  For any positive $\gamma$, there are two positive constants $c_0$ and
  $c_1$ depending on $s $, $\delta$ and $\gamma$ such that
  \begin{equation}
    \label{eq:const:13}
    c_0 {  \|u\|_{H_{s,\delta}}^2 \leq  \sum_j 2^{(
        \frac{3}{2} + \delta)2j} \| (\psi_j^\gamma u)_{(2^j)}
      \|_{H^{s}}^{2}  \leq c_1 {  \|u\|_{H_{s,\delta}}^2}}.
  \end{equation}

\end{cor}



\section{Some Properties of $H_{s,\delta}$}
\label{sec:some-properties-h_s}

\begin{thm}[Complex interpolation, Triebel]
  \label{const:thm:1}

Let $0<\theta<1$, $0\leq s_0<s_1$ and
  $s_\theta=\theta s_0 + (1-\theta)s_1$, then
  \begin{equation}
    \label{eq:const:16}
    [H_{s_0,\delta}, H_{s_1,\delta}]_\theta= H_{s_\theta,\delta},
  \end{equation}
  where (\ref{eq:const:16}) is a complex interpolation.

\end{thm}

As a consequence of the interpolation Theorem
\ref{const:thm:1} we get 
\begin{cor}[$H_{s,\delta}$-norm of a derivative]
  \label{cor:const:1}
  \begin{equation}
    \label{eq:const:17}
    \|\partial_i u\|_{H_{s-1,\delta+1}}\leq \| u\|_{H_{s,\delta}}
  \end{equation}
\end{cor}

\begin{proof}[of Corollary \ref{cor:const:1}]
  Let $m$ be a positive integer and define $T:H_{m,\delta}\to
  H_{m-1,\delta+1}$ by $T(u)=\partial_i u$. Using the norm
  (\ref{eq:const:1}) we see that $\|T(u)\|_{H_{m-1,\delta+1}}\leq
  \|u\|_{H_{m,\delta}}$. So (\ref{eq:const:17}) follows from Theorem
  \ref{const:thm:1}.
  \BeweisEnde
\end{proof}

\begin{rem}
  \label{rem:constr:1} 
  If $\supp{u}\subset \{|x|\leq R\}$, then for any $\delta$
  \begin{equation}
    \label{eq:const:14}
    c_1(R) \|u\|_{H^s}\leq \|u\|_{H_{s,\delta}}\leq
    c_2(R)\|u\|_{H^s}.
  \end{equation}
\end{rem}
This follows from the integral representation of the norm
(\ref{eq:const:1}) and the interpolation (\ref{eq:const:16}).

\begin{prop}[Multiplication by smooth functions]
  \label{prop:Properties:1}

Let $N\geq s$ be an integer. Assume $f\in C^N(\setR^3)$
    satisfies $\sup|D^k f|\leq K$ for $k=0,1,...N$, then
    \begin{equation}
      \label{eq:3} \|fu\|_{H_{s,\delta}}\leq C_sK \|u\|_{H_{s,\delta}}.
    \end{equation}
\end{prop}

\begin{proof}[of Proposition \ref{prop:Properties:1}]
  By the well known property of $H^s$
\begin{equation}
 \label{eq:Properties:2} 
\|fu\|_{H^{s}}\leq C_sK \|u\|_{H^{s}}
\end{equation} 
  we have
    \begin{equation}
\label{eq:wellposs-einstein-euler-hyper2}
\begin{split}
        \|f u\|_{H_{s,\delta}}^2 & =
        \sum_{j=0}^\infty 2^{\left( \frac{3}{2} +\delta \right)2j}
        \left\| \left( \psi_j fu \right)_{2^j} \right\|_{H^s}^2
       \\ & \leq  (C_sK)^2
        \sum_{j=0}^\infty 2^{\left( \frac{3}{2} +\delta \right)2j}
        \left\| \left( \psi_j u \right)_{2^j} \right\|_{H^s}^2\ =(C_sK)^2
        \| u\|_{H_{s,\delta}}^2.
      \end{split}
    \end{equation}
\BeweisEnde
\end{proof}

\begin{prop}[Multiplication by cutoff function]
\label{prop:Properties:11}
Let $ \chi_R\in C^\infty(\mathbb{R}^3)$ satisfies $\chi_R(x)=1$
for $|x|\leq R$, $\chi_R(x)=0$ for $|x|\geq 2R$ and
\begin{math}
\label{eq:2}
  |\partial^\alpha \chi_R|\leq c_\alpha R^{-|\alpha|}
\end{math}. Then for $\delta'<\delta$ there holds
    \begin{equation}
    \label{eq:properties:1}
      \|(1-\chi_R)u\|_{H_{s,\delta'}}\leq
      \frac{C(\delta,\delta')}{R^{\delta-\delta'}} \|u\|_{H_{s,\delta}}.
    \end{equation}
\end{prop}

\begin{proof}[of Proposition \ref{prop:Properties:11}]
    Let $J_0$ be the smallest  integer such that $R\leq 2^{J_0-3}$.
    Then $(1-\chi_R)\psi_j=0$ for $j=0,1,...,J_0-1$. Hence
    \begin{equation}
      \begin{split}
        & \|(1-\chi_R) u\|_{H_{s,\delta'}}^2 =
        \sum_{j=J_0}^\infty 2^{\left( \frac{3}{2} +\delta' \right)2j}
        \left\| \left( \psi_j(1-\chi_R) u \right)_{2^j} \right\|_{H^s}^2\\
        \leq C^2 &
        \sum_{j=J_0}^\infty 2^{\left( \frac{3}{2} +\delta' \right)2j}
        \left\| \left( \psi_j u \right)_{2^j} \right\|_{H^s}^2=
        C^2 \sum_{j=J_0}^\infty 2^{\left( \frac{3}{2} +\delta
          \right)2j}2^{(\delta'-\delta)2j}
        \left\| \left( \psi_j u \right)_{2^j} \right\|_{H^s}^2\\
        \leq C^2 &  2^{(\delta'-\delta)2J_0}\sum_{j=0}^\infty 2^{\left(
            \frac{3}{2} +\delta 
          \right)2j}
        \left\| \left( \psi_j u \right)_{2^j} \right\|_{H^s}^2  \leq
        \frac{C^2}{(8R)^{(\delta-\delta')2}}\|u\|_{H_{s,\delta}}^2.
      \end{split}
    \end{equation}
  \BeweisEnde
\end{proof}

\label{sec:two-interm-estim}

\begin{prop}[An intermediate estimate]
  \label{prop:Properties:4}
 Let $0\leq s_0<s<s_1$ and  $\varepsilon>0$, then
    there is a constant $C=C(\varepsilon)$ such that
    \begin{equation}
      \label{eq:11}
      \|u\|_{H_{s,\delta}}\leq \sqrt{2}\varepsilon \|u
      \|_{H_{s_1,\delta}}+ C\|u \|_{H_{s_0,\delta}},
    \end{equation}
    holds for all $u\in H_{s_1,\delta}$.
\end{prop}

\begin{proof}[of Proposition \ref{prop:Properties:4}]
  \label{sec:properties-4}
  Inequality (\ref{eq:11})  is well known
  in $H^s$ spaces. We apply it to each term of the norm
  (\ref{eq:weighted:4}) and get
  \begin{equation*}
    \begin{split}
      \| u\|_{H_{s,\delta}}^2& =
      \sum_{j=0}^\infty 2^{\left( \frac{3}{2} +\delta \right)2j}
      \left\| \left( \psi_j u \right)_{2^j} \right\|_{H^s}^2\\
      &\leq 2\epsilon^2
      \sum_{j=0}^\infty 2^{\left( \frac{3}{2} +\delta \right)2j}
      \left\| \left( \psi_j u \right)_{2^j} \right\|_{H^{s_1}}^2+
      2C^2(\epsilon) \sum_{j=0}^\infty 2^{\left( \frac{3}{2} +\delta
        \right)2j}
      \left\| \left( \psi_j u \right)_{2^j} \right\|_{H^{s_0}}^2\\& =
      2\epsilon^2 \| u\|_{H_{s_{1},\delta}}^2+2C^2(\epsilon)\|
      u\|_{H_{s_{0},\delta}}^2,
    \end{split}
  \end{equation*}
\BeweisEnde
\end{proof}

\begin{prop}[An intermediate estimate]
\label{prop:Properties:5}
     Let $0< s'<s $, then
    \begin{equation}
      \label{eq:12}
      \|u\|_{H_{s',\delta}}\leq
      \|u\|_{H_{s,\delta}}^{\frac{s'}{s}}
      \| u\|_{H_{0,\delta}}^{1-\frac{s'}{s}}.
    \end{equation}
  \end{prop}

\begin{proof}[of Proposition \ref{prop:Properties:5}]
Using  H\"older inequality we get $\|u\|_{H^{s'}}\leq \|u\|_{H^{s}}^{\frac{s'}{s}}
      \| u\|_{L^2}^{1-\frac{s'}{s}}$ and applying it and using again 
   H\"older inequality yields
  \begin{equation*}
    \label{eq:generalized-gag-nirenb:1}
    \begin{split}   \|
      u\|_{H_{s',\delta}}^2 & = \sum_j 2^{\left( \frac{3}{2}+ \delta
        \right)2j}
      \left\| \left( \psi_ju \right)_{2^j} \right\|_{H^{s'}}^{{2}} \\
      & \leq  \sum\limits_j^{}
      2^{\left( \frac{3}{2}+ \delta  \right)2j \left( \frac{s^{\prime}}{s}
        \right)}
      \left\| \left( \psi_ju \right)_{2^j} \right\|_{H^{s}}^{2
        \frac{s^{\prime}}{s}} \quad
      2^{\left( \frac{3}{2}+ \delta  \right)2j \left( \frac{s-s^{\prime}}{s}
        \right)}
      \left\| \left( \psi_ju \right)_{2^j} \right\|_{L_2}^{2
        \frac{s-s^{\prime}}{s}}
      \\
      &\leq \left( \sum\limits_j^{}
        2^{\left( \frac{3}{2}+ \delta  \right)2j}
        \left\| \left( \psi_ju \right)_{2^j} \right\|_{H^{s}}^{2 }
      \right)^{\frac{s^{\prime}}{s}} \quad \left(   \sum\limits_j^{}
        2^{\left( \frac{3}{2}+ \delta  \right)2j }
        \left\| \left( \psi_ju \right)
        \right\|^2_{L_2}\right)^{\frac{s-s^{\prime}}{s}} \\
      &= \left( \|u\|_{H_{s,\delta}} \right)^{\frac{2s^{\prime}}{s}}
      \quad \left( \|u\|_{H_{0,\delta}}
      \right)^{\frac{2\left(s^{\prime}-1    \right)}{s}}.
    \end{split}
  \end{equation*}
  \BeweisEnde
\end{proof}

\subsection{Algebra}
\label{sec:algebra}

\begin{prop}[Algebra in $H_{s,\delta}$]
  \label{prop:Properties:3}

 If $s_1,s_2\geq s$, $s_1+s_2>s+\frac{3}{2}$ and
    $\delta_1+\delta_2\geq\delta-\frac{3}{2}$, then
    \begin{equation}
    \label{eq:elliptic_part:12}
      \|uv\|_{H_{s,\delta}}\leq C \|u\|_{H_{s_1,\delta_1}}
      \|v\|_{H_{s_2,\delta_2}}.
    \end{equation}
\end{prop}

\begin{proof}[of Proposition \ref{prop:Properties:3}]
    \label{sec:properties-5}
By Corollary  \ref{cor:app11},
\begin{equation}
 \label{eq:const:20} 
\| uv\|_{H_{s,\delta}}^2 \simeq
  \sum_{j=0}^\infty 2^{\left( \frac{3}{2} +\delta \right)2j}
  \left\| \left( \psi_j^2 uv \right)_{2^j} \right\|_{H^s}^2.
\end{equation}  
We apply the classic algebra property
\begin{math}
      \|uv\|_{H^{s}}\leq C \|u\|_{H^{s_1}}
      \|v\|_{H^{s_2}}
\end{math}
    (see e.~g.~ \cite{Taylor91} Ch. 3, Section 5), to each term of the norm
    (\ref{eq:const:20}) and then we use Cauchy Schwarz inequality,
\begin{equation*}
\begin{split}
 & \| uv\|_{H_{s,\delta}}^2 \leq C
  \sum_{j=0}^\infty 2^{\left( \frac{3}{2} +\delta \right)2j}
  \left\| \left( \psi_j^2 uv \right)_{2^j} \right\|_{H^s}^2\\
  &\leq C^2
 \sum_{j=0}^\infty 2^{\left( \frac{3}{2} +\delta \right)2j}
  \left\| \left( \psi_j u \right)_{2^j} \right\|_{H^{s_1}}^2
  \left\| \left( \psi_j v \right)_{2^j} \right\|_{H^{s_2}}^2\\ &
  \leq C^2\sum_{j=0}^\infty \left(2^{\left( \frac{3}{2} +\delta_1 \right)2j}
  \left\| \left( \psi_j u \right)_{2^j}
  \right\|_{H^{s_1}}^2\right)\left(2^{\left( \frac{3}{2} +\delta_2 \right)2j}
  \left\| \left( \psi_j v \right)_{2^j}
  \right\|_{H^{s_2}}^2\right)\\ &\leq C^2
  \left(\sum_{j=0}^\infty \left(2^{\left( \frac{3}{2} +\delta_1 \right)2j}
  \left\| \left( \psi_j u \right)_{2^j}
  \right\|_{H^{s_1}}^2\right)^2\right)^{\frac{1}{2}}
  \left(\sum_{j=0}^\infty \left(2^{\left( \frac{3}{2} +\delta_2 \right)2j}
  \left\| \left( \psi_j v \right)_{2^j}
  \right\|_{H^{s_2}}^2\right)^2\right)^{\frac{1}{2}}\\ & \leq C^2
  \left(\sum_{j=0}^\infty \left(2^{\left( \frac{3}{2} +\delta_1 \right)2j}
  \left\| \left( \psi_j u \right)_{2^j}
  \right\|_{H^{s_1}}^2\right)\right)
  \left(\sum_{j=0}^\infty \left(2^{\left( \frac{3}{2} +\delta_2 \right)2j}
  \left\| \left( \psi_j v \right)_{2^j}
  \right\|_{H^{s_2}}^2\right)\right)\\ & \leq C^2
  \| u\|_{H_{s_1,\delta_1}}^2\| v\|_{H_{s_2,\delta_2}}^2.
 \end{split}
\end{equation*}
\BeweisEnde
\end{proof}

\subsection{Fractional power $|u|^\gamma$} 
\label{sec:fract-power-ugamma}
In \cite{kateb03:_besov} Kateb showed that
if $u\in H^s\cap L^\infty$, $1<\gamma$ and  $ 0<s<\gamma +\frac{1}{2}$, then
\begin{equation}
\label{eq:crm-broken-3:2} 
\||u|^\gamma\|_{H^{s}}\leq
C(\|u\|_{L^\infty})
    \|u\|_{H^{s}}.
\end{equation}

\begin{prop}[Fractional power in $H_{s,\delta}$]
  \label{prop:properties:5a} 
  Let $u\in H_{s,\delta}\cap L^\infty$, $1<\gamma$, $
  0<s<\gamma +\frac{1}{2}$ and $\delta\in \mathbb{R}$, then
  \begin{equation}
    \label{eq:15} \||u|^\gamma\|_{H_{s,\delta}}\leq
    C(\|u\|_{L^\infty})
    \|u\|_{H_{s,\delta}}.
  \end{equation}
\end{prop}

\begin{proof}[of Proposition \ref{prop:properties:5a}]
Inequality (\ref{eq:15}) is a direct consequence of the
Corollary \ref{cor:app11}  and (\ref{eq:crm-broken-3:2}). Because
\begin{equation}
    \label{eq:frac_power}
    \begin{split}
     \| |u|^\gamma\|_{H{s,\delta}}^2  &\simeq
    \sum_{j=0}^\infty 2^{( \frac{3}{2} + \delta)2j}
    \| (\psi_j^\gamma |u|^\gamma)_{(2^j)} \|_{H^{s}}^{2}\\ & \leq \left(
    C(\|u\|_{L^\infty})\right)^2
    \sum_{j=0}^\infty 2^{( \frac{3}{2} + \delta)2j}
    \| (\psi_j u)_{(2^j)}\|_{H^{s}}^{2}\leq \left(
    C(\|u\|_{L^\infty})\right)^2 \| u \|_{H_{s,\delta}}^2.
    \end{split}
\end{equation}
\BeweisEnde
\end{proof}

\subsection{Moser type estimates}
\label{sec:moser-type-estim}

Y. Meyer proved the below Moser type estimate \cite{meyer84:_regul
}.
See also
Taylor \cite{Taylor91
}. 
\begin{thm}[Third Moser inequality for {{Bessel potentials spaces}}]
  \label{prop:Appendix:2}

Let $F: \setR^m\to\setR^l$ be $C^{N+1}$
  function such that $F(0)=0$. Let $s>0$ and $u\in H^s\cap L^\infty$.
  Then
  \begin{equation}
    \|F(u)\|_{H^s}\leq K \|u\|_{H^s},
  \end{equation}
  where
  \begin{equation}
    \label{eq:13}
    K=K_N(F,\|u\|_{L^\infty})\leq C
    \|F\|_{C^{N+1}}\left(1+\|u\|_{L^\infty}^N\right),
  \end{equation}
  here $N$ is a positive integer such that $N\geq
  [s]+1$. 
\end{thm}

We  generalize this important inequality  to the $H_{s,\delta}$ spaces. 
\begin{thm}[Third Moser inequality in $H_{s,\delta}$] 
\label{thr:Appendix:5}

Let $F: \setR^m\to\setR^l$ be $C^{N+1}$ function such that $F(0)=0$.
Let $s>0$, $\delta\in\mathbb{R}$ and $u\in H_{s,\delta}\cap L^\infty$.
Then 
\begin{equation}
\label{eq:14}
 \|F(u)\|_{H_{s,\delta}}\leq K \|u\|_{H_{s,\delta}},
\end{equation}
The constant $K$ in (\ref{eq:14}) depends on one
in (\ref{eq:13}) and in addition on $\delta$.
\end{thm}
\begin{proof}[of Theorem \ref{thr:Appendix:5}] 
Let $\{\psi_j\}$ be the
sequence satisfying (\ref{eq:const:4}) and $\Psi_j(x)=
\frac{1}{\varphi(x)}\psi_j(x)$, where $\varphi(x)=\sum_{j=0}^\infty
\psi_j(x)$. From the properties of the sequence $\{\psi_j\}$, it
follows that $1\leq\varphi(x)\leq7$. So the sequence
$\{\Psi_j\}\subset C_0^\infty(\setR^3)$ and $\sum_{j=0}^\infty
\Psi_j(x)=1$. From (\ref{eq:const:9}) we conclude that
\begin{equation}
\label{eq:m:15}
 \|u_\epsilon\|_{H^s}^2\leq
 C\max\{\epsilon^{2s-3},\epsilon^{-3}\}\|u\|_{H^s}^2
\end{equation}
and with the combination of (\ref{eq:multiplication:1}) and
Meyer's  Theorem\ref{prop:Appendix:2} we have, 
\begin{equation}
 \begin{split}
  \|F(u)\|_{H_{s,\delta}}^2 & =
 \sum_{j=0}^\infty 2^{( \frac{3}{2} + \delta)2j} \| (\psi_j (F(u))_{(2^j)}
 \|_{H^{s}}^{2}
\\ &
=\sum_{j=0}^\infty 2^{( \frac{3}{2} + \delta)2j} \left\|
\left(\psi_j F\left(\sum_{k=0}^\infty
\Psi_k(x)u\right)\right)_{(2^j)}\right\|_{H^{s}}^{2}
\\ &
 =\sum_{j=0}^\infty 2^{( \frac{3}{2} + \delta)2j} \left\|
\left(\psi_jF\left(\sum_{k=j-4}^{j+3}
\Psi_k(x)u\right)\right)_{(2^j)}\right\|_{H^{s}}^{2}
\\ &
\leq CK^2
\sum_{j=0}^\infty 2^{( \frac{3}{2} + \delta)2j}
 \sum_{k=j-4}^{j+3} \|\left(\Psi_k u\right)_{(2^j)}
 \|_{H^{s}}^{2}
\\ &
\leq CK^2 \sum_{j=0}^\infty 2^{( \frac{3}{2} + \delta)2j}
 \sum_{k=j-4}^{j+3} \|\left(\left(\Psi_k u\right)_{2^{j-k}}\right)_{(2^k)}
 \|_{H^{s}}^{2}
\\ &
\leq CK^2 \sum_{j=0}^\infty 2^{( \frac{3}{2} + \delta)2j}
 \sum_{k=j-4}^{j+3} \max\{2^{(2s-3)(j-k)}, 2^{-3(j-k)}\}\|\left(\Psi_k
   u\right)_{(2^k)} 
 \|_{H^{s}}^{2}
\\ &
\leq C(s)K^2 \sum_{j=0}^\infty 2^{(
\frac{3}{2} + \delta)2j}
 \sum_{k=j-4}^{j+3}\|\left(\psi_k u\right)_{(2^k)}
 \|_{H^{s}}^{2}
\\ &
\leq C(s,\delta)K^2
 \sum_{j=0}^\infty
 \sum_{k=j-4}^{j+3}2^{( \frac{3}{2} + \delta)2k}\|\left(\psi_k u\right)_{(2^k)}
 \|_{H^{s}}^{2}
\\ &
\leq 7 C(s,\delta)K^2
 \sum_{k=0}^{\infty}2^{( \frac{3}{2} + \delta)2k}\|\left(\psi_k u\right)_{(2^k)}
 \|_{H^{s}}^{2}
\leq 7C(s,\delta)K^2\|u\|_{H_{s,\delta}}^2.
\end{split}
\end{equation} 
\BeweisEnde
\end{proof}

\begin{rem}
  \label{rem:Properties:2}
  If $F(0)\not=0$ and $F(0)\in H_{s,\delta}$,
  then we can apply  Theorem \ref{thr:Appendix:5} to
  $\widetilde{F}(u):=F(u)-F(0)$  and get
  \begin{equation}
    \label{eq:Moser:3}
    \|F(u)\|_{H_{s,\delta}}\leq  \|\widetilde{F}(u)\|_{H_{s,\delta}}
    + \|F(0)\|_{H_{s,\delta}}\leq
    K\|u\|_{H_{s,\delta}}+\|F(0)\|_{H_{s,\delta}}.
  \end{equation}
\end{rem}

We may apply Theorem \ref{thr:Appendix:5} to the estimate the
difference $F(u)-F(v)$.
\begin{cor}[A difference estimate in $H_{s,\delta}$]
  \label{cor:appendix:3}
  Suppose $F$ is a $C^{N+2}$ function and $u,v\in H_{s,\delta}\cap
  L^\infty$. Then
  \begin{equation}
    \label{eq:Moser:6}
    \|F(u)-F(v)\|_{H_{s,\delta}}\leq C(\|u\|_{L^\infty},\|v\|_{L^\infty})
    \left(\|u\|_{H_{s,\delta}}+\|v\|_{H_{s,\delta}}
    \right)\left\|u-v\right\|_{H_{s,\delta}}. 
  \end{equation}
\end{cor}

\begin{proof}[of Corollary \ref{cor:appendix:3}]
  Put $\tilde{F}(u)=F(u)-F(0)-DF'(0)u$, then it suffices to show
  inequality (\ref{eq:Moser:6}) for  $\tilde{F}$. Now,
  \begin{equation}
    \label{eq:Moser:4} \tilde{F}(u)-\tilde{F}(v)=\int_0^1 \left(D
      \tilde{F}\left(tu+(1-t)v\right)\right)(u-v)dt=G(u,v)(u-v),
  \end{equation}
  where $G(u,v)=\int_0^1 D\tilde{F}\left(tu+(1-t)v\right)dt$. Since
  $G(0,0)=\int_0^1D\tilde{F}(0)dt=0$, we can apply Theorem \ref{thr:Appendix:5}
  to $G(u,v)$ and get: 
  \begin{equation}
    \label{eq:Moser:7} 
    \|G(u,v)\|_{H_{s,\delta}}\leq
    C(\|u\|_{L^\infty},\|v\|_{L^\infty})\left(\|u\|_{H_{s,\delta}}
      +\|v\|_{H_{s,\delta}}\right).
  \end{equation} 
  Applying Proposition (\ref{prop:Properties:3}) to the right side of
  (\ref{eq:Moser:4}), we have
  \begin{equation}
    \label{eq:Moser:5} \left\|
      \tilde{F}(u)-\tilde{F}(v)\right\|_{H_{s,\delta}}\leq
    C\left\|G(u,v)\right\|_{H_{s,\delta}}
    \left\|(u-v)\right\|_{H_{s,\delta}}
  \end{equation}
  and its combination with (\ref{eq:Moser:7}) gives (\ref{eq:Moser:6}).
  \BeweisEnde
\end{proof}

\subsection{Compact embedding } 
\label{sec:compact-embedding--1}

\begin{thm}[Compact embedding]
  \label{thr:Appendix:1}
Let $0\leq s'<s$ and $\delta'<\delta$, then the embedding
\begin{equation}
\label{eq:elliptic part:8} H_{s,\delta}\hookrightarrow
H_{s',\delta'}.
\end{equation}
is compact.
\end{thm}

\begin{proof}[of Theorem \ref{thr:Appendix:1}]
  \label{sec:compact-embedding-}
  Let $\{u_n\}\subset H_{s,\delta}$ be a sequence with
  $\|u_n\|_{H_{s,\delta}}\leq 1$. Since $H_{s,\delta}$ is a Hilbert
  space there is a subsequence, denoted by $\{u_n\}$, which converges
  weakly to $u_0$. We will complete the proof by showing that $u_n\to
  u_0$ strongly in $H_{s',\delta'}$.

Let $\chi_R\in C_0^\infty$ such that $\chi_R(x)=1$ for $|x|\leq R$ and
$\supp(\chi_R)\subset B_{2R}$. For a given $\epsilon>0$, we take $R$
such that $2\frac{C(\delta,\delta')}{R^{\delta-\delta'}} <\epsilon$,
where $C(\delta,\delta')$ is the constant of inequality
(\ref{eq:properties:1}) of Proposition \ref{prop:Properties:11}.  For
a bounded domain $\Omega$, it is known that the embedding
$H^s(\Omega)\hookrightarrow H^{s'}(\Omega)$ is compact and from Remark
\ref{rem:constr:1} it follows that $\|\chi_R u_n\|_{H^s}\leq C$, where
$C$ does not depend on $n$ . Hence $\chi_R u_n$ converges strongly to
$\hat{u}_0$ in $H^{s'}$.  In addition, we have that $\chi_R u_n\to
\chi_R u_0$ weakly in $H^s$ and hence $\chi_R u_n\to \chi_R u_0$
weakly in $H^{s'}$. Thus the sequence $\{\chi_R u_n\}$ converges both
strongly to $\hat{u}_0$ and weakly to $ \chi_R u_0$ in $H^{s'}$, hence
$\hat{u}_0=\chi_R u_0$ (because {{$\lim_n \langle (\chi_R u_n-\chi_R
u_0),(\hat{u}_0-\chi_R u_0) \rangle_{s'}= \langle ( \hat{u}_0-\chi_R
u_0),( \hat{u}_0-\chi_R u_0)\rangle_{s'}=\|\hat{u}_0-\chi_R
u_0\|_{H^{s'}}^2=0$}}
$\lim_n \|{\chi_Ru}_n-\chi_R u_0\|_{H_{s',\delta'}}=0$, hence we may
take $n$ sufficiently large so that $\|{\chi_Ru}_n-\chi_R
u_0\|_{H_{s',\delta'}}<\epsilon$. Therefore
\begin{equation}
\begin{split}
  & \|u_n-u_0\|_{H_{s',\delta'}}=\|(\chi_R u_n-\chi_R
  u_0)+(1-\chi_R)(u_n-u_0)\|_{H_{s',\delta'}}\\& \leq \|(\chi_R
  u_n-\chi_R
  u_0)\|_{H_{s',\delta'}}+\|(1-\chi_R)(u_n-u_0)\|_{H_{s',\delta'}}\\
  & <\epsilon+
  \frac{C}{R^{\delta-\delta'}}\|(u_n-u_0)\|_{H_{s,\delta}}\leq
  \epsilon+ \frac{C}{R^{\delta-\delta'}}
  \left(\|u_n\|_{H_{s,\delta}}+\|u_0\|_{H_{s,\delta}}\right) \\ &\leq
  \epsilon+ 2\frac{C(\delta,\delta')}{R^{\delta-\delta'}}<2\epsilon
\end{split}
\end{equation}
and that completes the proof.
\BeweisEnde
\end{proof}

\subsection{Embedding into the continuous}
\label{subsec:Embedding_into_the_continuous}

We introduce the following  notations. For a   nonnegative
integer $m$ and $\beta\in \mathbb{R}$, we set
\begin{eqnarray*}
& \|u\|_{C_\beta^m}=\sum_{|\alpha|\leq
m}\sup_x\left((1+|x|)^{\beta+|\alpha|} |\partial^\alpha u(x)|\right)
\end{eqnarray*}

Let  $ C_{\beta}^m$ be the functions space
corresponding to the above norms.

\begin{thm}[Embedding into the continuous]
  \label{thr:Appendix:2}

  If $s>\frac{3}{2} + m$ and $\delta+\frac 3  2\geq\beta$, then any $u\in
  H_{s,\delta}$ has a representative $\tilde{u}\in C_{\beta}^m $ satisfying
\begin{equation}
\label{eq:em:2} \| \tilde{u}\|_{C_\beta^m}\leq C \|{u}
\|_{H_{s,\delta}}.
\end{equation}
\end{thm}

\begin{proof}[of Theorem \ref{thr:Appendix:2}]
  We first show (\ref{eq:em:2})  when $m=0$.
 In order to make notations simpler we will use the convention
 $2^k=0$ if $k<0$.
  Recall that $\psi_j(x)=1 $
on  $K_j:=\{2^{j-3}\leq |x|\leq 2^{j+2}\}$. Using the known
embedding
\begin{math}
 \sup_x|u(x)|\leq C \|u\|_{H^s}
\end{math} (see e.~g.~\cite{krantz92:_partial_differ_equat_compl_analy
}),
we have
\begin{equation}
\label{eq:6}
\begin{split}
 &   \sup_x(1+|x|)^\beta|u(x)|  \leq 2^\beta \sup_{j\geq
-1}\left(2^{\beta
j}\sup_{\{2^j\leq|x|\leq2^{j+1}\}}|u(x)|\right)
\\
 \leq & 2^\beta \sup_{j\geq -1}\left(2^{\beta
j}\sup|\psi_j(x)u(x)|\right)
 =  2^\beta \sup_{j\geq
-1}\left(2^{\beta j}\sup|\psi_j(2^j x)u(2^j x)|\right)
\\ 
 \leq &
2^\beta C \sup_{j\geq -1}\left(2^{\beta j}\|(\psi_j
u)_{2^j}\|_{H^s}\right)
 \leq 2^\beta C \sup_{j\geq
-1}\left(2^{(\frac{3}{2}+\delta) j}\|(\psi_j u)_{2^j}\|_{H^s}\right)
  \leq 2^\beta C \|u\|_{H_{s,\delta}}.
\end{split}
\end{equation}

If $m>1$, $s>\frac{3}{2}+m$  and
$\delta+\frac{3}{2} \geq \beta$, then $\partial^\alpha u\in
H_{s-|\alpha|,\delta+|\alpha|}$ for $1\leq |\alpha|\leq m$. So we may
apply (\ref{eq:6})  to $\partial^\alpha u$ and obtain
$ \|\partial^\alpha u\|_{C_{\beta+k}}\leq C \|\partial^\alpha
u\|_{H_{s-|\alpha|,\delta+|\alpha|}}$.  \BeweisEnde
\end{proof}

\subsection{Density}
\label{subsec:Density}

\begin{thm}[Density of $C_0^\infty$ functions]
  \label{thm:density} $\quad$
  \begin{enumerate}
    \item[{\rm (a)}] The class $C_0^\infty(\mathbb{R}^3)$ is dense in
    $H_{s,\delta}$.

    \item[{\rm (b)}] Given $u\in H_{s,\delta}$ and $s'>s\geq 0$. Then
    for $\rho>0$  there is $u_\rho\in C_0^\infty(\mathbb{R}^3)$ and  a
    positive constant $C(\rho)$ such that
    \begin{equation}
      \|u_\rho -u\|_{H_{s,\delta}}\leq \rho\quad \text{and}\quad
      \|u_\rho\|_{H_{s',\delta}}\leq C(\rho) \|u\|_{H_{s,\delta}}.
    \end{equation}
  \end{enumerate}
\end{thm}

Property (a) was proved by Triebel 
\cite{triebel95
}.
We prove both of them here since (b) relies on (a).

\begin{proof}[of Theorem \ref{thm:density}]
  Let $J_\epsilon$ be the standard mollifier, that is,
  $\supp(J_\epsilon)\subset B(0,\epsilon)$,
  $\hat{J_\epsilon}(\xi)=\hat{J_1}(\epsilon\xi)=\hat{J}(\epsilon\xi)$
  and $\hat{J}(0)=1$. It is well known that for any $v\in H^s$,
  $\|J_\epsilon\ast v-v\|_{H^s}\to 0$ and that $J_\epsilon\ast v$
  belongs to $C^\infty(\mathbb{R}^3)$. In addition, we claim that
  there is $C=C(\epsilon,s,s')$ such that
  \begin{equation}
    \label{eq:16} \|J_\epsilon\ast v\|_{H^{s'}}\leq C \|v\|_{H^{s}}.
  \end{equation}
  Indeed, since $J\in C_0^\infty(\mathbb{R}^3)$, $|\hat{J}(\xi)|\leq
  C_m(1+|\xi|)^{-m}$ for any integer $m$. Therefore, for a given $s'$
  and $\epsilon$, we chose $m$ and the constant $C(\epsilon,s,s')$ so
  that $(1+|\xi|^2)^{s'-s}|\hat{J}(\epsilon\xi)|^2\leq
  C^2(\epsilon,s,s')$.  Hence
  \begin{equation*}
    \begin{split}
      \|J_\epsilon\ast v\|_{H^{s'}}^2 & =\int
      (1+|\xi|^2)^{s'}|\hat{J}(\epsilon\xi)|^2 |\hat{v}(\xi)|^2
      d\xi=\int
      (1+|\xi|^2)^{s}|\hat{v}(\xi)|^2(1+|\xi|^2)^{s'-s}\hat{J}(\epsilon\xi)|^2d
      \xi \\ \leq & C^2(\epsilon,s,s') \int
      (1+|\xi|^2)^{s}|\hat{v}(\xi)|^2d \xi =C^2(\epsilon,s,s')\|
      v\|_{H^{s}}^2.
    \end{split}
  \end{equation*}
  \begin{enumerate}
    \item [{\rm (a)}] Given $u\in H_{s,\delta}$ and $\rho>0$ we may
    chose $N$ such that
    \begin{equation*}
      \sum_{j=N-2}^\infty 2^{( \frac{3}{2} + \delta)2j} \| (\psi_j
      (u)_{(2^j)}
      \|_{H^{s}}^{2}\leq \rho^2.
    \end{equation*}
    Set now $u_N=\sum_{j=0}^N \Psi_k u$, where $\Psi_k$ is defined as in
    the proof of Theorem \ref{thr:Appendix:5}.  We use 
and get
\begin{equation}
 \label{eq:Properties:2} 
\|fu\|_{H^{s}}\leq C_sK \|u\|_{H^{s}}
\end{equation} 
    \begin{equation*}
      \begin{split}
        & \| u-u_N\|_{H_{s,\delta}}^2 \leq \sum_{j=0}^\infty 2^{(
          \frac{3}{2} + \delta)2j} \left\| \left(\psi_j
            \left(\sum_{k=N+1}^\infty \Psi_k u\right)\right)_{(2^j)}
        \right\|_{H^{s}}^{2}  
        \\ = &
        \sum_{j=N-2}^\infty 2^{(
          \frac{3}{2} + \delta)2j} \left\|  \left(\sum_{k=j-3}^{j+4}
            \psi_j\Psi_k u\right)_{(2^j)}
        \right\|_{H^{s}}^{2}
        \leq C
        \sum_{j=N-2}^\infty 2^{( \frac{3}{2} +
          \delta)2j}\sum_{k=j-3}^{j+4} \left\|  \left( \psi_j
            u\right)_{(2^j)}
        \right\|_{H^{s}}^{2}
        \\ 
        \leq  7 C &
        \sum_{j=N-2}^\infty 2^{( \frac{3}{2} + \delta)2j} \left\|  \left(
            \psi_j u\right)_{(2^j)}
        \right\|_{H^{s}}^{2}=7 C \rho^2.
      \end{split}
    \end{equation*}
    Now $u_N$ has compact support, therefore $J_\epsilon\ast u_N\in
    C_0^\infty(\mathbb{R}^3)$ and
    \begin{equation*}
      \begin{split}
        \|J_\epsilon\ast u_N-u_N\|_{H_{s,\delta}}^2 & \leq
        \sum_{j=0}^{N+4} 2^{( \frac{3}{2} + \delta)2j} \left\|
          \left(\psi_j (J_\epsilon\ast u_N-u_N)\right)_{(2^j)}
        \right\|_{H^{s}}^{2}\to 0 \quad \text{as}\quad \epsilon\to0.
      \end{split}
    \end{equation*}

    \item [{\rm (b)}] Let $ u\in H_{s,\delta}$ and $\rho>0$, then by
    (a) we can chose $N$ sufficiently large and $\epsilon$ small so
    that $ \|J_\epsilon\ast u_N-u\|_{H_{s,\delta}}< \rho$ and by
    (\ref{eq:16})

    \begin{equation*}
      \begin{split}
        & \|J_\epsilon\ast u_N\|_{H_{s',\delta}}^2 \leq \sum_{j=0}^{N+4}
        2^{( \frac{3}{2} + \delta)2j} \left\| \left(\psi_j (J_\epsilon\ast
            u_N)\right)_{(2^j)}
        \right\|_{H^{s'}}^{2}\\ \leq C^2(\epsilon,s,s')&  \sum_{j=0}^{N+4}
        2^{( \frac{3}{2} + \delta)2j} \left\| \left(\psi_j
            u_N)\right)_{(2^j)}
        \right\|_{H^{s}}^{2}  \leq C^2C^2(\epsilon,s,s')
        \| u\|_{H_{s,\delta}}^2.
      \end{split}
    \end{equation*}
    Thus, $u_\rho=J_\epsilon\ast u_N$.
  \end{enumerate}
  \BeweisEnde
\end{proof}

\vspace{5mm}

\bibliographystyle{plain}
\bibliography{bibgraf}

\end{document}